\pdfoutput=1

\documentclass[a4paper,12pt]{article}

\usepackage{mathtools, bbm, bm, array, amsmath, amsthm, amsfonts, amssymb, mathrsfs}
\usepackage{tikz}
\usepackage[margin=1cm]{caption}
\usepackage{subcaption}
\usepackage{environ}
\usepackage{graphicx}
\usepackage[
  linkcolor=blue,
  colorlinks=true,
  citecolor=blue,
  urlcolor=black,
  filecolor=black%
]{hyperref}
\usepackage[
  a4paper,
  total={170mm,237mm},
  left=20mm,
  top=30mm%
]{geometry}
\usepackage{datetime}
\usepackage[T1]{fontenc}
\usepackage{lmodern}
\usepackage[normalem]{ulem}
\usepackage{extarrows}
\usepackage[inline]{enumitem}
\usepackage[numbers, sort, compress]{natbib}
\usepackage{setspace}
\usepackage{changepage} 

\usepackage{empheq}
\usepackage{xcolor}

\newtheorem{theorem}{Theorem}%
\newtheorem{lemma}[theorem]{Lemma}%
\newtheorem{definition}[theorem]{Definition}%

\newcommand{\ecv}{M_{\mu}}
\newcommand{\intk}[1]{I_{#1}}

\newcommand{\dumvar}{u}
\newcommand{\dumvartwo}{v}
\newcommand{\typevar}{\tau}
\newcommand{\lowesteps}{\varepsilon'}
\newcommand{\mediumeps}{\tilde \varepsilon}
\newcommand{\mediumepsgam}{\tilde \varepsilon_{\gamma}}

\DeclareUnicodeCharacter{025B}{\ensuremath{\varepsilon}}
\newcommand{\diff}{\mathop{}\!d}

\newcommand{\semax}{\{\mu_n\}_{n\in\Ns}}
\newcommand{\ellip}{\mathcal{E}_{\mu}}
\newcommand{\vol}{\text{vol}}
\newcommand{\gsm}{\mathcal{G}(\ellip, \sigma)}
\newcommand{\gsmk}{\mathcal{G}(\mathcal{K}, \sigma)}
\newcommand{\rcb}{\text{(RC)}_b}

\newcommand{\N}{\mathbb{N}}
\newcommand{\Ns}{\mathbb{N}^*}
\newcommand{\R}{\mathbb{R}}
\newcommand{\Rp}{\mathbb{R}_+^*}

\let\originalleft\left
\let\originalright\right
\renewcommand{\left}{\mathopen{}\mathclose\bgroup\originalleft}
\renewcommand{\right}{\aftergroup\egroup\originalright}
\newcommand{\quadtext}[1]{\quad \text{#1}\quad}

\newcommand{\newinf}{\mathop{\mathrm{inf}\vphantom{\mathrm{sup}}}}

\usetikzlibrary{shapes}
\usetikzlibrary{decorations.markings}
\usetikzlibrary{patterns}

\newdateformat{monthyeardate}{%
  \monthname[\THEMONTH] \THEYEAR} 

\colorlet{eqframe}{black!35}   
\title{Metric Entropy and Minimax Risk of Ellipsoids with an Application to Pinsker's Theorem}
\date{}
\author{Thomas Allard \\ tallard@ethz.ch}

\begin{document}

\maketitle

\abstract{\noindent
  We study how large an $\ell^2$ ellipsoid is by introducing type-$\tau$ integrals that capture the average decay of its semi-axes.
  These integrals turn out to be closely related to standard complexity measures:
  we show that the metric entropy of the ellipsoid is asymptotically equivalent to the type-$1$ integral, and that the minimax risk in non-parametric estimation is asymptotically determined by the type-$2$ and type-$3$ integrals.
  This allows us to retrieve and sharpen classical results about metric entropy and minimax risk of ellipsoids through a systematic analysis of the type-$\typevar$ integrals, and yields an explicit formula linking the two.
  As an application, we improve on the best-known characterization of the metric entropy of the Sobolev ellipsoid, and extend Pinsker's Sobolev theorem in two ways: (i) to any bounded open domain in arbitrary finite dimension, and (ii) by providing the second-order term in the asymptotic expansion of the minimax risk.
}



\section{Introduction}

Quantifying compactness in metric spaces is a central problem in mathematics.
Compactness is a qualitative property---a set is either compact or it is not---but in many situations it is desirable to introduce a quantitative measure of how compact it is.
Metric entropy is a standard choice:
it appears across fields where a quantitative analysis involving compact sets is required, including approximation theory \cite{lorentzApproximationFunctions1966,lorentzMetricEntropyApproximation1966,lorentzConstructiveApproximationAdvanced1996,cohenTreeApproximationOptimal2001, cohen2022optimal}, 
empirical processes \cite{dudleyUniversalDonskerClasses1987,vandervaartWeakConvergenceEmpirical1996,dudleyUniformCentralLimit1999,luschgyFunctionalQuantizationGaussian2002},
harmonic analysis \cite{donohoUnconditionalBasesAre1993,donohoUnconditionalBasesBitLevel1996,donohoDataCompressionHarmonic1998,Grohs2015},
high-dimensional probability and statistics \cite{birge1983approximation,Vershynin_2018,wainwrightHighDimensionalStatistics2019},
information theory \cite{prosserEEntropyECapacityCertain1966,haussler1997mutual,yang1999information},
linear operator theory \cite{carlInequalitiesEigenvaluesEntropy1980,carl1981entropy,konigEigenvalueDistributionCompact1986,carlEntropyCompactnessApproximation1990,psidopaper,ALLARD2025101762},
and, more recently, neural network theory \cite{schmidt2020ReLU,elbrachterDeepNeuralNetwork2021,hutterMetricEntropyLimits2022,pan2025metricentropylimitsapproximationnonlinear}.
Other quantitative measures of compactness also exist.
For instance, in Section~\ref{sec:minimax-risk} we argue that the minimax risk in non-parametric estimation over a compact domain can be interpreted as a measure of its compactness.
The aim of the paper is to study such measures for a particular type of compact sets: ellipsoids in separable Hilbert spaces.

Ellipsoids are prototypical compact sets in separable Hilbert spaces.
Examples include unit balls in analytic, Sobolev, and Besov spaces \cite{firstpaper,secondpaper,triebelTheoryFunctionSpaces1983,triebelTheoryFunctionSpaces1992,triebelTheoryFunctionSpaces2006,triebelTheoryFunctionSpaces2020},
as well as images of unit balls under compact linear operators \cite{ALLARD2025101762,psidopaper,prosserEEntropyECapacityCertain1966,brezisFunctionalAnalysisSobolev2011}.
Since every infinite-dimensional separable Hilbert space is isometrically isomorphic to $\ell^2(\Ns)$ 
(see, e.g., \cite[Remark~5.10]{brezisFunctionalAnalysisSobolev2011}), 
we lose no generality by working with ellipsoids in sequence space.
Formally, for a sequence of non-negative real numbers $\mu=\semax$,
the ellipsoid determined by $\mu$ is 
\begin{equation}\label{eq:def-ellipsoid}
  \ellip 
  \coloneqq \Bigg \{\{x_n\}_{n\in\Ns}\in\ell^2(\Ns)
  \, \Bigm| \, 
  x_n=0, \text{ if } \mu_n=0, \text{ and } \sum_{n \colon \mu_n>0} \left\lvert {x_n}/{\mu_n}\right\rvert^2 \leq 1 \Bigg \}.
\end{equation} 
The $\semax$ are referred to as the semi-axes of $\ellip$; 
the value $\mu_n$ quantifies its spread in the $n$-th dimension.
An ellipsoid with only finitely many non-zero semi-axes is called finite-dimensional, and, to avoid working with trivial cases, 
we assume throughout that at least one semi-axis is non-zero.

Ellipsoids have a rich structure, 
which we can exploit to quantify their compactness.
To this end, 
we introduce two functions.
First,
the  semi-axis-counting function $\ecv$ of the ellipsoid $\ellip$ 
is defined by
\begin{equation}\label{eq:definition-ecf}
  \ecv (\varepsilon)
  \coloneqq \left\lvert \left\{ n \in\Ns \mid \mu_n \geq \varepsilon \right\}\right\rvert,
  \quad\text{for all } \varepsilon>0,
\end{equation}
where $|S|$ denotes the cardinality of a set $S$.
In other words, $\ecv(\varepsilon)$ counts the number of semi-axes greater than or equal to $\varepsilon$.
Second, given a real number $\typevar\geq 1$,
we introduce the type-$\typevar$ integral $\intk{\typevar}$ associated with $\ellip$ according to
\begin{equation}\label{eq:main-definition-integrals-order-k}
  \intk{\typevar}(\varepsilon)
  \coloneqq
  \int_{\varepsilon}^\infty \frac{\ecv (\dumvar)}{\dumvar^\typevar} \diff \dumvar,
  \quad \text{for all } \varepsilon >0.
\end{equation}
This integral can be interpreted as a weighted average of the semi-axis-counting function above the threshold $\varepsilon$, 
and therefore captures how the semi-axes decay on average.

The compactness of the ellipsoid $\ellip$ in $\ell^2(\Ns)$ is equivalent to any of the following three conditions: 
(i) $\mu_n\to 0$ as $n\to\infty$,
(ii) $\ecv(\varepsilon)$ is finite for all $\varepsilon>0$,
and (iii) $\intk{\typevar}(\varepsilon)$ is finite for all $\typevar \geq 1$ and $\varepsilon >0$. 
The equivalence of (i) and compactness is well-known (see, e.g., \cite[Lemma~6.2 and Remark~6.7]{brezisFunctionalAnalysisSobolev2011}); 
that (ii) and (iii) are equivalent to (i) follows from the definitions \eqref{eq:definition-ecf}--\eqref{eq:main-definition-integrals-order-k}.
Thus, we have two ways to quantify the compactness of $\ellip$ via its semi-axes: 
either pointwise, 
via the rate of convergence of $\mu_n\to 0$ as $n\to\infty$ 
(equivalently, the behavior of $\ecv(\varepsilon)$ as $\varepsilon\to 0$), 
or on average, via the asymptotics of $\intk{\typevar}(\varepsilon)$ as $\varepsilon \to 0$, for fixed $\typevar\geq 1$.
Are these quantitative measures equivalent? 
If not, which should be preferred: the pointwise or the averaged?

Our claim is that compactness is more naturally measured by the asymptotics of the type-$\typevar$ integrals  $\intk{\typevar}$.
We justify this by relating the type-$\typevar$ integrals 
to classical measures of compactness, 
namely metric entropy and minimax risk.
Specifically, Section~\ref{sec:metric-entropy} shows that the metric entropy of an ellipsoid is asymptotically equivalent to its type-$1$ integral $\intk{1}$, and
Section~\ref{sec:minimax-risk} shows
that the asymptotics of the minimax risk are determined by the type-$2$ and type-$3$ integrals $\intk{2}$ and $\intk{3}$.
Section~\ref{sec:metric-entropy-vs-minimax-risk} 
clarifies the relation between the integrals $\intk{\typevar}$ for different values of $\typevar$ and,
in particular, yields a formula linking metric entropy and minimax risk.
Section~\ref{sec:appli-Sobolev} focuses on the Sobolev ellipsoid and serves two purposes: 
it illustrates the advantage of averaged over pointwise measures of compactness,
and it improves on the best-known characterization of the metric entropy for Sobolev spaces while extending Pinsker's Sobolev theorem in two ways---to any bounded open domain in arbitrary finite dimension, and to the second-order term in the asymptotic expansion of the minimax risk.
More generally, the theory developed in the paper refines classical characterizations of metric entropy and minimax risk for a wide variety of semi-axis decay rates;
a selection of these refinements is presented in Appendix~\ref{sec:table}.

\paragraph{Notation and terminology.}
Given a set $A$, $A^c$ denotes its complementary and
$\mathbbm{1}_A$  its indicator function, that is, the function that returns $1$ if the input belongs to $A$ and $0$ otherwise.
$ |A| $ denotes the (potentially infinite) cardinality of $A$.
$\N$ and $\Ns$ stand for the set of natural numbers including and, respectively, excluding zero.
$\R$ and $\R_+$ are the real and non-negative real numbers; $\R^*$ and $\R_+^*$ their counterparts excluding zero.
Given $d\in\mathbb{N}^*$, $\mathbb{R}^d$ is the $d$-dimensional Euclidean space. 
$\R^{\Ns}$ denotes the space of real-valued sequences equipped with the product $\sigma$-algebra, and $\ell^2(\Ns)$ the Hilbert space of square-integrable sequences equipped with the usual $\|\cdot\|_2$-norm.


We write $C^1(\Rp,\Rp)$ for the space of functions $f$ from $\Rp$ to $\Rp$ that are differentiable with continuous derivative (denoted $f'$).
We write $\ln(\cdot)$ for the natural logarithm.
For $k\in \mathbb{N}^*$, 
the $k$-fold iterated logarithm is
\begin{equation*}
\ln^{(k)}(\cdot) 
\coloneqq \underbrace{\ln \circ \cdots \circ \ln}_{k \text{ times}} \, (\cdot).
\end{equation*}
We write $|\cdot |$ for the absolute value.
The positive and negative parts of a real number are denoted by
\begin{equation*}
	(x)_+ \coloneqq \frac{|x|+x}{2}
	\quad \text{and} \quad 
	(x)_- \coloneqq \frac{|x|-x}{2},
	\quad \text{for all } x \in \mathbb{R}.
\end{equation*}
We emphasize that both quantities are non-negative.
We also set $\ln_+(x)= \max\{\ln(x), 0\}$ for $x\geq 0$ (note that we allow $x$ to be zero, in which case $\ln_+(x)=0$).

When comparing the asymptotic behavior of two functions $f$ and $g$, as $x \to 0$, we use the notation 
\begin{equation*}
f(x) = o_{x \to 0}(g(x)) \, \text{ if }\, \lim_{x \to 0} \frac{f(x)}{g(x)} =0,
\, \text{ and } \, 
f(x) = O_{x \to 0}(g(x)) \, \text{ if } \, \lim_{x \to 0} \,  \left\lvert \frac{f(x)}{g(x)}\right\rvert \leq C,
\end{equation*}
for some constant $C>0$.
We further indicate asymptotic equivalence according to
\begin{equation*}
f(x) \sim g(x)\  \text{ as } x\to 0
\iff \lim_{x \to 0} \frac{f(x)}{g(x)} =1
\iff f(x)=g(x)+o_{x \to 0}(g(x)),
\end{equation*}
and write $f(x) \stackrel{K}{\sim} g(x)$, as $x\to 0$, when there exists an (unspecified) constant $K>0$ such that $f(x) \sim K g(x)$ as $x\to 0$.
The functions $g_1, \dots, g_N$ provide an order-$N$ asymptotic expansion of $f$, with $N\in\Ns$, if 
\begin{equation*}
g_{i+1}(x) =o_{x \to 0}(g_i(x)), \text{ for all } i\in \{1, \dots, N-1\}, \text{ and } 
f(x)=\sum_{i=1}^N g_i(x) +o_{x \to 0}(g_N(x)).
\end{equation*}
In particular, $f(x) \sim g_1(x)$ as $x\to 0$, and we say that $g_1$ is the leading term of $f$.


\section{Metric Entropy and Type-1 Integral}\label{sec:metric-entropy}

Metric entropy is the most standard quantitative measure of compactness.
Its definition, provided in Definition~\ref{def:metric-entropy},
is motivated as follows.
In a metric space, a set is compact if and only if it is complete and totally bounded---that is, it can be covered by finitely many balls of same arbitrary radius.
Metric entropy goes one step further by measuring the minimal cardinality of such coverings: 
the smaller this cardinality, the more compact the set. 
In this sense, metric entropy indeed provides a quantitative measure of compactness.

\begin{definition}\label{def:metric-entropy}
  Let $\varepsilon>0$ and let $\mathcal{K}$ be a compact subset of $\ell^2(\Ns)$.
  An $\varepsilon$-covering of $\mathcal{K}$ is a set $\{x_{(1)}, \dots, x_{(N)}\}\subset \ell^2(\Ns)$ such that for each $x\in \mathcal{K}$, there exists an $i\in\{1, \dots, N\}$ satisfying $\| x- x_{(i)}\|_2\leq \varepsilon$.
  The $\varepsilon$-covering number $N(\varepsilon;\mathcal{K})$ is the smallest cardinality of an $\varepsilon$-covering of $\mathcal{K}$.
  The metric entropy of $\mathcal{K}$ is obtained by taking the natural logarithm of the $\varepsilon$-covering number, that is $H(\varepsilon;\mathcal{K}) \coloneqq  \ln \left(N(\varepsilon;\mathcal{K})\right)$.
\end{definition}

\noindent
Some authors prefer defining metric entropy as the logarithm in base two of the covering number. 
This is common in information theory when interpreting $H(\varepsilon;\mathcal{K})$ as the bit length required to uniformly encode elements of $\mathcal{K}$ to accuracy $\varepsilon$ \cite{elbrachterDeepNeuralNetwork2021,donohoUnconditionalBasesBitLevel1996,donohoCountingBitsKolmogorov2000,shiryayevSelectedWorksKolmogorov1993}.
In that case, the corresponding results differ from our results by a multiplicative factor $\ln(2)$.

We consider the case when the compact set $\mathcal{K}$ is an ellipsoid.
Before stating the main result, 
we need to introduce an additional assumption on the semi-axis-counting function.

\vspace{.25cm}

\begin{adjustwidth}{-.075cm}{-.075cm}
\begin{quote}
  \textbf{Regularity Condition (RC) on $\ecv$.}
  There exists $f\in C^1(\Rp,\Rp)$ non-increasing such that $ \ecv (x) \sim f(x)$, as $x\to 0$, and whose elasticity 
  \begin{equation}\label{eq:definition-elasticity}
    \rho (t) \coloneqq h'(t),
    \quadtext{where } h(t) \coloneqq \ln\left(f\left(e^{-t}\right)\right),
    \quad \text{for all } t>0,
  \end{equation}
  satisfies either 
  (i) $\rho(t)$ has a finite limit as $t\to\infty$, or
  (ii) $\rho(t) \to \infty$ as $t\to\infty$, with $\ln(\rho(t)) = O_{t\to\infty}(\rho(t/2))$, and $\rho$  non-decreasing on $(t^*, \infty)$ for some $t^*>0$.
\end{quote}
\end{adjustwidth}

\vspace{.25cm}

\noindent
We refer to this regularity condition as (RC), or  $\rcb$ 
when we need to specify the limit $\lim_{t\to\infty}\rho(t) \eqcolon b\in[0, \infty]$. 
It is set to avoid pathological cases, notably to prevent an accumulation of jumps of $\ecv$. 
It is not clear whether (RC) is necessary for our result or whether it can be relaxed.
In any case, it is not restrictive; 
see Appendix~\ref{sec:table} for a list of examples satisfying it.
We note that the finite-$b$ case (item (i) in (RC)) is equivalent to $\ecv$ being regularly varying at zero (in the sense of Karamata) with index $b$; see  Appendix~\ref{sec:reg-var-zero}.

Theorem~\ref{thm:metric-entropy-integral-form} 
establishes that, under (RC), the metric entropy of ellipsoids is asymptotically equivalent to the type-1 integral $\intk{1}$, and provides a precise characterization of the error term.

\begin{theorem}\label{thm:metric-entropy-integral-form}
  Let $\mu=\semax$ be a sequence such that $\mu_n\to 0$ as $n\to\infty$,
  with corresponding semi-axis-counting function $\ecv$ satisfying (RC).
  Then, the metric entropy of $\ellip$ and its type-$1$ integral $\intk{1}$ satisfy
  \begin{equation}\label{eq:main-result-ME}
    H\left(\varepsilon; \ellip\right)
    = \intk{1}(\varepsilon) 
    + O_{\varepsilon \to 0}\left(\min\left\{\ecv (\varepsilon), \sqrt{\ecv (\varepsilon)\ln\left(\ecv (\varepsilon)\right)\ln\left(\varepsilon^{-1}\right)}\right\}\right).
  \end{equation}
  In particular, we have 
  \begin{equation}\label{eq:asymp-equivalence-ME-I1-first-order}
   H\left(\varepsilon; \ellip\right)
    \sim \intk{1}(\varepsilon),
    \quad \text{as } \varepsilon\to 0.  
  \end{equation}
\end{theorem}

\begin{proof}[Proof sketch.]
  The relation \eqref{eq:main-result-ME} is established by proving matching upper and lower bounds.
  The lower bound comes from a standard volume argument.
  The upper bound relies on a block decomposition technique as introduced in \cite{secondpaper}.
  Finally, \eqref{eq:asymp-equivalence-ME-I1-first-order} is obtained from \eqref{eq:main-result-ME} by showing that the $O_{\varepsilon \to 0}(\cdot)$-term is always asymptotically negligible compared to $\intk{1}$. 
  The detailed proof is provided in Appendix~\ref{sec:proof-metric-entropy-integral-form}.
\end{proof}

\noindent
In many situations, one has access to the asymptotic expansion of $\ecv$.
The asymptotics are preserved through integration and therefore yield corresponding asymptotics for $\intk{1}$; 
this is formalized in Lemma~\ref{lem:abelian-integration} and the subsequent discussion in Appendix~\ref{sec:asymp-equiv-integration}.
Although the metric entropy of ellipsoids is well studied (see, e.g., \cite{firstpaper,secondpaper,ALLARD2025101762,prosserEEntropyECapacityCertain1966,edmunds1979embeddings,carlInequalitiesEigenvaluesEntropy1980,carl1981entropy,carlEntropyCompactnessApproximation1990,donohoCountingBitsKolmogorov2000,luschgyFunctionalQuantizationGaussian2002,grafSharpAsymptoticsMetric2004,luschgySharpAsymptoticsKolmogorov2004}),
prior works typically use different techniques to treat different decay behaviors of semi-axes---e.g., polynomial, exponential, or logarithmic---and often involve the tedious explicit construction of coverings,
with two exceptions discussed at the end of the section. 
On the other hand, 
Theorem~\ref{thm:metric-entropy-integral-form} offers one unified ready-to-use formula that yields sharp rates for all types of decay behavior: 
application of Theorem~\ref{thm:metric-entropy-integral-form} either retrieves the best-known results, or improves upon them.
We next illustrate this point on standard examples, 
and refer to Appendix~\ref{sec:table} for a list of results that can be obtained from \eqref{eq:asymp-equivalence-ME-I1-first-order}.

First, for finite-dimensional ellipsoids, i.e., when $\ecv(\varepsilon) \to d$ as $\varepsilon \to 0$ for some $d \in \Ns$,
we recover the standard result (see, e.g., \cite[Section~3]{shiryayevSelectedWorksKolmogorov1993})
\begin{equation}\label{eq:finite-dim-ME-formula}
    H\left(\varepsilon; \ellip\right)
    \sim \int_{\varepsilon}^\infty \frac{\ecv (\dumvar)}{\dumvar} \diff \dumvar
    \sim d \ln\left(\varepsilon^{-1}\right),
    \quad \text{as } \varepsilon\to 0. 
\end{equation}
Another fundamental case consists of polynomially-decaying semi-axes, say $\mu_n \sim (c/n)^{1/\alpha}$ as $n\to \infty$ for $c, \alpha>0$, or, equivalently, $\ecv(\varepsilon) \sim c \, \varepsilon^{-\alpha}$ as $\varepsilon\to 0$. 
Such ellipsoids typically arise as unit balls in smoothness classes (e.g., Sobolev spaces; see Section~\ref{sec:appli-Sobolev}) with $\alpha= d/k$, where $k$ is the smoothness index and $d$ the dimension of the domain.
Their metric entropy has been characterized in \cite{donohoCountingBitsKolmogorov2000,luschgySharpAsymptoticsKolmogorov2004,ALLARD2025101762,secondpaper,grafSharpAsymptoticsMetric2004};
this is one of the rare instances of infinite-dimensional set for which the leading term of the metric entropy was previously known precisely.
Once again, we recover standard results using \eqref{eq:asymp-equivalence-ME-I1-first-order} according to
\begin{equation}\label{eq:poly-decay-illustration-order-1}
  H\left(\varepsilon; \ellip\right)
  \sim \int_{\varepsilon}^\infty \frac{\ecv (\dumvar)}{\dumvar} \diff \dumvar
  \sim \int_{\varepsilon}^{\infty} \frac{c \, \dumvar^{-\alpha}}{\dumvar} \diff \dumvar
  = \frac{c\, \varepsilon^{-\alpha}}{\alpha},
  \quad \text{as } \varepsilon\to 0.
\end{equation}
Let us emphasize that there is much more to Theorem~\ref{thm:metric-entropy-integral-form} than just the asymptotic equivalence \eqref{eq:asymp-equivalence-ME-I1-first-order}:  being able to quantify the error term via \eqref{eq:main-result-ME} is also a central contribution.
For instance, given $N\in\Ns$, $c_1, \dots, c_N>0$, and $\alpha_1>\dots>\alpha_N>0$, we obtain an $N$-term generalization of \eqref{eq:poly-decay-illustration-order-1} since
\begin{equation*}
  \ecv(\varepsilon) 
  = \sum_{i=1}^N c_i \varepsilon^{-\alpha_i}+o_{\varepsilon\to0}\left(\varepsilon^{-\alpha_N}\right),
  \quadtext{implies}
  H\left(\varepsilon; \ellip\right)
  = \sum_{i=1}^N \frac{c_i\varepsilon^{-\alpha_i}}{\alpha_i}  +o_{\varepsilon\to0}\left(\varepsilon^{-\alpha_N}\right),
\end{equation*}
provided $\alpha_1 < 2\alpha_N$ (which ensures that the error term in \eqref{eq:main-result-ME} can be absorbed in the $o_{\varepsilon\to0}(\varepsilon^{-\alpha_N})$-term). 
The case $N=2$ appears already in \cite[Corollary~4]{ALLARD2025101762} and will play a central role in Section~\ref{sec:appli-Sobolev}.

By \eqref{eq:sum-formua-I1} in the proof of Theorem~\ref{thm:metric-entropy-integral-form}, 
the type-1 integral can equivalently be expressed as 
\begin{equation}\label{eq:water-filling-I1}
  \intk{1} (\varepsilon)
  = \sum_{n \in \Ns} \ln_+ \left(\frac{\mu_n}{\varepsilon}\right),
  \quad \text{for all } \varepsilon >0.
\end{equation}
With $\intk{1}$ of the form \eqref{eq:water-filling-I1}, the relation $H\left(\varepsilon; \ellip\right) \sim \intk{1}(\varepsilon)$ has a standard `water-filling' interpretation familiar to information- and communication-theorists;  see, e.g.,  \cite[Chapters~9 and 10]{cover1999elements} for a general background on water-filling arguments and \cite{donohoCountingBitsKolmogorov2000} for a discussion in the context of metric entropy of ellipsoids.
Specifically, the $n$-th dimension of the ellipsoid contributes to the metric entropy only if $\mu_n\geq \varepsilon$, and each active dimension contributes logarithmically in $\varepsilon^{-1}$, which is consistent with \eqref{eq:finite-dim-ME-formula}.
In that sense, the effective dimension of the problem is the number of semi-axes greater than $\varepsilon$, namely $\ecv(\varepsilon)$.

We conclude this section by comparing Theorem~\ref{thm:metric-entropy-integral-form} with two results that do not assume a specific decay of the semi-axes.
First, in his study of nuclear spaces,  
Mityagin established  in \cite[Theorem~3]{mityagin1961} a non-asymptotic characterization of the metric entropy in terms of $\intk{1}$,
which, with our notation, reads
\begin{equation}\label{eq:mitjagin}
  \intk{1}(2\varepsilon)
    \leq H\left(\varepsilon; \ellip\right)
    \leq \intk{1}\left(\frac{\varepsilon}{2}\right) + 2\ln(2)\ecv\left(\frac{\varepsilon}{2}\right),
  \quad \text{for all } \varepsilon > 0.
\end{equation}
Theorem~\ref{thm:metric-entropy-integral-form} and \eqref{eq:mitjagin} are complementary:  
the former is sharper but only asymptotic, 
while the latter holds non-asymptotically but fails to capture even the leading term precisely.
Moreover, using \eqref{eq:sum-formua-I1}--\eqref{eq:lower-bound-integral-ME} in the proof of Theorem~\ref{thm:metric-entropy-integral-form}, one can improve on the lower bound in \eqref{eq:mitjagin} according to $\intk{1}(\varepsilon)\leq H(\varepsilon; \ellip)$, for all $\varepsilon>0$.

Second, for finite-dimensional ellipsoids, 
the asymptotic equivalence $H\left(\varepsilon; \ellip\right) \sim \intk{1}(\varepsilon)$ 
was already obtained by Dumer, Pinsker, and Prelov in \cite{DUMER20061667}, \cite[Theorem~2]{dumer2004coverings}, and  \cite[Theorems~8 and 9]{Dumer2006}. 
Their setting differs slightly from ours as they allow the dimension $d$ of the ellipsoid, its semi-axes, and $\varepsilon$ to vary in such a way that $\intk{1}(\varepsilon)$ tends to infinity while satisfying
\begin{equation}\label{eq:condition-DPP}
  \max_{1\leq n \leq d}\ln\left(\frac{\mu_n}{\varepsilon}\right)
  = o\, \left(\frac{\left(\intk{1}(\varepsilon)\right)^2}{\ecv(\varepsilon)\ln(d)}\right).
\end{equation}
Heuristically, one can interpret \eqref{eq:condition-DPP} by identifying the dimension $d$ with the effective dimension $\ecv(\varepsilon)$ and by keeping the semi-axes fixed so that only $\varepsilon$ varies.
In that case, \eqref{eq:condition-DPP} becomes
\begin{equation*}
  \sqrt{\ecv (\varepsilon)\ln\left(\ecv (\varepsilon)\right)\ln\left(\varepsilon^{-1}\right)} = o_{\varepsilon\to 0}(\intk{1}(\varepsilon)),
\end{equation*}
which plays the same role as (RC) in our setting, 
by ensuring that the error term in \eqref{eq:main-result-ME} is negligible compared to $\intk{1}(\varepsilon)$.

\section{Minimax Risk, Type-2, and Type-3 Integrals}\label{sec:minimax-risk}

We leave metric entropy for a moment and turn to the study of the minimax risk for non-parametric estimation in the Gaussian sequence model.
Informally, the task is to denoise an unknown element known to lie in a compact set $\mathcal{K}$, observed with Gaussian noise of level $\sigma$.
When $\sigma$ scales as $1/\sqrt{m}$, the problem is equivalent to non-parametric regression from $m$ noisy samples; 
we focus on the small-noise---equivalently, large-sample---regime.
In practice, $\mathcal{K}$ is often a coefficient body relative to a basis---e.g., Besov classes in a wavelet basis, see the work of Donoho and Johnstone \cite{donoho1998minimax,donoho1994minimax,donoho2002noising,donoho1994ideal,donoho1996neo}.
We refer to  \cite{tsybakovIntroductionNonparametricEstimation2009,johnstone2019estimation,nussbaum1999minimax} for a general treatment of non-parametric estimation, with a strong emphasis on the minimax risk for the Gaussian sequence model.

Specifically, for $\sigma>0$ and $\mathcal{K} \subseteq \ell^2(\Ns)$,
the Gaussian sequence model $\gsmk$ aims to recover a sequence $x \in \mathcal{K}$ from noisy observations $y = x + \sigma \xi$, 
or, in components
\begin{equation}\label{eq:gaussian-sequence-model}
  y_n = x_n + \sigma \xi_n, 
  \quad \text{for all } n\in\Ns,
\end{equation}
where $\xi=\{\xi_n\}_{n\in\Ns}$ has independent standard normal coordinates.
A statistical estimator is a measurable map  $\hat x_\sigma\colon \R^{\Ns} \to \ell^2(\Ns)$ that takes the noisy observation $y$ to an estimate $\hat x_\sigma(y)$ of the original parameter $x$.
Its performance is assessed by the mean squared error
\begin{equation*}
  \textnormal{MSE}_{\sigma}  \left(x, \hat x_\sigma; \mathcal{K}\right) 
  \coloneqq \mathbb{E}_{y\sim x}\left[\left\|\hat x_\sigma (y)-x \right\|^2_2\right],
\end{equation*}
where $y\sim x$ indicates that the expectation is taken over $y$ distributed according to \eqref{eq:gaussian-sequence-model}. 
The minimax risk is then defined as the best achievable performance in the worst-case scenario:
take the supremum of $\textnormal{MSE}_{\sigma}  (x, \hat x_\sigma; \mathcal{K})$ over $x \in \mathcal{K}$, and then the infimum over $\hat x_\sigma$.

\begin{definition}\label{def:minimax-linear-and-nonlinear}
  Let $\sigma>0$
  and $\mathcal{K} \subseteq \ell^2(\Ns)$.
  The minimax risk associated with the Gaussian sequence model $\gsmk$ is 
    \begin{equation}\label{eq:def-non-linear-risk}
    R_{\sigma}\left(\mathcal{K}\right)
    \coloneqq \newinf_{\hat x_{\sigma}}\,  \sup_{x \in \mathcal{K}} \, \textnormal{MSE}_{\sigma}  \left(x, \hat x_\sigma; \mathcal{K}\right).
  \end{equation}
  The linear minimax risk associated with the model $\gsmk$ is obtained from the minimax risk by restricting the infimum to linear estimators, that is,
  \begin{equation}\label{eq:def-linear-risk}
    R^L_{\sigma}\left(\mathcal{K}\right)
    \coloneqq \newinf_{\hat x_{\sigma} \text{ linear}}\,  \sup_{x \in \mathcal{K}} \, \textnormal{MSE}_{\sigma} \left(x, \hat x_\sigma; \mathcal{K}\right),
  \end{equation}
  where we say that an estimator is linear if there exists a sequence of coefficients  $\{c_n\}_{n\in\Ns}$ such that 
  $\hat x_{\sigma}(y) = \{c_ny_n\}_{n\in\Ns}$.
\end{definition}

\noindent 
We call the model $\gsmk$ consistent if $R_{\sigma}\left(\mathcal{K}\right) \to 0$ as $\sigma\to 0$.
For a closed and bounded $\mathcal{K}\subset \ell^2(\Ns)$, 
consistency of $\gsmk$ is equivalent to compactness of $\mathcal{K}$
(see \cite{Ibragimov1997,Ibragimov1977} and \cite[Theorem~5.7]{johnstone2019estimation}).
This motivates using the decay of the minimax risk as a measure of compactness:
the faster $R_{\sigma}\left(\mathcal{K}\right)$ tends to zero, the more compact $\mathcal{K}$ is.
Specializing to ellipsoids $\mathcal{K}=\ellip$, 
we thus expect a direct link between $R_{\sigma}\left(\ellip\right)$ (or $R_{\sigma}^L\left(\ellip\right)$)
and other measures of compactness previously considered, such as the type-$\typevar$ integrals $\intk{\typevar}$ for $\tau\geq 1$.
The next result makes this link precise by relating the linear minimax risk to the type-2 and type-3 integrals.

\begin{theorem}\label{thm:linear-risk-formula}
  Let $\sigma>0$
  and let $\mu=\semax$ be a sequence such that $\mu_n\to 0$ as $n\to\infty$.
  The linear minimax risk associated with the Gaussian sequence model $\gsm$ satisfies
  \begin{equation}\label{eq:linear-risk-closed-form}
     R_{\sigma}^L \left(\ellip\right)
     = \sigma^2 \varepsilon_{\sigma} \intk{2}(\varepsilon_{\sigma}),
  \end{equation}
  where $\varepsilon_{\sigma}$ is defined as the unique solution of the equation (in $\varepsilon >0$)
  \begin{equation}\label{eq:definition-critical-radius}
    \sigma^2 \left(2 \intk{3}(\varepsilon) - \frac{\intk{2}(\varepsilon)}{\varepsilon} \right) = 1.
  \end{equation}
\end{theorem}

\begin{proof}[Proof sketch]
  The relation \eqref{eq:linear-risk-closed-form} follows from \cite{pinsker1980optimal} after a Stieltjes integration by parts.
  The existence and uniqueness of the solution of \eqref{eq:definition-critical-radius} are obtained by studying the variations of the left-hand side as a function of $\varepsilon$.
  We refer to Appendix~\ref{sec:proof-theorem-linear-risk-formula} for the detailed proof.
\end{proof}

\noindent
Theorem~\ref{thm:linear-risk-formula} guarantees that the equation \eqref{eq:definition-critical-radius} admits a unique solution $\varepsilon_{\sigma}$,
which we call the critical radius of $\ellip$ for reasons that will become clear in Section~\ref{sec:metric-entropy-vs-minimax-risk}.
Using \eqref{eq:definition-Phi-RL}, \eqref{eq:reformulation-linear-risk-integral}, and \eqref{eq:last-step-proof-RL-1}--\eqref{eq:last-step-proof-RL-2} from the proof of Theorem~\ref{thm:linear-risk-formula}, the relation between the critical radius and the linear minimax risk can be equivalently written as
\begin{equation}\label{eq:linear-risk-integral}
    R_{\sigma}^L \left(\ellip\right)
    = \newinf_{\varepsilon>0} \left\{ 2 \sigma^2  \varepsilon \,  \left(\intk{2}(\varepsilon)-\intk{3}(\varepsilon) \, \varepsilon \right)+ \varepsilon^2\right\}
    =2\sigma^2 \varepsilon_{\sigma} \,  \left(\intk{2}(\varepsilon_{\sigma})-\intk{3}(\varepsilon_{\sigma}) \, \varepsilon_{\sigma} \right)+ \varepsilon_{\sigma}^2,
\end{equation}
for all $\sigma>0$.
The identity \eqref{eq:linear-risk-integral} is central to the proof of Theorem~\ref{thm:bias-variance-decomp}, stated in Section~\ref{sec:metric-entropy-vs-minimax-risk}.

To illustrate the use of Theorem~\ref{thm:linear-risk-formula}, 
we come back to the regularly varying semi-axis-counting function considered after Theorem~\ref{thm:metric-entropy-integral-form}, namely $\ecv(\varepsilon) \sim c \, \varepsilon^{-\alpha}$ as $\varepsilon\to 0$ for $c, \alpha>0$.
In this case (see \eqref{eq:ecv-implies-Itau-asymp} in Appendix~\ref{sec:asymp-equiv-integration}),
\begin{equation}\label{eq:I2-I3-reg-var-asymptotics}
  \intk{2}(\varepsilon) 
  \sim \frac{c \, \varepsilon^{-\alpha-1}}{\alpha+  1}
  \quadtext{and}
  \intk{3}(\varepsilon) 
  \sim \frac{c \, \varepsilon^{-\alpha-2}}{\alpha+2},
  \quad \text{as } \varepsilon\to 0.
\end{equation}
Using these relations to solve \eqref{eq:definition-critical-radius} asymptotically yields 
\begin{equation}\label{eq:critical-radius-asymp-reg-var}
  \varepsilon_\sigma 
  \sim \left(\frac{c\, \alpha\, \sigma^2}{(\alpha+1)(\alpha+2)}\right)^{\frac{1}{\alpha+2}},
  \quad \text{as } \sigma\to 0,
\end{equation}
which, when substituted into \eqref{eq:linear-risk-closed-form}, gives 
\begin{equation}\label{eq:linear-risk-reg-var}
  R_{\sigma}^L \left(\ellip\right)
  \sim \left(\frac{\alpha+2}{\alpha}\right)^{\frac{\alpha}{\alpha+2}} \left(\frac{c\, \sigma^2 }{\alpha+1}\right)^{\frac{2}{\alpha+2}},
  \quad \text{as } \sigma\to 0.
\end{equation}
This recovers the so-called Pinsker constant; 
compare with \cite[Theorem~3.1]{tsybakovIntroductionNonparametricEstimation2009} upon setting $\alpha = 1/\beta$ and $c = Q^{1/\beta}$.
Similar arguments yield the leading term of the linear minimax risk for various asymptotic behaviors of $\ecv$; 
we provide some examples in Appendix~\ref{sec:table}.
Note that Theorem~\ref{thm:linear-risk-formula} further allows for the derivation of higher orders.
For instance, if $\ecv(\varepsilon) = c _1 \varepsilon^{-\alpha_1} + c _2 \varepsilon^{-\alpha_2} +o_{\varepsilon\to 0}(\varepsilon^{-\alpha_2})$, for $c_1, c_2>0$ and $\alpha_1>\alpha_2>0$, 
following the same steps as \eqref{eq:I2-I3-reg-var-asymptotics}--\eqref{eq:linear-risk-reg-var} while keeping track of the second order terms yields
\begin{align}
  R_{\sigma}^L \left(\ellip\right)
  = 
  &\left(\frac{\alpha_1+2}{\alpha_1}\right)^{\frac{\alpha_1}{\alpha_1+2}} \left(\frac{c_1 \sigma^2 }{\alpha_1+1}\right)^{\frac{2}{\alpha_1+2}} \label{eq:two-term-linear-risk-reg-var}\\
  &\quad + \frac{2c_2(\alpha_1+1)}{c_1(\alpha_2+1)(\alpha_2+2)}\left(\frac{\alpha_1+2}{\alpha_1}\right)^{\frac{\alpha_2}{\alpha_2+2}}\left(\frac{c_1 \sigma^2 }{\alpha_1+1}\right)^{\frac{\alpha_1-\alpha_2+2}{\alpha_1+2}}
  + o_{\sigma\to 0}\left(\sigma^{\frac{2\alpha_1-2\alpha_2+4}{\alpha_1+2}}\right).\nonumber
\end{align}
An analogous derivation is provided in great details in the proof of Theorem~\ref{thm:pinsker-sobolev-order-two} (see Appendix~\ref{sec:proof-pinsker-sob-order-two}).

Corresponding results for the non-linear minimax risk follow by asymptotically relating it to the linear minimax risk.
In fact, Pinsker's asymptotic minimaxity theorem \cite[Theorem~5.3]{johnstone2019estimation} guarantees $R_{\sigma}\left(\ellip\right) \sim R^L_{\sigma}\left(\ellip\right)$, as $\sigma\to 0$.
Consequently, Theorem~\ref{thm:linear-risk-formula} also provides the leading term of the non-linear minimax risk: $R_{\sigma}\left(\ellip\right) \sim \sigma^2 \varepsilon_{\sigma} \intk{2}(\varepsilon_{\sigma})$ as $\sigma\to 0$.
To characterize higher order terms, however, 
a stronger version of Pinsker's asymptotic minimaxity theorem is required.
Theorem~\ref{thm:risk-linear-2-nonlinear} characterizes the error term in Pinsker's asymptotic minimaxity theorem in terms of the Lambert W function \cite{corlessLambertFunction1996}.

\begin{theorem}\label{thm:risk-linear-2-nonlinear}
  Let $\mu=\semax$ be a sequence such that $\mu_n\to 0$ as $n\to\infty$.
  Then, the linear and non-linear minimax risks associated with the Gaussian sequence model $\gsm$ satisfy
  \begin{equation}\label{eq:non-asymp-ratio-lin-non-lin-risks}
    \left\lvert \frac{R_{\sigma}\left(\ellip\right)}{R^L_{\sigma}\left(\ellip\right)} -1 \right\rvert
    \leq \frac{4\sqrt{2} \,  \sigma}{\varepsilon_\sigma}\sqrt{W \left(\left(\frac{\left(1+\sqrt{3}\right) \mu_*^2 \, \varepsilon_\sigma}{ \sqrt{2} \, \sigma R^L_{\sigma}\left(\ellip\right)}\right)^2\right)},
    \quad \text{for all } \sigma>0,
  \end{equation}
  where $\varepsilon_\sigma$ is the critical radius of $\ellip$, $W$ is the Lambert W function, and $\mu_* \coloneqq \max_{n\in\Ns} \mu_n$.
\end{theorem}

\begin{proof}[Proof sketch]
  Theorem~\ref{thm:risk-linear-2-nonlinear} is obtained by keeping track of the error terms in the classical proof of Pinsker's  asymptotic minimaxity theorem;
  see Appendix~\ref{sec:proof-theorem-risk-linear-2-nonlinear}.
\end{proof}

\noindent
Using the asymptotics of the Lambert W function---namely, $W(x) \sim \ln(x)$ as $x\to\infty$ (see \cite{corlessLambertFunction1996} and \cite[Chapter~2.4]{de1981asymptotic})---Theorem~\ref{thm:risk-linear-2-nonlinear} can be turned into the asymptotic statement 
\begin{equation}\label{eq:asymp-linear-to-nonlinear-risk}
  R_{\sigma}\left(\ellip\right)
  = R^L_{\sigma}\left(\ellip\right) + O_{\sigma\to 0} \left(\frac{\sqrt{\ln(v_\sigma)}}{v_\sigma}\right),
  \quadtext{where} v_\sigma \coloneqq \frac{\varepsilon_\sigma}{\sigma R^L_{\sigma}\left(\ellip\right)},
\end{equation}
if $v_\sigma \to \infty$ as $\sigma\to 0$.
Thus, the linear and non-linear risks differ by a remainder that is asymptotically negligible whenever $\sqrt{\ln(v_\sigma)} = o_{\sigma\to 0}(v_\sigma R^L_{\sigma}\left(\ellip\right))$, in which case \eqref{eq:asymp-linear-to-nonlinear-risk} improves on Pinsker's asymptotic minimaxity theorem.
This provides a general strategy for obtaining higher-order asymptotics of the non-linear risk: 
(i) derive the expansion of the linear risk via Theorem~\ref{thm:linear-risk-formula}, and (ii) transfer it by showing that the $O_{\sigma\to 0}(\cdot)$-term in \eqref{eq:asymp-linear-to-nonlinear-risk} is negligible. 
For instance, if $\ecv(\varepsilon) \sim c_1 \, \varepsilon^{-\alpha_1}$ as $\varepsilon\to 0$ with $c_1, \alpha_1>0$, 
then \eqref{eq:critical-radius-asymp-reg-var} and \eqref{eq:linear-risk-reg-var} give $v_\sigma \stackrel{K}{\sim} \sigma^{-(\alpha_1+4)/(\alpha_1+2)}$ as $\sigma\to 0$.
In particular, the $O_{\sigma\to 0}(\cdot)$-term in \eqref{eq:asymp-linear-to-nonlinear-risk} is dominated by the $o_{\sigma\to 0}(\cdot)$-term appearing in \eqref{eq:two-term-linear-risk-reg-var} if $2\alpha_2>\alpha_1$.
Consequently, combining \eqref{eq:two-term-linear-risk-reg-var} with \eqref{eq:asymp-linear-to-nonlinear-risk}, if  $\ecv(\varepsilon) = c _1 \varepsilon^{-\alpha_1} + c _2 \varepsilon^{-\alpha_2} +o_{\varepsilon\to 0}(\varepsilon^{-\alpha_2})$, for $c_1, c_2>0$ and $ \alpha_1, \alpha_2>0$ such that $\alpha_2<\alpha_1<2\alpha_1$, then
\begin{align*}
  R_{\sigma} \left(\ellip\right)
  = 
  &\left(\frac{\alpha_1+2}{\alpha_1}\right)^{\frac{\alpha_1}{\alpha_1+2}} \left(\frac{c_1 \sigma^2 }{\alpha_1+1}\right)^{\frac{2}{\alpha_1+2}}\\
  &\quad + \frac{2c_2(\alpha_1+1)}{c_1(\alpha_2+1)(\alpha_2+2)}\left(\frac{\alpha_1+2}{\alpha_1}\right)^{\frac{\alpha_2}{\alpha_2+2}}\left(\frac{c_1 \sigma^2 }{\alpha_1+1}\right)^{\frac{\alpha_1-\alpha_2+2}{\alpha_1+2}}
  + o_{\sigma\to 0}\left(\sigma^{\frac{2\alpha_1-2\alpha_2+4}{\alpha_1+2}}\right).
\end{align*}

\section{Metric Entropy and Minimax Risk}\label{sec:metric-entropy-vs-minimax-risk}

Section~\ref{sec:metric-entropy} shows that the metric entropy of $\ellip$ is asymptotically determined by the type-$1$ integral $\intk{1}$.
Section~\ref{sec:minimax-risk} characterizes the linear minimax risk---and therefore the asymptotics of the (non-linear) minimax risk---in terms of the type-$2$ and type-$3$ integrals $\intk{2}$ and $\intk{3}$.
To obtain an explicit relation between metric entropy and minimax risk,
we thus need to relate the asymptotics of the various integrals $\intk{\typevar}$ as $\typevar$ varies.

First, the asymptotics of $\ecv$ always determine those of $\intk{\typevar}$, for every $\typevar\geq 1$; 
this is a consequence of the preservation of asymptotics under integration (see Lemma~\ref{lem:abelian-integration} in Appendix~\ref{sec:asymp-equiv-integration} and the subsequent discussion).
The converse fails in general, 
but it does hold for regularly varying $\intk{\typevar}$ (see Lemma~\ref{lem:tauberian-integration} in Appendix~\ref{sec:asymp-equiv-integration}).
However, even in the regularly varying case, second-order asymptotics for $\intk{\typevar}$ do not in general transfer to $\ecv$;
this limitation will play a concrete role in Section~\ref{sec:appli-Sobolev}.

We just saw that $\ecv$ asymptotically determines $\intk{\typevar}$,
but how do the integrals $\intk{\typevar}$ compare for different values of $\typevar$? 
The next lemma provides a formula to transfer properties from a type-$\typevar_1$ integral to a type-$\typevar_2$ integral.

\begin{lemma}\label{lem:integral-k-to-k+1}
  Let $\typevar_1, \typevar_2 \geq 1$   
  and let $\mu=\semax$ be a sequence of semi-axes such that $\mu_n\to 0$ as $n\to\infty$.
  Then, the corresponding 
  integrals of types $\typevar_1$ and $\typevar_2$ satisfy
  \begin{equation}\label{eq:lemma-integral-k-to-k+1}
  \intk{\typevar_1}(\varepsilon) 
  = \intk{\typevar_2}(\varepsilon)\, \varepsilon^{\typevar_2-\typevar_1} 
  + (\typevar_2-\typevar_1) \int_{\varepsilon}^{\infty} \intk{\typevar_2}(\dumvar)\, \dumvar^{\typevar_2-\typevar_1-1} \diff \dumvar,
  \quad \text{for all } \varepsilon>0.
\end{equation}
\end{lemma}

\begin{proof}[Proof sketch]
  The result follows directly via integration by parts; 
  the details are provided in Appendix~\ref{sec:proof-lemma-integral-k-to-k+1}. 
\end{proof}

\noindent
Lemma~\ref{lem:integral-k-to-k+1} is non-asymptotic in nature.
However, when $\rcb$ holds with $b \in [0, \infty]$, 
it can be turned into an asymptotic statement, 
except in the degenerate case $b = 0$ and $\typevar_2 = 1$.
(The later condition prevents the asymptotics of the two terms on the right-hand side of \eqref{eq:lemma-integral-k-to-k+1} from cancelling out, see Appendix~\ref{sec:asymp-equiv-integration}.)
In all other cases, the leading term of $\intk{\typevar_2}$ determines the leading term of $\intk{\typevar_1}$, for $\typevar_1, \typevar_2 \geq 1$.

Particularizing to metric entropy,
we conclude that its leading term is determined by that of $\ecv$ and by that of $\intk{\typevar}$, for any $\typevar\geq 1$; 
the converse however does not hold.
The situation for the minimax risk is slightly more involved and requires to split the analysis into three regimes, depending on the value of $b$.
This is the content of the next theorem.

\begin{theorem}\label{thm:bias-variance-decomp}
  Let $b\in[0,\infty]$ and
  let $\mu=\semax$ be a sequence such that $\mu_n\to 0$ as $n\to\infty$,
  with corresponding semi-axis-counting function $\ecv$ satisfying  $\rcb$.
  Then, the minimax risk associated with the model $\gsm$ satisfies
  \begin{equation}\label{eq:bias-variance}
    R_{\sigma} \left(\ellip\right)
    \sim \newinf_{\varepsilon > 0} \left\{ \sigma^2 \, V_b(\varepsilon) + \varepsilon^2\right\}
    \sim \sigma^2 \, V_b(\varepsilon_\sigma) + \varepsilon_\sigma^2,
    \quad\text{as } \sigma\to 0, 
  \end{equation}
  where $\varepsilon_\sigma$ is the critical radius introduced in Theorem~\ref{thm:linear-risk-formula} and 
  \begin{equation*}
    V_b(\varepsilon) \coloneqq  
    \begin{dcases}
      \\[-.35cm]
         \ecv(\varepsilon), &\quad\text{if } b=0, \\[.5cm]
          \frac{2 \, b\, H\left(\varepsilon; \ellip\right)}{(b+1)(b+2)} , &\quad\text{if } b \in (0, \infty),\\[.225cm]
          \int_{\varepsilon}^{\infty}\frac{2H\left(\dumvar; \ellip\right)}{\dumvar}\diff \dumvar , & \quad\text{if } b=\infty, 
          \\[.175cm]
    \end{dcases}
    \quad \text{for all } \varepsilon>0.
  \end{equation*}
\end{theorem}

\begin{proof}[Proof sketch]
Upon application of Pinsker's asymptotic minimaxity  theorem,
we get that the linear minimax risk associated with the model $\mathcal{G}(\ellip, \sigma)$ is asymptotically optimal, that is,  $R_{\sigma} \left(\ellip\right) \sim  R_{\sigma}^L \left(\ellip\right)$, as $ \sigma \to 0$.
It is therefore sufficient to prove (\ref{eq:bias-variance}) for the linear minimax risk.
This is done by first applying Lemma~\ref{lem:integral-k-to-k+1} to establish a relation between the integrals $\intk{2}$ and $\intk{3}$ appearing in \eqref{eq:linear-risk-integral}, and then applying Theorem~\ref{thm:metric-entropy-integral-form}.
The detailed proof is provided in Appendix~\ref{sec:proof-lemma-RnL-vs-ME}.
\end{proof}

\noindent
Theorem~\ref{thm:bias-variance-decomp} splits into three regimes. 
When $b=0$, the leading term of the minimax risk is determined by the asymptotics of the semi-axis-counting function; see Lemma~\ref{lem:abelian-infimum} in Appendix~\ref{sec:asymp-equiv-integration}.
When $b\in (0, \infty)$, knowledge of the leading term of the minimax risk is equivalent to knowledge of the leading term of the metric entropy (and hence of $\intk{\typevar}$, for all $\typevar \geq 1$); see Lemmas~\ref{lem:abelian-infimum} and \ref{lem:tauberian-infimum} in Appendix~\ref{sec:asymp-equiv-integration}.
When $b=\infty$, the leading term of the minimax risk is determined by that of the metric entropy (and thus of  $\intk{\typevar}$, for all $\typevar\geq 1$), but the converse generally fails.
These dependencies are made explicit for the typical examples presented in Appendix~\ref{sec:table}.

To interpret \eqref{eq:bias-variance},
we focus on $b \in (0, \infty)$ (the discussion extends to the other regimes).
By Theorem~\ref{thm:bias-variance-decomp},
the minimax risk admits the asymptotic decomposition
\begin{equation}\label{eq:bias-variance-decomp}
R_{\sigma} \left(\ellip\right)
\sim
\frac{2 \,  b \,   \sigma^2}{(b+1)(b+2)} H \left(\varepsilon_\sigma ; \ellip \right)
+ \varepsilon^2_\sigma,
\quad \text{as } \sigma \to 0.
\end{equation}
This decomposition is an incarnation of the bias--variance tradeoff from statistical learning theory (see, e.g., \cite[Chapter 5]{understandingML2014}): 
smaller values of $\varepsilon_\sigma$ reduce the bias term $\varepsilon^2_\sigma$, but at the expense of increasing the variance---or complexity---term captured by the metric entropy. Conversely, larger values of $\varepsilon_\sigma$ increase the bias while reducing the variance.
Theorem~\ref{thm:bias-variance-decomp} shows that the leading term of the minimax risk is obtained by optimizing this tradeoff; $\varepsilon_\sigma$ achieves the optimum, hence the terminology `critical radius'.

We conclude with a short discussion on related work.
Characterizing the performance of statistical estimators via metric-entropy methods is a long-standing idea, 
dating back at least to the
works of LeCam and Birgé \cite{birge1983approximation,10.1214/aos/1193342380};
see also \cite{vandervaartWeakConvergenceEmpirical1996,geer2000empirical,GineNickl2021}.
For minimax risk over ellipsoids, Donoho \cite{donohoCountingBitsKolmogorov2000} highlighted a strong parallel between minimax techniques and metric-entropy methods. More recently, \cite{zhu2018quantized} studied the minimax risk as an optimization problem under storage or communication constraints captured by metric entropy. 
A bias-variance tradeoff akin to \eqref{eq:bias-variance-decomp}---i.e., where an $\varepsilon^2$-term competes against a metric-entropy term---already appears in \cite{yang1999information} to derive lower bounds on the minimax risk in a broader setting than ellipsoids.

\section{Sobolev Ellipsoids and Pinsker's Theorem}\label{sec:appli-Sobolev}

Sections~\ref{sec:metric-entropy}--\ref{sec:metric-entropy-vs-minimax-risk} 
establish that the type-$\typevar$ integrals $\intk{\typevar}$ (for $\typevar\geq 1$)
quantify the compactness of ellipsoids 
as they directly relate to metric entropy and minimax risk. 
Section~\ref{sec:metric-entropy-vs-minimax-risk} and Appendix~\ref{sec:asymp-equiv-integration} further show that the asymptotics of $\intk{\typevar}$ are determined by those of $\ecv$, 
which might suggest that $\ecv$ is an equally suitable measure of compactness.
We claim otherwise: the semi-axis-counting function carries, in a sense, too much information.
To support our claim, we particularize our analysis to the Sobolev ellipsoid and show that, in this case, working directly with type-$\typevar$ integrals yields stronger results than methods based on the asymptotics of $\ecv$.
Additionally, the Sobolev ellipsoid provides a natural test-case for the results of the paper
given the extensive literature on its metric entropy \cite{edmunds1979embeddings,donohoCountingBitsKolmogorov2000,luschgyFunctionalQuantizationGaussian2002,luschgySharpAsymptoticsKolmogorov2004,ALLARD2025101762} 
and its minimax risk (mainly via Pinsker's Sobolev theorem \cite{pinsker1980optimal,nussbaum1999minimax}).
The theory laid out in Sections~\ref{sec:metric-entropy}--\ref{sec:metric-entropy-vs-minimax-risk}  improves on all these classical characterizations.

Specifically, for $k,d\in\Ns$, we consider the Sobolev space of order $k$ on a bounded open domain $\Omega\subset \R^d$ (see \cite{adamsSobolevSpaces2003, brezisFunctionalAnalysisSobolev2011}). 
Informally, this space corresponds to the class of square-integrable functions on $\Omega$---i.e., elements of $L^2(\Omega)$---whose $k$-th (weak)  derivatives also lie in $L^2(\Omega)$.
The Sobolev ellipsoid $\mathcal{E}^{\text{Sob}}_{d, k}$ is the unit ball of this space. 
Equivalently,  $\mathcal{E}^{\text{Sob}}_{d, k}$  
is the image of the unit ball of $L^2(\Omega)$ under the compact linear operator
\begin{equation}\label{eq:operator-T}
  T= \left(\text{Id}+ (-\Delta)^k\right)^{-\frac{1}{2}},
\end{equation}
where $-\Delta$ denotes the Dirichlet Laplacian on $\Omega$;
see the discussion in \cite[Section~3.2]{ALLARD2025101762}.
This makes the terminology `Sobolev ellipsoid' transparent:
$\mathcal{E}^{\text{Sob}}_{d, k}$ is isometric to the ellipsoid in $\ell^2(\Ns)$ with semi-axes $\semax$ given by the eigenvalues of the operator $T$.

To apply the results of the previous sections,
we need information on the asymptotic properties of $\semax$.  
By \eqref{eq:operator-T},
this is equivalent to an asymptotic characterization of the eigenvalues $\{\lambda_n\}_{n\in\Ns}$ of $-\Delta$.
On the template of \eqref{eq:definition-ecf}, we thus introduce the eigenvalue-counting function of the Laplacian
\begin{equation*}
    M_\lambda (s) \coloneqq 
    \left| \left\{n \mid \lambda_n \leq s \right\}\right|,
    \quad \text{for all } s>0.
\end{equation*}
Note the fundamental difference with $\ecv$: 
$M_\lambda(s)$ counts the eigenvalues of $-\Delta$ below a threshold $s$, 
while $\ecv(\varepsilon)$ counts the semi-axes of $\mathcal{E}^{\text{Sob}}_{d, k}$ above a threshold $\varepsilon$.
This difference reflects that the eigenvalues of the Laplacian tend to infinity \cite[Chapter~9.5]{shubinInvitationPartialDifferential2020}, 
while the semi-axes of the ellipsoid tend to zero.
Accordingly, we are interested in the asymptotics of $M_\lambda(s)$ in the limit $s\to\infty$.
The leading term was first characterized by Weyl in \cite{weylUeberAsymptotischeVerteilung1911,weyl1912asymptotische,weyl1913randwertaufgabe} (see also \cite[Chapter~9.5]{shubinInvitationPartialDifferential2020}) according to 
\begin{equation}\label{eq:Weyl-law-one-term}
  M_{\lambda} (s) 
  \sim \frac{\omega_d \mathcal{H}^d(\Omega)}{(2\pi)^d} \, s^{\frac{d}{2}},
  \quad \text{as } s\to\infty,
\end{equation}
where $\omega_d$ is the volume of the unit ball in $\mathbb{R}^d$ and $\mathcal{H}^d(\cdot)$ denotes the $d$-dimensional Hausdorff measure.

A second-order term can also be obtained, typically under the additional assumptions
\begin{equation}\label{eq:billiard-condition}
  \text{$\Omega$ has smooth boundary $\partial\Omega$ and the measure of its periodic billiards is zero.}
\end{equation}
The domain $\Omega$ has smooth boundary if it is locally the graph of a smooth---i.e., infinitely differentiable---function; see  \cite[Chapter~4]{adamsSobolevSpaces2003} and \cite[Appendix~C]{evans2022partial}.
Informally, a billiard 
is 
the data of an initial point and an initial direction, generating a trajectory in $\Omega$ which is said to be periodic if it loops back onto itself
; we refer to \cite[Chapter 1.3]{safarov1997asymptotic} for a formal introduction to billiards.
Under \eqref{eq:billiard-condition}, V. Ivrii established in \cite{ivriiSecondTermSpectral1980} (see also \cite[Corollary 29.3.4]{hormanderAnalysisLinearPartial2009} and the survey \cite{Ivrii_2016}) the two-term expansion
\begin{equation}\label{eq:Weyl-two-terms-ivrii}
  M_{\lambda} (s) 
  = d \, \chi_d(\Omega) \, s^{\frac{d}{2}}  
   - \frac{(d-1) \chi_{d-1}(\partial \Omega)}{4} \, s^{\frac{d-1}{2}} 
  + o_{s \to \infty} \left(s^{\frac{d-1}{2}} \right),
\end{equation}
where $\chi_r (S) 
\coloneqq \omega_r  \mathcal{H}^r(S) /(r \, (2\pi)^r) $, for all $r\in\{1, \dots, d\}$ and measurable $S\subset \mathbb{R}^d$.
In \cite[p. 1001]{frank2024riesz}, 
the authors report that Bronstein--Ivrii claim \eqref{eq:Weyl-two-terms-ivrii} holds under a  $C^1$-Dini smoothness assumption on $\partial\Omega$, slightly weaker than \eqref{eq:billiard-condition}, 
still under the periodic-billiard condition.
Combining \eqref{eq:operator-T} and \eqref{eq:Weyl-two-terms-ivrii},
one can show
\begin{equation}\label{eq:ecv-Sobb-order-2}
  \ecv(\varepsilon)
  = d\, \chi_d(\Omega) \, {\varepsilon}^{-\frac{d}{k}}  
  - \frac{(d-1)\, \chi_{d-1}(\partial \Omega)}{4}\, 
  \varepsilon^{-\frac{d-1}{k}} 
  + o_{\varepsilon \to 0}\left(\varepsilon^{-\frac{d-1}{k}}\right),
\end{equation}
see \cite[Proof of Theorem~6]{ALLARD2025101762} for an explicit derivation.

There exists a corresponding literature that characterizes two-term asymptotics of the Laplacian eigenvalues `on average';
see, e.g., \cite{Frank_2011,geisingerSemiclassicalLimitDirichlet2011,Frank_2019,larson2019asymptotic,frankErrorTwotermWeyl2020,frank2024riesz}.
In particular, \cite[Theorem~1.1]{frank2024riesz} shows that, if $d\geq 2$ and $\Omega$ is a bounded open subset of $\mathbb{R}^d$ with Lipschitz boundary (i.e., it is locally the graph of a Lipschitz function), the Riesz means of the eigenvalues $\{\lambda_n\}_{n\in\Ns}$ satisfy
\begin{equation}\label{eq:asymptotics-Laplace-story}
    \sum_{n\in\Ns} \left(1-h^2\lambda_n\right)_+
    = \frac{2 \,  d \, \chi_d(\Omega)}{(d+2)} h^{-d} 
    - \frac{(d-1) \chi_{d-1}(\partial \Omega)}{2(d+1)} h^{-(d-1)}
    +o_{h\to 0}\left(h^{-(d-1)}\right).
\end{equation}
For background on Riesz means, we refer to \cite[Remark~18.3]{korevaar2004tauberian} and the references therein.
Just as the two-term asymptotics \eqref{eq:Weyl-two-terms-ivrii} for $M_{\lambda}$ are equivalent to the two-term asymptotics \eqref{eq:ecv-Sobb-order-2} for $\ecv$,
we  will see in the proof of Theorem~\ref{thm:Metric-Entropy-Sobolev} (cf. \eqref{eq:asymptotics-integral-12k}) that \eqref{eq:asymptotics-Laplace-story} is equivalent to a two-term expansion for the type-$(1+2/k)$ integral $\intk{1+2/k}$ associated with $\mathcal{E}^{\text{Sob}}_{d, k}$.
The advantage of \eqref{eq:asymptotics-Laplace-story} over \eqref{eq:Weyl-two-terms-ivrii}---when one is satisfied with information on $\intk{1+2/k}$ instead of $\ecv$---is that \eqref{eq:asymptotics-Laplace-story} requires only 
$\Omega$ to have Lipschitz boundary, rather than smooth boundary with a periodic-billiard condition as in \eqref{eq:billiard-condition}.
This observation is particularly relevant for our analysis since, as discussed in Section~\ref{sec:metric-entropy-vs-minimax-risk}, asymptotics for $\intk{1+2/k}$ imply asymptotics for $\intk{\typevar}$ for all $\typevar \geq 1$, and hence for metric entropy and minimax risk.

To make our discussion concrete, let us first consider metric entropy.
It was shown in \cite[Theorem~5]{ALLARD2025101762} that the metric entropy of the Sobolev ellipsoid scales to the first order according to
\begin{equation}\label{eq:metric-entropy-sob-first-order}
  H \left(\varepsilon; \mathcal{E}^{\text{Sob}}_{d, k}\right) 
  \sim k\, \chi_d(\Omega) \, {\varepsilon}^{-\frac{d}{k}},
  \quad \text{as } \varepsilon \to 0,
\end{equation}
and, when $d\geq 3$ and under the additional assumption \eqref{eq:billiard-condition},
to the second order according to
\begin{equation}\label{eq:Sobolev-ME-order-two-ACHA}
  H \left(\varepsilon; \mathcal{E}^{\text{Sob}}_{d, k}\right) 
  =k\, \chi_d(\Omega) \, {\varepsilon}^{-\frac{d}{k}}  
  - \frac{k\, \chi_{d-1}(\partial \Omega)}{4}\, 
  \varepsilon^{-\frac{d-1}{k}} 
  + o_{\varepsilon \to 0}\left(\varepsilon^{-\frac{d-1}{k}}\right).
\end{equation}
The asymptotic expansion \eqref{eq:Sobolev-ME-order-two-ACHA} was derived from \eqref{eq:Weyl-two-terms-ivrii}.
Building on \eqref{eq:asymptotics-Laplace-story} and Theorem~\ref{thm:metric-entropy-integral-form}, 
the next theorem shows that \eqref{eq:Sobolev-ME-order-two-ACHA} can in fact be recovered without  \eqref{eq:billiard-condition}, assuming only that $\Omega$ has Lipschitz boundary.

\begin{theorem}\label{thm:Metric-Entropy-Sobolev}
  Let $k \in \mathbb{N}^*$, let $d\geq 3$ be an integer, and let $\Omega$ be a bounded open subset of $\mathbb{R}^d$ with Lipschitz boundary.
  Then, 
  \begin{equation*}
  H \left(\varepsilon; \mathcal{E}^{\text{Sob}}_{d, k}\right) 
  =k\, \chi_d(\Omega) \, {\varepsilon}^{-\frac{d}{k}}  
  - \frac{k\, \chi_{d-1}(\partial \Omega)}{4}\, 
  \varepsilon^{-\frac{d-1}{k}} 
  + o_{\varepsilon \to 0}\left(\varepsilon^{-\frac{d-1}{k}}\right),
  \end{equation*}
  where $\chi_r (S) = \omega_r  \mathcal{H}^r(S) /(r \, (2\pi)^r) $, for all $r\in\{1, \dots, d\}$ and measurable $S\subset \mathbb{R}^d$, 
  with $\omega_r$ the volume of the unit ball in $\mathbb{R}^r$ and $\mathcal{H}^r(\cdot)$ the $r$-dimensional Hausdorff measure.
\end{theorem}

\begin{proof}[Proof sketch]
  Theorem~\ref{thm:Metric-Entropy-Sobolev} is obtained by first deriving the asymptotic behavior of the type-$1$ integral $\intk{1}$ from \eqref{eq:asymptotics-Laplace-story}, and then applying Theorem~\ref{thm:metric-entropy-integral-form}.
  The details are provided in Appendix~\ref{sec:proof-theorem-Metric-Entropy-Sobolev}.
\end{proof}

\noindent
Our next application concerns the minimax risk for Sobolev ellipsoids in the Gaussian sequence model.
For $d=1$ and $\Omega=(0,1)$,
Pinsker proved in \cite{pinsker1980optimal} that the minimax risk associated with the model $\mathcal{G}(\mathcal{E}^{\text{Sob}}_{1, k}, \sigma)$ scales according to
\begin{equation}\label{eq:Pinsker-original-version}
  R_{\sigma} \left(\mathcal{E}^{\text{Sob}}_{1, k}\right)
  \sim P_k \,  \sigma^{\frac{4k}{2k+1}},
  \quad \text{as } \sigma\to 0, \text{ where }
  P_k \coloneqq \left(2k+1\right)^{\frac{1}{2k+1}}\left(\frac{k}{\pi(k+1)}\right)^{\frac{2k}{2k+1}},
\end{equation}
is the Sobolev Pinsker constant. 
Theorem~\ref{thm:pinsker-sobolev-order-two} extends \eqref{eq:Pinsker-original-version} in two ways: 
(i) to general open bounded domains $\Omega\subset\R^d$ in \eqref{eq:Pinsker-order-1}, and 
(ii) by providing the second-order term when $d\geq 3$ and $\Omega$ has Lipschitz boundary in \eqref{eq:Pinsker-order-2}.
Again, the assumptions in \eqref{eq:billiard-condition} are not required.

\begin{theorem}\label{thm:pinsker-sobolev-order-two}
  Let $k \in \Ns$, $d\in\Ns$, and let $\Omega$ be a bounded open subset of $\mathbb{R}^d$.
  Then, 
  \begin{equation}\label{eq:Pinsker-order-1}
    R_{\sigma} \left(\mathcal{E}^{\text{Sob}}_{d, k}\right)
    \sim \frac{d+2k}{d} \left(\frac{k \,  d^2 \,  \chi_d(\Omega)  \, \sigma^2}{(d+k)(d+2k)}\right)^{\frac{2k}{d+2k}},
    \quad\text{as } \sigma\to 0,
  \end{equation}
  where $\chi_r (S) = \omega_r  \mathcal{H}^r(S) /(r \, (2\pi)^r) $, for all $r\in\{1, \dots, d\}$ and measurable $S\subset \mathbb{R}^d$, 
  with $\omega_r$ the volume of the unit ball in $\mathbb{R}^r$ and $\mathcal{H}^r(\cdot)$ the $r$-dimensional Hausdorff measure.

  Additionally, if $d\geq 3$ and $\Omega$ has Lipschitz boundary, then
  \begin{equation}\label{eq:Pinsker-order-2}
    R_{\sigma} \left(\mathcal{E}^{\text{Sob}}_{d, k}\right)
    = K_1 \left(\kappa\sigma^2\right)^{\frac{2k}{d+2k}}
    + K_2 \left(\kappa\sigma^2\right)^{\frac{2k+1}{d+2k}} + o_{\sigma\to 0} \left(\sigma^{\frac{4k+2}{d+2k}} \right),
  \end{equation}
  where 
  \begin{equation*}
    \kappa 
    = \frac{k \,  d^2 \,  \chi_d(\Omega)}{(d+k)(d+2k)},
    \quad 
    K_1
    = \frac{d+2k}{d},
    \quadtext{and}
    K_2 
    = - \frac{k(d-1)(d+k)(d+2k)\chi_{d-1}(\partial\Omega)}{2 d^2 (d+k-1)(d+2k-1)\chi_d(\Omega)}.
  \end{equation*}
\end{theorem}

\begin{proof}[Proof sketch]
  Theorem~\ref{thm:pinsker-sobolev-order-two} is obtained in three steps. First, we derive the asymptotic behavior of the integrals $\intk{2}$ and $\intk{3}$ from \eqref{eq:asymptotics-Laplace-story}.
  This allows us to determine the asymptotic behavior of the critical radius $\varepsilon_{\sigma}$, and thus of the linear minimax risk $R_{\sigma}^L (\mathcal{E}^{\text{Sob}}_{d, k})$, via the formulas \eqref{eq:definition-critical-radius} and \eqref{eq:linear-risk-closed-form} of Theorem~\ref{thm:linear-risk-formula}, respectively.
  A direct application of Theorem~\ref{thm:risk-linear-2-nonlinear} concludes the proof.
  The details are provided in Appendix~\ref{sec:proof-pinsker-sob-order-two}.
\end{proof}

\noindent
In summary, 
Theorems~\ref{thm:Metric-Entropy-Sobolev} and \ref{thm:pinsker-sobolev-order-two} 
improve on the best-known characterizations of the metric entropy and minimax risk for the Sobolev ellipsoid.
The proofs rely on the analysis of the corresponding type-$\typevar$ integrals.
One could reach similar results via the analysis of the semi-axis-counting function \eqref{eq:ecv-Sobb-order-2}, 
but only at the cost of imposing the additional assumptions  \eqref{eq:billiard-condition}.
This illustrates the advantage of working with averaged,  rather than pointwise, characterizations.

\bibliography{main}

\newpage 
\appendix

\section{Typical Asymptotics}\label{sec:table}

Theorems~\ref{thm:metric-entropy-integral-form} and \ref{thm:bias-variance-decomp} allow us to derive the asymptotics of the metric entropy and the minimax risk given that of the semi-axis-counting function.
The next table compiles the results for typical examples.
We fix $d \in \Ns$, $c, c_0>0$, $\alpha>0$, and $\beta\in\R$; the limits are taken either as $\varepsilon\to 0$, as $n\to\infty$, or as $\sigma\to 0$.

{\fontsize{9.5}{11.5}\selectfont
\begin{empheq}[box=\fcolorbox{eqframe}{white}]{alignat*=4}
  & (i) \quad  &&\ecv(\varepsilon) \to d 
  && \quadtext{is equivalent to} 
  && \semax \text{ having only $d$ non-zero elements,}\\
  & && && \quadtext{and to} 
  && H\left(\varepsilon; \ellip\right) \sim d\ln\left(\varepsilon^{-1}\right), \\ 
  & && && \quadtext{and to} 
  && R_{\sigma} \left(\ellip\right) \sim d \, \sigma^2; \\[.5cm]
  & (ii) \quad  && \ecv(\varepsilon) \sim c \ln\left(\varepsilon^{-1}\right) && \quadtext{is equivalent to} 
  && \ln\left(\mu_n\right) \sim - \frac{n}{c},\\
  & && && \quadtext{and implies} 
  && H\left(\varepsilon; \ellip\right) \sim \frac{c}{2}\ln^2\left(\varepsilon^{-1}\right), \\ 
  & && && \quadtext{and} 
  && R_{\sigma} \left(\ellip\right) \sim c\, \sigma^2\ln\left(\sigma^{-1}\right); \\[.5cm]
  & (iii) \quad  && \ecv(\varepsilon) \sim 
  \frac{c\ln\left(\varepsilon^{-1}\right)}{\ln^{(2)}\left(\varepsilon^{-1}\right)} 
  && \quadtext{is equivalent to} 
    && \ln\left(\mu_n\right) \sim - \frac{n \ln(n)}{c},\\
  & && && \quadtext{and implies} 
  && H\left(\varepsilon; \ellip\right) \sim 
  \frac{c\ln^2\left(\varepsilon^{-1}\right)}{2\ln^{(2)}\left(\varepsilon^{-1}\right)}, \\ 
  & && && \quadtext{and} 
  && R_{\sigma} \left(\ellip\right) \sim \frac{c\sigma^2\ln\left(\sigma^{-1}\right)}{\ln^{(2)}\left(\sigma^{-1}\right)}; \\[.5cm]
  & (iv) \quad  && \ecv(\varepsilon) \sim 
  c\, \varepsilon^{-\alpha}
  && \quadtext{is equivalent to} 
  && \mu_n \sim c^{\frac{1}{\alpha}}n^{-\frac{1}{\alpha}},\\
  & && && \quadtext{and to} 
  && H\left(\varepsilon; \ellip\right) \sim 
  \frac{c\, \varepsilon^{-\alpha}}{\alpha}, \\ 
  & && && \quadtext{and to} 
  && R_{\sigma} \left(\ellip\right) \sim 
  \frac{\alpha+2}{\alpha}\left(\frac{c\, \alpha\, \sigma^2}{(\alpha+1)(\alpha+2)}\right)^{\frac{2}{\alpha+2}}; \\[.5cm]
  & (v) \quad  && \ecv(\varepsilon) \sim 
  c\, \varepsilon^{-\alpha} \left(\ln\left(\varepsilon^{-1}\right)\right)^\beta
  && \quadtext{is equivalent to} 
  && \mu_n \sim \left(\frac{c \,  \left(\ln(n)\right)^{\beta}}{\alpha^{\beta}n}\right)^{\frac{1}{\alpha}},\\
  & && && \quadtext{and to} 
  && H\left(\varepsilon; \ellip\right) \sim 
  \frac{c\, \varepsilon^{-\alpha}}{\alpha}\left(\ln\left(\varepsilon^{-1}\right)\right)^\beta, \\ 
  & && && \quadtext{and to} 
  && R_{\sigma} \left(\ellip\right) \sim 
  \frac{\alpha+2}{\alpha }
  \left(\frac{c\,\alpha\,\sigma^2\left(\ln\left(\sigma^{-\frac{2}{\alpha+2}}\right)\right)^{\beta}}{(\alpha+1)(\alpha+2)}\right)^\frac{2}{\alpha+2}; \\[.5cm]
  & (vi) \quad  && \ecv(\varepsilon) \sim 
  c_0\exp\left(c\, \varepsilon^{-\alpha}\right)
  && \quadtext{is equivalent to} 
  && \exp\left(c\, \mu_n^{-\alpha}\right) \sim c_0^{-1}\, n,\\
  & && && \quadtext{and implies} 
  && H\left(\varepsilon; \ellip\right) \sim 
  \frac{c_0}{c\, \alpha}  \varepsilon^{\alpha}\exp(c\, \varepsilon^{-\alpha}), \\[.25cm]
  & && && \quadtext{which implies} 
  && R_{\sigma} \left(\ellip\right) \sim 
    \left(\frac{c}{\ln\left(\sigma^{-2}\right)}\right)^{\frac{2}{\alpha}}.
\end{empheq}
}

\noindent 
All the examples covered in the table above satisfy $\rcb$---with $b=0$ for items (i)--(iii),
$b=\alpha$ for items (iv)--(v),
and $b=\infty$
for item (vi).

Item (i) concerns finite-dimensional ellipsoids. 
Their metric entropy and minimax risk---and, more generally, those of arbitrary finite-dimensional compact sets---are classical; see, e.g., \cite[Chapter~5]{wainwrightHighDimensionalStatistics2019} and \cite[Chapter~2]{johnstone2019estimation}.
In this regime, $H(\varepsilon; \ellip) \sim d\ln(\varepsilon^{-1})$, as $\varepsilon\to 0$, heuristically means that $\ellip$ has $d$ degrees of freedom.
This heuristic was used in \cite[Theorem~5]{ALLARD2025101762} to derive a metric-entropy version of the sampling theorem by counting  degrees of freedom per unit of time for band-limited functions.

Ellipsoids with semi-axes as in items (ii) and (iii) (and, more generally, $\ell^p$-ellipsoids in $\ell^q(\Ns)$ with exponentially decaying semi-axes) typically model classes of analytic functions.
Such function classes include periodic functions analytic on a strip, functions bounded on a disk, and functions of exponential type.
Sharp asymptotics for their metric entropy were obtained in \cite{firstpaper}, and the leading term of their minimax risk up to a multiplicative constant in \cite{ibrag1998,ibragimov2013statistical}.

The regularly varying cases (iv)--(v) are the most extensively studied.
For (iv), both metric entropy \cite{donohoCountingBitsKolmogorov2000,luschgySharpAsymptoticsKolmogorov2004,ALLARD2025101762,secondpaper,grafSharpAsymptoticsMetric2004} and minimax risk \cite{pinsker1980optimal,nussbaum1999minimax} are sharply characterized.
This setting arises when considering unit balls in smoothness classes (Sobolev, H\"older, Besov), 
where the exponent is typically $\alpha=d/k$ with $k$ the smoothness index and $d$ the dimension of the domain.
In particular, both quantities depend exponentially on the dimension of the domain, a phenomenon known as the `curse of dimensionality'.
In this case, asymptotics on the semi-axis-counting function can be obtained via the Weyl law (see Section~\ref{sec:appli-Sobolev}) or, for general pseudo-differential operators, via \cite[Theorem~30.1]{shubinInvitationPartialDifferential2020}; see \cite{psidopaper} for an application.

Ellipsoids falling in the regime (vi) are less common.
An example is provided by the polynomially Lipschitz fading-memory (PLFM) systems introduced in \cite{pan2025metricentropylimitsapproximationnonlinear}, whose metric entropy grows exponentially.

\section{Proofs of Main Results}\label{sec:proofs-main-results}

\subsection{Proof of Theorem~\ref{thm:metric-entropy-integral-form}}\label{sec:proof-metric-entropy-integral-form}

\paragraph{A series formula for $\intk{1}$.}
We first prove that 
\begin{equation}\label{eq:sum-formua-I1}
  \intk{1} (\varepsilon)
  = \sum_{n \in \Ns} \ln_+ \left(\frac{\mu_n}{\varepsilon}\right),
  \quad \text{for all } \varepsilon >0,
\end{equation}
where we recall the notation $ \ln_+(x)=\max\{0, \ln(x)\}$, for all $x\geq 0$.
To this end, let us fix $A>0$ large enough to have $\ecv (A)=0$.
Such an $A$ is guaranteed to exist thanks to the assumption $\mu_n\to 0$ as $n\to\infty$.
The integral $\intk{1}$ can then be rewritten as the Stieltjes integral
\begin{equation}\label{eq:entropy-stieltjes-decomp}
  \intk{1} (\varepsilon)
  = \int_{\varepsilon}^\infty \frac{\ecv (\dumvar)}{\dumvar} \diff \dumvar
  = \int_{\varepsilon}^A \frac{\ecv (\dumvar)}{\dumvar} \diff \dumvar
  = \int_{\varepsilon}^A \ecv (\dumvar) \diff \ln (\dumvar),
  \quad \text{for all } \varepsilon >0.
\end{equation}
Upon application of the integration-by-parts formula for the Stieltjes integral (see, e.g., \cite[Theorem~7.6]{apostol1974analysis}), we obtain 
\begin{equation}\label{eq:ipp-stieltjes-entropy}
  \int_{\varepsilon}^A \ecv (\dumvar) \diff \ln (\dumvar)
  =\ecv (A)\ln(A) - \ecv (\varepsilon) \ln(\varepsilon) - \int_{\varepsilon}^A \ln (\dumvar) \diff \ecv  (\dumvar),
  \quad \text{for all } \varepsilon >0.
\end{equation}
Now, recall that $A$ has been chosen such that $\ecv (A)=0$, so that the first term in the right-hand side of \eqref{eq:ipp-stieltjes-entropy} vanishes.
Furthermore, since $\ecv $ is a step function with negative jumps located at each $\mu_n$, 
we can reduce the Stieltjes integral appearing in the right-hand side of \eqref{eq:ipp-stieltjes-entropy} to a finite sum (see, e.g., \cite[Theorem~7.11]{apostol1974analysis}) according to
\begin{equation}\label{eq:stiltjes-to-sum}
  - \int_{\varepsilon}^A \ln (\dumvar) \diff \ecv  (\dumvar)
  = \sum_{\mu_n\in [\varepsilon, A)} \ln \left(\mu_n\right),
  \quad \text{for all } \varepsilon >0.
\end{equation}
Note that, with our choice of $A$, the condition $\mu_n<A$ is automatically satisfied for all $n\in\Ns$.
Therefore, combining \eqref{eq:entropy-stieltjes-decomp}, \eqref{eq:ipp-stieltjes-entropy}, and \eqref{eq:stiltjes-to-sum}, and using that there are, by definition of $\ecv $, precisely $\ecv (\varepsilon)$ semi-axes satisfying $\mu_n \geq \varepsilon$,
we get 
\begin{equation*}
  \intk{1}(\varepsilon)
  = \sum_{\mu_n \geq \varepsilon} \ln \left(\mu_n\right)- \ecv (\varepsilon) \ln(\varepsilon)
  = \sum_{n \in \Ns} \ln_+ \left(\frac{\mu_n}{\varepsilon}\right),
  \quad \text{for all } \varepsilon >0,
\end{equation*}
which proves \eqref{eq:sum-formua-I1}.

In what follows, we may assume that at least two semi-axes are non-zero.
If all semi-axes are zero, $H(\varepsilon; \mathcal{E}_\mu) = \intk{1} (\varepsilon) = 0$, for all $\varepsilon>0$.
If there is a single non-zero semi-axis $\mu_*$,
\eqref{eq:sum-formua-I1} reduces to $\intk{1} (\varepsilon) =  \ln_+ (\mu_*/\varepsilon)$
and one easily verifies that $H(\varepsilon; \mathcal{E}_\mu) = \ln(\lceil\mu_*/\varepsilon\rceil)$, for all $\varepsilon>0$.
In both cases \eqref{eq:main-result-ME} and \eqref{eq:asymp-equivalence-ME-I1-first-order} hold trivially.

\paragraph{A lower bound on metric entropy.}
We next prove the following lower bound on metric entropy
\begin{equation}\label{eq:lower-bound-integral-ME}
  H\left(\varepsilon; \mathcal{E}_\mu\right)
  \geq \sum_{n \in \Ns} \ln_+ \left(\frac{\mu_n}{\varepsilon}\right),
  \quad \text{for all } \varepsilon > 0.
\end{equation}
To this end, let us fix $\varepsilon>0$ and
introduce the ellipsoid $ \bar{\mathcal{E}}_\mu$ obtained from $\mathcal{E}_\mu$ by retaining only the semi-axes greater than or equal to $\varepsilon$ and setting the remaining ones to zero.
$ \bar{\mathcal{E}}_\mu$ is therefore a finite-dimensional ellipsoid 
of dimension $\mathfrak{d}\coloneqq \ecv (\varepsilon)$.
By construction, we have the inclusion $ \bar{\mathcal{E}}_\mu \subseteq \mathcal{E}_\mu$, which, in terms of metric entropy, translates into 
\begin{equation}\label{eq:finite-infinite-dim-ME-comparison}
  H(\varepsilon; \mathcal{E}_\mu)
  \geq H(\varepsilon; \bar{\mathcal{E}}_\mu).
\end{equation}
We have reduced the problem of lower-bounding the metric entropy of $\mathcal{E}_\mu$ to that of lower-bounding the metric entropy of a finite-dimensional ellipsoid.
We are thus in position to use a volume argument---namely \cite[Lemma~5.7]{wainwrightHighDimensionalStatistics2019}---to get
\begin{equation}\label{eq:finite-dim-cov-volume-argument}
  N(\varepsilon; \bar{\mathcal{E}}_\mu)
  \geq \varepsilon^{-\mathfrak{d}} \frac{\vol(\bar{\mathcal{E}}_\mu)}{\omega_{\mathfrak{d}}}
  = \varepsilon^{-\ecv (\varepsilon)} \prod_{\mu_n \geq \varepsilon} \mu_n,
\end{equation}
where we recall that $\omega_{\mathfrak{d}}$ denotes the volume of the $\mathfrak{d}$-dimensional Euclidean ball, and we have used the standard formula for the volume of a finite-dimensional ellipsoid (for more details see, e.g., the proof of \cite[Theorem~4]{firstpaper}).
Taking logarithm on both sides of \eqref{eq:finite-dim-cov-volume-argument} and using \eqref{eq:finite-infinite-dim-ME-comparison} concludes the proof of \eqref{eq:lower-bound-integral-ME}.

\paragraph{Preparation.}
Before turning to the upper bound,
we fix $\varepsilon>0$ small enough for $\ecv (\varepsilon) \geq 2$ to hold (which is possible since we assumed that at least two semi-axes are non-zero), and introduce
\begin{equation}\label{eq:define-auxiliary-eps}
   \lowesteps \coloneqq \sup \left\{x>0 \mid x \left(1+\ecv (x)^{-1}\right) < \varepsilon\right\}.
\end{equation}
We readily check from its definition in \eqref{eq:definition-ecf} that $\ecv $ is a left-continuous function, that is, 
$\lim_{\eta \downarrow 0} \ecv (x-\eta) = \ecv (x)$, for all $x>0$.
In particular, using the left-continuity of $\ecv $ in \eqref{eq:define-auxiliary-eps} guarantees that 
\begin{equation}\label{eq:first-inequality-epsilons}
  \lowesteps \left(1+\ecv (\lowesteps)^{-1}\right) \leq  \varepsilon,
  \quadtext{and, therefore, that}
  \lowesteps \leq  \frac{\varepsilon}{1+\ecv ( \lowesteps)^{-1}} \eqcolon  \mediumeps.
\end{equation}
We further set $d\coloneqq \ecv (\mediumeps)$, 
fix an integer $k \leq d$,
and define 
\begin{equation}\label{eq:definition-nb-block}
   \mediumepsgam 
   \coloneqq \mediumeps \left(1+\frac{\sqrt{k+1}}{d^{\gamma}}\right),
\quadtext{where}
  \gamma \coloneqq \frac{\ln\left(\ecv ( \lowesteps)\right)}{\ln(d)}+1.
\end{equation}
Using that $\ecv$ is non-increasing, we get $d\geq \ecv(\varepsilon)\geq 2$ and  $\gamma \geq 2$.
Additionally, using the definition \eqref{eq:definition-nb-block} of $\gamma$, we have
\begin{equation}\label{eq:bound-epsgamm-eps}
  \frac{\sqrt{k+1}}{d^{\gamma}}
  = \frac{\sqrt{k+1}}{d\ecv ( \lowesteps)}
  < \frac{1}{\ecv ( \lowesteps)},
\end{equation}
where we used both $k\leq d$ and $d\geq 2$ in the last step.
Plugging the definition \eqref{eq:first-inequality-epsilons} of $\mediumeps$ in the definition \eqref{eq:definition-nb-block} of $\mediumepsgam$ and applying the bound obtained in \eqref{eq:bound-epsgamm-eps} yields 
\begin{equation}\label{eq:bound-epsgamma-vs-epsmid}
  \mediumepsgam 
  = \varepsilon \left(\frac{1+d^{-\gamma}\sqrt{k+1}}{1+\ecv ( \lowesteps)^{-1}}\right)
  < \varepsilon.
\end{equation}
Combining \eqref{eq:first-inequality-epsilons}, \eqref{eq:definition-nb-block}, and \eqref{eq:bound-epsgamma-vs-epsmid},
we get $\lowesteps \leq  \mediumeps<\mediumepsgam< \varepsilon$,
which, using that the semi-axis-counting function $\ecv $ is non-increasing, turns into $\ecv ( \lowesteps) \geq \ecv ( \mediumeps) \geq \ecv ( \mediumepsgam) \geq \ecv ( \varepsilon)$.

In fact, it will be useful to additionally prove that $\ecv ( \lowesteps)$ and $\ecv ( \varepsilon)$ are asymptotically equivalent, in the sense that 
\begin{equation}\label{eq:asymp-equivalence-allecv}
  \ecv ( \lowesteps) 
  \sim \ecv ( \mediumeps) 
  \sim \ecv ( \mediumepsgam) 
  \sim \ecv ( \varepsilon), 
  \quad \text{as }\varepsilon \to 0,
\end{equation}
which we do next.
Note that it is sufficient to establish $ \ecv \left(\lowesteps\right) \sim \ecv (\varepsilon)$, as $\varepsilon \to 0$.
To this end,
we first come back to the definition of $\lowesteps$ in \eqref{eq:define-auxiliary-eps}, which gives
\begin{equation}\label{eq:def-eps-eta}
  \lowesteps_\eta \coloneqq \left(\lowesteps+\eta\right)\left(1+\ecv (\lowesteps+\eta)^{-1}\right) \geq \varepsilon,
  \quad \text{for all } \eta >0.
\end{equation}
We next take $f$ as in (RC) and apply Lemma~\ref{lem:regularity-evf-o-term}  with $x = \lowesteps+\eta(\lowesteps)$ (see Appendix~\ref{sec:proof-lem-regularity-evf-o-term}), to get 
\begin{equation*}
  f\left(\lowesteps_{\eta(\lowesteps)}\right) 
  \sim f\left(\lowesteps+\eta(\lowesteps)\right),
  \quad \text{as } \lowesteps \to 0, 
\end{equation*}
for any map $\eta \colon \Rp\to\Rp$ satisfying $\eta(\lowesteps) \to 0$, as $\lowesteps\to 0$.
By continuity of $f$, there exists such a map---say $\bar{\eta}$---satisfying 
\begin{equation*}
  f\left(\lowesteps+\bar{\eta}(\lowesteps)\right) 
  \sim f\left(\lowesteps\right),
  \quadtext{and, therefore,}
  f \left(\lowesteps_{\bar{\eta}(\lowesteps)}\right) 
  \sim f \left(\lowesteps\right),
  \quad \text{as } \lowesteps \to 0.
\end{equation*}
Using that $f$ is asymptotically equivalent to $\ecv$, we then obtain
\begin{equation}\label{eq:asymp-equiv-shifted-ecv}
  \ecv \left(\lowesteps_{\bar{\eta}(\lowesteps)}\right) 
  \sim \ecv \left(\lowesteps\right),
  \quad \text{as } \lowesteps \to 0.
\end{equation}
Furthermore, we get from \eqref{eq:first-inequality-epsilons} and \eqref{eq:def-eps-eta} the inequalities $\lowesteps < \varepsilon \leq \lowesteps_{\bar{\eta}(\lowesteps)}$, which, owing to $\ecv$ being non-increasing, yield
\begin{equation}\label{eq:relation-ecv-diverse-eps}
  \ecv \left(\lowesteps_{\bar{\eta}(\lowesteps)}\right) 
  \leq \ecv (\varepsilon)
  \leq \ecv \left(\lowesteps\right),
  \quad \text{for all } \varepsilon > 0.
\end{equation}
Combining the asymptotic equivalence in \eqref{eq:asymp-equiv-shifted-ecv} with the relation obtained in \eqref{eq:relation-ecv-diverse-eps}, 
we get $ \ecv \left(\lowesteps\right) \sim \ecv (\varepsilon)$, as $\varepsilon \to 0$, which establishes \eqref{eq:asymp-equivalence-allecv}.

Finally, we relate the integrals $\intk{1}(\mediumeps)$ and $\intk{1}(\varepsilon)$.
To this end, we first bound their difference according to
\begin{equation}\label{eq:bound-diff-integrals-ME}
  \intk{1}(\mediumeps) - \intk{1}(\varepsilon)
  = \int_{\mediumeps}^\infty\frac{\ecv (\dumvar)}{\dumvar} \diff \dumvar - \int_{\varepsilon}^\infty\frac{\ecv (\dumvar)}{\dumvar} \diff \dumvar
  = \int_{\mediumeps}^\varepsilon\frac{\ecv (\dumvar)}{\dumvar} \diff \dumvar
  \leq \left(\varepsilon-\mediumeps\right) \frac{\ecv (\mediumeps)}{\mediumeps},
\end{equation}
where the last step uses that $\ecv $ is non-increasing.
Recalling the definition of $\mediumeps$ in \eqref{eq:first-inequality-epsilons},
we see that the right-hand side in \eqref{eq:bound-diff-integrals-ME} satisfies
\begin{equation}\label{eq:bound-diff-integrals-ME2}
  \left(\varepsilon-\mediumeps\right) \frac{\ecv (\mediumeps)}{\mediumeps}
  = \left(1- \frac{1}{1+\ecv ( \lowesteps)^{-1}}\right) \left(1+\ecv ( \lowesteps)^{-1}\right)
  \ecv (\mediumeps)
  = \frac{\ecv (\mediumeps)}{\ecv ( \lowesteps)}.
\end{equation}
We have seen that $\ecv (\mediumeps)\leq \ecv (\lowesteps)$,
so that, upon combining \eqref{eq:bound-diff-integrals-ME} and \eqref{eq:bound-diff-integrals-ME2}, we get $  \intk{1}(\mediumeps) - \intk{1}(\varepsilon) \leq 1 $. 
Consequently, in the limit $\varepsilon\to 0$, we get 
\begin{equation}\label{eq:transfer-sum-tildeeps-Ieps}
    \intk{1}(\mediumeps) 
    = \intk{1}(\varepsilon) + O_{\varepsilon\to 0}(1).
\end{equation}

\paragraph{Two upper bounds on metric entropy.}
We now turn to the derivation of a corresponding upper bound and split the task into proving both
\begin{equation}\label{eq:upper-bound-integral-ME1}
  H(\varepsilon; \mathcal{E}_\mu)
  \leq \sum_{n \in \Ns} \ln_+ \left(\frac{\mu_n}{\varepsilon}\right) + O_{\varepsilon\to 0}\left(\ecv (\varepsilon)\right)
\end{equation}
and
\begin{equation}\label{eq:upper-bound-integral-ME2}
  H(\varepsilon; \mathcal{E}_\mu)
  \leq \sum_{n \in \Ns} \ln_+ \left(\frac{\mu_n}{\varepsilon}\right) + O_{\varepsilon\to 0}\left( \sqrt{\ecv (\varepsilon)\ln\left(\ecv (\varepsilon)\right)\ln\left(\varepsilon^{-1}\right)}\right).
\end{equation}
The desired relation \eqref{eq:main-result-ME} then follows by combining the best of the upper bounds \eqref{eq:upper-bound-integral-ME1}--\eqref{eq:upper-bound-integral-ME2}---that is, the one with smallest $O$-term---with the lower bound
\eqref{eq:lower-bound-integral-ME} and the formula
\eqref{eq:sum-formua-I1}.
In view of \eqref{eq:transfer-sum-tildeeps-Ieps}, 
in order to establish \eqref{eq:upper-bound-integral-ME1} and \eqref{eq:upper-bound-integral-ME2}, it is sufficient to prove respectively
\begin{equation}\label{eq:partial-result-OM}
  H(\varepsilon; \ellip)
  \leq \sum_{n \in \Ns} \ln_+ \left(\frac{\mu_n}{\mediumeps}\right) 
  + O_{\varepsilon\to 0}\left(\ecv (\varepsilon)\right)
\end{equation}
and
\begin{equation}\label{eq:upper-bound-integral-ME2-OM}
  H(\varepsilon; \mathcal{E}_\mu)
  \leq \sum_{n \in \Ns} \ln_+ \left(\frac{\mu_n}{\mediumeps}\right) + O_{\varepsilon\to 0}\left( \sqrt{\ecv (\varepsilon)\ln\left(\ecv (\varepsilon)\right)\ln\left(\varepsilon^{-1}\right)}\right).
\end{equation}

\paragraph{Proof of \eqref{eq:partial-result-OM}.}
In this paragraph, we set $k=1$ and introduce the ellipsoid $\tilde \ellip$ obtained from $\ellip$ by retaining only the semi-axes greater or equal to $\mediumeps$ and setting the remaining ones to zero---without loss of generality, we may assume that $\tilde \ellip$ retains the first $d$ components of $\ellip$.
$\tilde \ellip$ is therefore a finite-dimensional ellipsoid, of dimension $d$.
Let us further define the finite lattice 
\begin{equation}\label{eq:definition-lattice-two-blocks}
  \Omega
  \coloneqq 
  \left\{\omega \in \mathbb{R}_+^{2} \mid \| 
  \omega \|_{2} \leq 1+ d^{-\gamma}\sqrt{2} \text{ with } d^{\,2\gamma} \omega_1^2 \in \mathbb{N} \text{ and } d^{\,2\gamma} \omega_2^2 \in \mathbb{N} \right\},
\end{equation}
whose cardinality can be bounded according to 
\begin{equation}\label{eq:cardinality-lattice}
  \lvert\Omega\rvert \leq K d^{4\gamma},
  \quad \text{for some } K>0.
\end{equation}
In view of application of \cite[Lemma~6]{secondpaper},
we need to verify that 
\begin{equation}\label{eq:inclusion-ellipsoid-product-two}
\ellip
\subseteq \bigcup_{\omega\in\Omega} \left( \omega_1 \tilde \ellip \times \omega_{2}    \mathcal{B} ( \mediumeps)\right),
\end{equation}
where $\mathcal{B} ( \mediumeps)$ denotes the ball in $\ell^2(\Ns)$ centered at zero and of radius $\mediumeps$.
In order to prove \eqref{eq:inclusion-ellipsoid-product-two}, let us take an arbitrary $x\in \ellip$, and introduce 
\begin{equation}\label{eq:definition-points-latice}
  \omega_1 \coloneqq d^{-\gamma} \left\lceil\left\| \left\{\mu_j^{-1} d^{\gamma} x_j \right\}_{j=1}^d \right\|_{2}^2 \right\rceil^{1/2}
  \quadtext{and}
  \omega_2 \coloneqq d^{-\gamma} \left\lceil\left\| \left\{\mediumeps^{-1} d^{\gamma} x_j \right\}_{j=d+1}^\infty \right\|_{2}^2 \right\rceil^{1/2}.
\end{equation}
We immediately get from \eqref{eq:definition-points-latice} the inequalities
\begin{equation*}
  \left\| \left\{\mu_j^{-1}  x_j \right\}_{j=1}^d \right\|_{2} \leq \omega_1
   \quadtext{and}
   \left\| \left\{\mediumeps^{-1}  x_j \right\}_{j=d+1}^\infty \right\|_{2} \leq \omega_2,
\end{equation*}
from which we deduce $x\in  \omega_1 \tilde \ellip \times \omega_{2}    \mathcal{B} ( \mediumeps)$.
Moreover, $(\omega_1,\omega_2)\in\Omega$ since
\begin{align*}
 \| \omega \|_{2}^2
  = \omega_1^2+\omega_2^2 
  &= d^{-2\gamma} \left\lceil\left\| \left\{\mu_j^{-1} d^{\gamma} x_j \right\}_{j=1}^d \right\|_{2}^2 \right\rceil 
  + d^{-2\gamma} \left\lceil\left\| \left\{\mediumeps^{-1} d^{\gamma} x_j \right\}_{j=d+1}^\infty \right\|_{2}^2 \right\rceil\\
  &\leq   \left\| \left\{\mu_j^{-1}  x_j \right\}_{j=1}^d \right\|_{2}^2 
  + \left\| \left\{\mediumeps^{-1}  x_j \right\}_{j=d+1}^\infty \right\|_{2}^2 + 2d^{-2\gamma}
  \\
  &\leq   \left\| \left\{\mu_j^{-1}  x_j \right\}_{j=1}^\infty \right\|_{2}^2 
   + 2d^{-2\gamma}
   \leq 1 + 2d^{-2\gamma} \leq \left(1 + d^{-\gamma}\sqrt{2}\right)^2,
\end{align*}
where the second inequality uses that 
$\mediumeps \geq \mu_j$, for all $j\geq d+1$, and the penultimate bound comes from the fact that $x$ belongs to $\ellip$.
(We used the convention $\mu_j^{-1}  x_j=0$ if $\mu_j = x_j=0$.)
We have thus proven that, for all  $x\in \ellip$, there exists $(\omega_1,\omega_2)\in\Omega$ such that $ x\in  \omega_1 \tilde \ellip \times \omega_{2}    \mathcal{B} ( \mediumeps)$,
which establishes \eqref{eq:inclusion-ellipsoid-product-two}.
The inclusion  \eqref{eq:inclusion-ellipsoid-product-two} ensures that the hypotheses of \cite[Lemma~6]{secondpaper} are satisfied
and, upon its application with $p=q=2$ and $\rho_1=\mediumeps$, 
we get 
\begin{equation}\label{eq:finite-infinite-dim-ME-comparison-UB-bbis}
  H\left(\max_{\omega\in\Omega} [ (\omega_1\mediumeps)^2 + (\omega_2\mediumeps)^2]^{1/2}; \mathcal{E}_\mu\right)
  \leq H\left(\mediumeps; \tilde\ellip\right)
  + \ln\left(\lvert\Omega\rvert\right).
\end{equation}
We further remark from \eqref{eq:definition-lattice-two-blocks} that the elements of $\Omega$ have bounded norm, implying 
\begin{equation*}
  \max_{\omega\in\Omega} [ (\omega_1\mediumeps)^2 + (\omega_2\mediumeps)^2]^{1/2} 
  = \mediumeps \max_{\omega\in\Omega} \|\omega\|_2 
  \leq \mediumeps \left(1+ d^{-\gamma}\sqrt{2}\right)
  = \mediumepsgam,
\end{equation*}
where the last step simply uses the definition \eqref{eq:definition-nb-block} of $\mediumepsgam$ with $k=1$.
Consequently, \eqref{eq:finite-infinite-dim-ME-comparison-UB-bbis} yields  
\begin{equation}\label{eq:finite-infinite-dim-ME-comparison-UB}
  H(\mediumepsgam; \mathcal{E}_\mu)
  \leq H\left(\mediumeps; \tilde\ellip\right)
  + \ln\left(\lvert\Omega\rvert\right),
\end{equation}
Next, we infer from $\varepsilon > \mediumepsgam$ and from \eqref{eq:cardinality-lattice} combined with the asymptotic equivalence $d=\ecv (\mediumeps)\sim \ecv (\varepsilon)$ as $\varepsilon\to 0$ obtained in \eqref{eq:asymp-equivalence-allecv}
the relations
\begin{equation}\label{eq:aux-inequalities-UB}
  H(\varepsilon; \mathcal{E}_\mu)
  \leq H(\mediumepsgam; \mathcal{E}_\mu)
  \quadtext{and}
  \ln\left(\lvert\Omega\rvert\right) 
  =  O_{\varepsilon\to 0}\left(\ln\left(\ecv (\varepsilon)\right)\right).
\end{equation}
Furthermore, upon application of a volume argument---see, e.g., \cite[Lemma~5.7]{wainwrightHighDimensionalStatistics2019}---we get an upper bound on the covering number of the $d$-dimensional ellipsoid $\tilde\ellip$ according to
\begin{equation}\label{eq:volume-argument-UB}
  N\left(\mediumeps; \tilde\ellip\right)
  \leq \left(K'\mediumeps\right)^{-d} \prod_{\mu_n \geq \mediumeps} \mu_n,
  \quad \text{for some } K'>0.
\end{equation}
Recalling that $d = \ecv (\mediumeps)$ and taking logarithm on both sides of \eqref{eq:volume-argument-UB} gives
\begin{equation}\label{eq:volume-argument-UB2}
  H\left(\mediumeps; \tilde\ellip\right)
  \leq \sum_{n \in \Ns} \ln_+ \left(\frac{\mu_n}{\mediumeps}\right) 
  + O_{\varepsilon\to 0}\left(\ecv (\mediumeps)\right)
  = \sum_{n \in \Ns} \ln_+ \left(\frac{\mu_n}{\mediumeps}\right) 
  + O_{\varepsilon\to 0}\left(\ecv (\varepsilon)\right),
\end{equation}
where we used once again $d=\ecv (\mediumeps)\sim \ecv (\varepsilon)$, 
as $\varepsilon\to 0$.
Putting \eqref{eq:finite-infinite-dim-ME-comparison-UB}, \eqref{eq:aux-inequalities-UB}, and \eqref{eq:volume-argument-UB2} together establishes \eqref{eq:partial-result-OM}.

\paragraph{Proof of \eqref{eq:upper-bound-integral-ME2-OM}.}
In view of \eqref{eq:partial-result-OM},
it is sufficient to establish 
\begin{equation*}
  H(\varepsilon_p; \mathcal{E}_\mu)
  \leq \sum_{n \in \Ns} \ln_+ \left(\frac{\mu_n}{\mediumeps_p}\right) + O_{p\to \infty}\left( \sqrt{\ecv (\varepsilon_p)\ln\left(\ecv (\varepsilon_p)\right)\ln\left(\varepsilon_p^{-1}\right)}\right)
\end{equation*}
for every sequence
$\{\varepsilon_p\}_{p\in\Ns}$  satisfying
$ \sqrt{\ecv (\varepsilon_p)\ln(\ecv (\varepsilon_p))\ln(\varepsilon_p^{-1})}
\leq \ecv(\varepsilon_p)$, for all $p\in\Ns$, and $\varepsilon_p\to 0$ as $p\to \infty$.
To simplify notation, 
we omit the index of the sequence,
i.e., we simply write $\varepsilon$ instead of $\varepsilon_p$, and $\varepsilon\to 0$ for the limit $p \to \infty$.
We thus assume $  \sqrt{\ecv (\varepsilon)\ln(\ecv (\varepsilon))\ln(\varepsilon^{-1})} \leq \ecv (\varepsilon)$, 
so that we must have $\ecv(\varepsilon)\to\infty$, as $\varepsilon\to 0$.
Recalling from \eqref{eq:asymp-equivalence-allecv} that $d = \ecv (\mediumeps) \sim \ecv (\varepsilon)$, as $\varepsilon\to 0$,
we also have $\sqrt{d\ln(d)\ln(\varepsilon^{-1})} = O_{\varepsilon\to0}(d)$.
In particular, setting
\begin{equation}\label{eq:definition-nb-block2}
    k \coloneqq \left\lfloor \sqrt{\frac{d \ln{\left(\varepsilon^{-1}\right)}}{\ln\left(d\right)}} \right\rfloor,
\end{equation}
we get $k= O_{\varepsilon\to 0}(d/\ln(d))$,
and we fix $\varepsilon$ small enough for $k\leq d/9$ to hold.
Next, we fix $s\in\Ns$---which will be specified shortly after---and introduce the non-increasing sequence of positive integers $\{d_j\}_{j\in\Ns}$, defined as $d_j \coloneqq \lceil d/k\rceil$, for $j < s$,
and $d_j \coloneqq \lfloor d/k\rfloor$, for $j\geq s$.
In particular, $d_j\geq 9$, for all $j\in\Ns$.
It will be convenient to introduce the notation $\bar d_0\coloneqq 0$ and  $\bar d_j \coloneqq d_1+\dots+d_j$, for all $j\in\Ns$.
Note that $s$ can be chosen such that $d=\bar d_k$.

There are exactly $d=\ecv (\mediumeps)$ semi-axes of $\ellip$ greater than or equal to $\mediumeps$; 
let us label their non-increasing rearrangement as $\tilde \mu_1 \geq \dots \geq \tilde \mu_d$ and set
\begin{equation}\label{eq:construction-axes-blocks}
  \bar{\mu}_j
  \coloneqq 
  \begin{cases}
    \tilde\mu_{\bar d_{j-1}+1}\left(1+d^{-1}\right), \quad &\text{if } j\in \{1, \dots, k\}, \text{ and } \\
    \mediumeps \quad & \text{otherwise.}
  \end{cases}
\end{equation}
In particular, the $\{\bar \mu_j\}_{j\in\Ns}$ are also ordered non-increasingly, and we have $\bar{\mu}_{k+1}\leq  \mediumeps < \bar{\mu}_k$.
Following \cite[Definition~10]{secondpaper},
we can now define the mixed ellipsoid $\mathcal{E}^{(m)}_{\bar \mu}$,
with parameters $p_1=p_2=2$, non-increasing semi-axes $\{\bar \mu_j\}_{j\in\Ns}$, and dimensions $\{d_j\}_{j\in\Ns}$.
Note that, by construction, there exists an isometry $\iota$ of $\ell^2(\Ns)$ satisfying $ \iota(\mathcal{E}_\mu)\subseteq \mathcal{E}^{(m)}_{\bar \mu}$.

We are now in position of applying \cite[Theorem~12]{secondpaper} to the mixed ellipsoid $\mathcal{E}^{(m)}_{\mu}$ with non-increasing semi-axes $\{\bar \mu_j\}_{j\in\Ns}$ and dimensions $\{d_j\}_{j\in\Ns}$,
which yields
\begin{equation}\label{eq:bound-block-decomp}
  N\left(\mediumepsgam ; \mathcal{E}^{(m)}_\mu\right)
  \leq  \kappa_k^d \frac{\prod_{j=1}^{k}\bar \mu_j^{d_j}}{\mediumepsgam^d},
\end{equation}
where the sequence $\{\kappa_k\}_{k\in\Ns}$ satisfies
\begin{equation}\label{eq:scaling-kappa-rogers}
  d \ln(\kappa_k)  
  = O_{\varepsilon\to 0}\left({\gamma k \ln(d)}\right)
  = O_{\varepsilon\to 0}\left({k \ln(d)}\right).
\end{equation}
We used in \eqref{eq:scaling-kappa-rogers} that $\gamma\to2$, as $\varepsilon \to 0$, 
which comes as a direct consequence of the  definition of $\gamma$ in \eqref{eq:definition-nb-block} and the asymptotic equivalence \eqref{eq:asymp-equivalence-allecv}.
Next, we combine three well-known properties of covering numbers---namely that they are non-increasing functions of the radius of the covering balls (covering with larger balls requires less covering balls), that they are non-decreasing functions with respect to set inclusion (covering a larger set requires more covering balls), and that they are invariant under isometries (covering numbers are determined by the metric structure)---with the relations $\mediumeps< \mediumepsgam< \varepsilon $ and $ \iota(\mathcal{E}_\mu)\subseteq \mathcal{E}^{(m)}_{\mu}$, 
to deduce that
\begin{equation}\label{eq:bound-block-decomp2}
  N\left(\varepsilon; \mathcal{E}_\mu\right)
  \leq N\left(\mediumepsgam ; \mathcal{E}^{(m)}_\mu\right),
  \text{ and, thus, from  \eqref{eq:bound-block-decomp}, that }
  N(\varepsilon; \mathcal{E}_\mu)
  \leq  \kappa_k^d \frac{\prod_{j=1}^{k}\bar \mu_j^{d_j}}{\mediumeps^d}.
\end{equation}
Moreover, by \eqref{eq:construction-axes-blocks} and the ordering $\tilde \mu_1 \geq \dots \geq \tilde \mu_d$, 
the semi-axis $\bar \mu_j$, 
for $j \in\{2, \dots, k\}$,
is always smaller than or equal to
$\tilde \mu_{i}(1+d^{-1})$, for all $i\in\{\bar d_{j-2} + 1, \dots, \bar d_{j-1}\}$, so that 
\begin{equation}\label{eq:bound-product-ordered-axes-bbis}
  \prod_{j=1}^{k}\left(\frac{\bar \mu_j}{\mediumeps}\right)^{d_j}
  \leq \left(1+d^{-1}\right)^d\left(\frac{\tilde \mu_1}{\mediumeps}\right)^{d_1} 
  \prod_{i=1}^{\bar d_{k-1}}\left(\frac{\tilde \mu_i}{\mediumeps}\right) \leq \left(1+d^{-1}\right)^d \left(\frac{\tilde \mu_1}{\mediumeps}\right)^{d_1} 
  \prod_{i=1}^{d}\left(\frac{\tilde \mu_i}{\mediumeps}\right), 
\end{equation}
where the first step uses that the $\{d_j\}_{j\in\Ns}$ are non-increasing,
and the second step is based on $\tilde \mu_i/\mediumeps\geq 1$,
for all $i\in\{\bar d_{k-1}+1, \dots, \bar d_k\}$.
Note that, by definition of $\tilde \mu_1, \dots, \tilde \mu_d$, the product in the right-most side of \eqref{eq:bound-product-ordered-axes-bbis} runs over all the semi-axes of $\ellip$ greater than or equal to $\mediumeps$.
Therefore, recalling that $d_1\leq d/k+1$ and that we can choose $A>0$ such that all the semi-axes $\{ \mu_n\}_{n\in\Ns}$ are upper-bounded by $A$, we get from \eqref{eq:bound-product-ordered-axes-bbis} that
\begin{equation}\label{eq:block-to-ellipsoid}
  \prod_{j=1}^{k}\left(\frac{\bar \mu_j}{\mediumeps}\right)^{d_j}
  \leq \left(1+d^{-1}\right)^d \left(A \mediumeps^{-1}\right)^{\frac{d}{k}+1} 
  \prod_{\mu_n \geq \mediumeps}\left(\frac{\mu_n}{\mediumeps}\right).
\end{equation}
Taking logarithm on both sides of \eqref{eq:bound-block-decomp2} and using \eqref{eq:block-to-ellipsoid}, 
we can upper-bound the metric entropy of $\mathcal{E}_\mu$ according to
\begin{equation}\label{eq:bound-ME-integral-error}
  H\left(\varepsilon; \mathcal{E}_\mu\right)
  \leq \sum_{\mu_n \geq \mediumeps} \ln\left(\frac{\mu_n}{\mediumeps}\right)
  + d \ln(\kappa_k) 
  + \left(\frac{d}{k}+1\right)\ln\left(A\mediumeps^{-1}\right)
  + d \ln\left(1+d^{-1}\right).
\end{equation}

We next observe that $d\ln(1+1/d) = O_{\varepsilon\to 0}(1)$,
and, from the definition of $\mediumeps$ in \eqref{eq:first-inequality-epsilons}, that $\ln(A\mediumeps^{-1})\sim \ln(\varepsilon^{-1})$, as $\varepsilon\to 0$.
Moreover, 
we can control the  error terms in \eqref{eq:bound-ME-integral-error} 
by combining the asymptotic scaling  \eqref{eq:scaling-kappa-rogers} with the definition of $k$ in \eqref{eq:definition-nb-block2}
according to
\begin{equation}\label{eq:bound-error-ME-UB}
  d \ln(\kappa_k)  
  = O_{\varepsilon\to 0}\left(\sqrt{d \ln(d) \ln\left(\varepsilon^{-1}\right)}\right),
  \text{ and }
  \frac{d}{k}\ln\left(\varepsilon^{-1}\right)
  = O_{\varepsilon\to 0}\left( \sqrt{d \ln(d) \ln\left(\varepsilon^{-1}\right)}\right).
\end{equation}
Putting the asymptotic bounds obtained in \eqref{eq:bound-error-ME-UB} together, and recalling from \eqref{eq:asymp-equivalence-allecv} that $d = \ecv (\mediumeps) \sim \ecv (\varepsilon)$, as $\varepsilon\to 0$, 
we get 
\begin{equation*}
    d \ln(\kappa_k) 
  + \left(\frac{d}{k}+1\right)\ln\left(A\mediumeps^{-1}\right)
  + d \ln\left(1+d^{-1}\right)
  = O_{\varepsilon\to 0}\left(\sqrt{\ecv (\varepsilon)\ln\left(\ecv (\varepsilon)\right)\ln\left(\varepsilon^{-1}\right)}\right),
\end{equation*}
which, with \eqref{eq:bound-ME-integral-error} yields \eqref{eq:upper-bound-integral-ME2-OM}.

\paragraph{Proof that \eqref{eq:main-result-ME} implies \eqref{eq:asymp-equivalence-ME-I1-first-order}.}
In order to prove that \eqref{eq:main-result-ME} implies \eqref{eq:asymp-equivalence-ME-I1-first-order}, 
we first take $f$ as in (RC) and introduce the integral 
\begin{equation*}
  \tilde I_1(\varepsilon)
  \coloneqq  \int_{\varepsilon}^{\infty} \frac{f(\dumvar)}{\dumvar} \diff \dumvar,
  \quad \text{for all } \varepsilon > 0.
\end{equation*}
The integral $\tilde I_1$ is continuously differentiable, 
with derivative given by
\begin{equation}\label{eq:derivative-regularized-integral}
   \tilde I_1'(\varepsilon)
   = -\frac{f(\varepsilon)}{\varepsilon}, 
   \quad \text{for all } \varepsilon > 0,
  \quadtext{and it satisfies}
  \tilde I_1(\varepsilon) \sim \intk{1}(\varepsilon),
  \quad \text{as } \varepsilon \to 0,
\end{equation}
see Lemma~\ref{lem:abelian-integration}.
Since, as $\varepsilon\to 0$, we have the equivalences $f(\varepsilon) \sim \ecv(\varepsilon)$ from (RC) and $\tilde I_1(\varepsilon) \sim \intk{1}(\varepsilon)$ from \eqref{eq:derivative-regularized-integral},
in order to show that the $O_{\varepsilon\to 0}(\cdot)$-term in \eqref{eq:main-result-ME} is negligible compared to $\intk{1}(\varepsilon)$,
it is enough to prove that 
\begin{equation}\label{eq:limit-to-prove-I1-dominates}
  \lim_{\varepsilon\to 0}
  \frac{\min\left\{f(\varepsilon), \sqrt{f (\varepsilon)\ln\left(f (\varepsilon)\right)\ln\left(\varepsilon^{-1}\right)}\right\}}{\tilde I_1(\varepsilon)}
  =0.
\end{equation}
To prove the limit \eqref{eq:limit-to-prove-I1-dominates}, 
we first observe that, upon application of the chain rule in the definition  \eqref{eq:definition-elasticity} of the elasticity $\rho$ of $f$, we obtain
\begin{equation}\label{eq:chain-rule-elasticity}
  \rho(t)
  = -\frac{e^{-t}f'\left(e^{-t}\right)}{f\left(e^{-t}\right)},
  \quad \text{for all } t>0.
\end{equation}
Therefore, choosing $t=\ln(\varepsilon^{-1})$ in \eqref{eq:chain-rule-elasticity} and using the formula for the derivative of $\tilde I_1$ given in \eqref{eq:derivative-regularized-integral}, 
we get
\begin{equation*}
  \rho\left(\ln\left(\varepsilon^{-1}\right)\right)
  = - \frac{\varepsilon f'(\varepsilon)}{f(\varepsilon)}
  = \frac{f'(\varepsilon)}{\tilde I'_1(\varepsilon)},
  \quad \text{for all } \varepsilon > 0.
\end{equation*}
We know that $\tilde I_1(\varepsilon)$ goes to infinity as $\varepsilon\to 0$.
Moreover, we know by (RC) that $\rho(\ln\left(\varepsilon^{-1}\right))$ possesses a (potentially infinite) limit $b$ as $\varepsilon \to 0$.
The conditions for application of l'H\^{o}pital's rule (see, e.g., \cite[Theorem~5.13]{rudin1953principles}) are thus satisfied, yielding
\begin{equation}\label{eq:application-lhopital-1}
  \lim_{\varepsilon\to 0} \frac{f(\varepsilon)}{\tilde I_1(\varepsilon)} 
  = \lim_{\varepsilon\to 0} \rho\left(\ln\left(\varepsilon^{-1}\right)\right)
  \eqqcolon b.
\end{equation}
We next split the analysis into three cases depending on whether $b$ is zero, non-zero and finite, or infinite.

First, in the case $b=0$, i.e.,
\begin{equation*}
  \lim_{\varepsilon\to 0} \frac{f(\varepsilon)}{\tilde I_1(\varepsilon)} 
  = 0,
\end{equation*}
the limit \eqref{eq:limit-to-prove-I1-dominates} is trivially verified and there is nothing to prove.

Second, in the case $b\in(0, \infty)$,
we know from Appendix~\ref{sec:reg-var-zero} that the function $f$ is regularly varying (at zero) with index $b$.
In particular, $f$ can be factorized according to $f(\varepsilon) = \varepsilon^{-b} \psi(\varepsilon)$, for all $\varepsilon>0$ and with $\psi$ a slowly varying function (see \cite[Theorem~1.4.1]{binghamRegularVariation1987}).
Consequently, we get the asymptotic equivalence 
\begin{equation}\label{eq:limit-case-2-O-negligible}
    \frac{\sqrt{f (\varepsilon)\ln\left(f (\varepsilon)\right)\ln\left(\varepsilon^{-1}\right)}}{\tilde I_1(\varepsilon)}
    \sim b\, \varepsilon^{b/2} \sqrt{\frac{\ln\left(\varepsilon^{-b} \psi(\varepsilon)\right)\ln\left(\varepsilon^{-1}\right)}{\psi(\varepsilon)}},
    \quad \text{as } \varepsilon\to 0,
\end{equation}
where we used that \eqref{eq:application-lhopital-1} implies $\tilde I_1(\varepsilon) \sim b f(\varepsilon)$, as $\varepsilon\to 0$.
With the help of \cite[Section~1.3.3]{binghamRegularVariation1987},
we can check that the square root appearing in the right-hand side of \eqref{eq:limit-case-2-O-negligible} is a slowly varying function of $\varepsilon$, which allows us to apply \cite[Proposition~1.3.6 (v)]{binghamRegularVariation1987} to deduce 
\begin{equation}\label{eq:limit-reg-var-zero}
  \lim_{\varepsilon\to 0}
  \varepsilon^{b/2} \sqrt{\frac{\ln\left(\varepsilon^{-b} \psi(\varepsilon)\right)\ln\left(\varepsilon^{-1}\right)}{\psi(\varepsilon)}}
  = 0.
\end{equation}
The desired result \eqref{eq:limit-to-prove-I1-dominates} follows from \eqref{eq:limit-case-2-O-negligible} and \eqref{eq:limit-reg-var-zero}.

Finally, we turn to the case $b=\infty$.
Similarly to the previous case, we aim to prove
\begin{equation}\label{eq:limit-proof-case-3-non-log}
  \lim_{\varepsilon\to 0} \frac{\sqrt{f(\varepsilon)\ln\left(f (\varepsilon)\right)\ln\left(\varepsilon^{-1}\right)}}{\tilde I_1(\varepsilon)} = 0.
\end{equation}
We will do so by working in logarithmic scale.
Concretely, setting $t=\ln(\varepsilon^{-1})$, 
we are equivalently required to prove the logarithmic-scale counterpart of \eqref{eq:limit-proof-case-3-non-log}, that is
\begin{equation}\label{eq:limit-proof-case-3-log}
  \lim_{t\to \infty}\frac{\sqrt{t h(t)}\, e^{h(t)/2}}{\tilde I_1(e^{-t})}=0,
\end{equation}
where $h$ is the same as in \eqref{eq:definition-elasticity}.

To this end,
we first prove that 
\begin{equation}\label{eq:limit-proof-case-3}
  \lim_{t\to \infty} \rho(t) \,  g(t) \sqrt{t}
  = 0,
\end{equation}
where we introduced the auxiliary function
\begin{equation}\label{eq:definition-aux-fct-g}
  g(t)\coloneqq \sqrt{h(t)}\, e^{-h(t)/2},
  \quad \text{for all } t>0.
\end{equation}
Since $h'(t) = \rho(t)$ tends to infinity as $t\to\infty$, 
so does $h(t)$,
and we can find $t_0>0$ for which $h(t)\geq 2$ for all $t\geq t_0$.
Again using $h'(t) = \rho(t)$,
we can thus compute the derivative of $g$ according to 
\begin{equation}\label{eq:ODE}
  g'(t)
  = \frac{h'(t)}{2\sqrt{h(t)}} \, e^{-h(t)/2} - \frac{h'(t)}{2}\sqrt{h(t)}\, e^{-h(t)/2}
  = -\frac{h(t)-1}{2 h(t)} \rho(t)\, g(t),
\end{equation}
for all $t\geq t_0$.
Integrating the differential equation \eqref{eq:ODE}, we get 
\begin{equation}\label{eq:integrated-ODE}
  g(t)
  = g(t_0) \exp\left(-\frac{1}{2}\int_{t_0}^{t}\frac{h(u)-1}{h(u)}\rho(u)\diff u\right)
  \leq g(t_0) \exp\left(-\frac{1}{4}\int_{t_0}^{t}\rho(u)\diff u\right),
\end{equation}
for all $t\geq t_0$,
where the last step uses that $h(t)\geq 2$.
We deduce from \eqref{eq:integrated-ODE} that the limit \eqref{eq:limit-proof-case-3}  we wish to prove is implied by the limit
\begin{equation}\label{eq:limit-proof-case-3-bis}
  \rho(t)\exp\left(-\frac{1}{4}\int_{t_0}^{t}\rho(u)\diff u\right) \sqrt{t} \xrightarrow{} 0,
  \quad \text{as } t \to \infty.
\end{equation}
Using successively that $\rho$ is non-negative (since $h$ is non-decreasing) and that,  by (RC), $\rho$ is non-decreasing on $(t^*,\infty)$, for some $t^*>0$, 
we can lower-bound the integral of $\rho$ according to
\begin{equation}\label{eq:LB-integral-rho}
  \frac{1}{4}\int_{t_0}^{t}\rho(u)\diff u 
  \geq \frac{1}{4}\int_{t/2}^{t}\rho(u)\diff u 
  \geq \frac{t\rho(t/2)}{8},
  \quad \text{for all } t\geq 2\max\{t_0, t^*\}.
\end{equation}
Again using (RC), we have 
\begin{equation}\label{eq:two-individual-limits}
  \frac{t\rho(t/2)}{16} - \ln\left(\rho(t)\right) \to \infty,
  \quadtext{and, since $\rho(t)\to\infty$,} \frac{t\rho(t/2)}{16} - \ln\left( \sqrt{t}\right) \to \infty,
\end{equation}
all limits being taken as $t\to\infty$.
In particular, \eqref{eq:LB-integral-rho} and \eqref{eq:two-individual-limits} together imply
\begin{equation*}
  \rho(t)\exp\left(-\frac{1}{4}\int_{t_0}^{t}\rho(u)\diff u\right) \sqrt{t} 
  \leq \exp\left(-\frac{t\rho(t/2)}{8} + \ln\left(\rho(t)\right) + \ln\left( \sqrt{t}\right)\right)\xrightarrow{} 0,
  \quad \text{as } t \to \infty,
\end{equation*}
which is \eqref{eq:limit-proof-case-3-bis}, and therefore implies \eqref{eq:limit-proof-case-3}.

The last step consists in computing the derivatives
\begin{equation*}
  \left(\tilde I_1\left(e^{-t}\right)\right)'
  = e^{-t} \frac{f(e^{-t})}{e^{-t}} = e^{h(t)},
  \quad \text{for all } t>0,
\end{equation*}
and 
\begin{align*}
  \left(\sqrt{t h(t)}\, e^{h(t)/2}\right)'
  &= \frac{h'(t)}{2} \sqrt{t h(t)}\, e^{h(t)/2} + \frac{h'(t)}{2} \sqrt{\frac{t}{h(t)} }\, e^{h(t)/2} + \frac{1}{2} \sqrt{\frac{h(t)}{t} }\, e^{h(t)/2}\\ 
  &= \left(\frac{\rho(t)(h(t)+1)}{2h(t)}+ \frac{1}{2t}\right) \sqrt{t h(t)}\, e^{h(t)/2}, \quad \text{for all } t>0,
\end{align*}
so that, applying \eqref{eq:limit-proof-case-3}, we get
\begin{equation*}
  \frac{\left(\sqrt{t h(t)}\, e^{h(t)/2}\right)'}{\left(\tilde I_1\left(e^{-t}\right)\right)'}
  \sim \frac{1}{2} \rho(t) \, g(t) \sqrt{t}
  \to 0,
  \quad \text{as } t\to \infty.
\end{equation*}
Upon application of l'H\^{o}pital's rule (see, e.g., \cite[Theorem~5.13]{rudin1953principles}),
which is justified since $\tilde I_1\left(e^{-t}\right)$ tends to infinity in the limit $t\to\infty$,
we get 
\begin{equation*}
  \lim_{t\to\infty} \frac{\sqrt{t h(t)}\, e^{h(t)/2}}{\tilde I_1\left(e^{-t}\right)}
  = \lim_{t\to\infty}  \frac{\left(\sqrt{t h(t)}\, e^{h(t)/2}\right)'}{\left(\tilde I_1\left(e^{-t}\right)\right)'}
  =0,
\end{equation*}
which is the desired result \eqref{eq:limit-proof-case-3-log}.

\subsection{Proof of Theorem~\ref{thm:linear-risk-formula}}\label{sec:proof-theorem-linear-risk-formula}

Let $\sigma>0$ be fixed throughout the proof.
Let $\mu_* \coloneqq \max_{n\in\Ns}\mu_n$ and let $n^*\in\Ns$ be such that $\mu_{n^*}=\mu_*$;
they are both well-defined thanks to the assumption $\mu_n\to 0$, as $n\to\infty$.
We start from the definition of the linear minimax risk (recall Definition~\ref{def:minimax-linear-and-nonlinear}), 
in which we substitute $\hat x_{\sigma}(y)$ by $\{c_n y_n\}_{n\in\Ns}$, 
to get
\begin{align}
  R^L_{\sigma}\left(\ellip\right)
  = \newinf_{c}\,  \sup_{x \in \ellip} \, \mathbb{E}_{y\sim x}\left[\left\| cy-x \right\|^2_2\right]
  &= \newinf_{c}\,  \sup_{x \in \ellip} \, \mathbb{E}_{y\sim x}\left[\left\| c \sigma \xi + (1-c)x  \right\|^2_2\right]\label{eq:def-lin-risk-in-proof-1}\\
  &= \newinf_{c}\,  \left\{\sigma^2\sum_{n\in\Ns} c_n^2 + \sup_{x \in \ellip} \, \left\{\sum_{n\in\Ns}\left(1-c_n\right)^2x_n^2\right\}\right\}.\label{eq:def-lin-risk-in-proof-2}
\end{align} 
It is known since Pinsker's seminal work \cite{pinsker1980optimal} (see also \cite[Theorem~5.1]{johnstone2019estimation} or \cite[Lemma~3.2]{tsybakovIntroductionNonparametricEstimation2009})
that the coefficients $\{c_n\}_{n\in\Ns}$ can be taken of the form $c_n=(1-\varepsilon/\mu_n)_+$ for some $\varepsilon \in (0, \mu_*)$, with the convention $c_n=0$ if $\mu_n=0$.
(Recall that $(t)_+$ is equal to $t$ if $t>0$, and $0$ otherwise.)
We can thus replace the infimum over $c$ by an infimum over $\varepsilon$ in \eqref{eq:def-lin-risk-in-proof-1}--\eqref{eq:def-lin-risk-in-proof-2}
to get 
\begin{equation}\label{eq:formula-risk-infimum-eps-almost-there}
  R^L_{\sigma}\left(\ellip\right)
  = \newinf_{\varepsilon>0}\,   \left\{\sigma^2\sum_{n\in\Ns}\left(1-\frac{\varepsilon}{\mu_n}\right)_+^2 
  + \sup_{x \in \ellip} \, \left\{\sum_{n\in\Ns}\left(1-\left(1-\frac{\varepsilon}{\mu_n}\right)_+\right)^2x_n^2\right\}\right\}.
\end{equation}
Now, for all $x \in \ellip$ and all $\varepsilon>0$, 
we can split the second sum appearing in \eqref{eq:formula-risk-infimum-eps-almost-there} into two sums, depending on whether the $(\cdot)_+$-term is saturated or not, to get
\begin{align}
  \sum_{n\in\Ns}\left(1-\left(1-\frac{\varepsilon}{\mu_n}\right)_+\right)^2x_n^2
  &= \sum_{\mu_n \geq \varepsilon }\left(\frac{\varepsilon}{\mu_n}\right)^2x_n^2 + \sum_{\mu_n < \varepsilon }x_n^2 \label{eq:bound-pinsker-linear-sup-1}\\
  &= \varepsilon^2\sum_{\mu_n \geq \varepsilon }\frac{x_n^2}{\mu_n^2} +\varepsilon^2 \sum_{\mu_n < \varepsilon }\frac{x_n^2}{\varepsilon^2} 
  \leq \varepsilon^2\sum_{n\in\Ns}\frac{x_n^2}{\mu_n^2} 
  \leq \varepsilon^2.\label{eq:bound-pinsker-linear-sup-2}
\end{align}
Both inequalities in \eqref{eq:bound-pinsker-linear-sup-1}--\eqref{eq:bound-pinsker-linear-sup-2} become equalities when $x \in \ellip$ satisfies $x_{n^*} = \mu_*$ and $x_{n} = 0$, for all $n \neq n^*$.
Consequently, we have
\begin{equation*}
  \sup_{x \in \ellip} \, \left\{\sum_{n\in\Ns}\left(1-\left(1-\frac{\varepsilon}{\mu_n}\right)_+\right)^2x_n^2\right\} 
  = \varepsilon^2,
  \quad\text{for all } \varepsilon >0,
\end{equation*}
which, when plugged in \eqref{eq:formula-risk-infimum-eps-almost-there}, 
yields 
\begin{equation}\label{eq:formula-risk-infimum-eps}
  R_{\sigma}^L \left(\ellip\right)
  = \newinf_{\varepsilon>0} \left\{\sigma^2 \sum_{n\in\Ns} \left(1-\frac{\varepsilon}{\mu_n}\right)_+^2 + \varepsilon^2 \right\}.
\end{equation}

We next rewrite the sum appearing in the right-hand side of \eqref{eq:formula-risk-infimum-eps} as a Stieltjes integral; 
recall that  similar arguments were used when establishing \eqref{eq:stiltjes-to-sum}.
Specifically, we have
\begin{equation}\label{eq:formula-risk-infimum-eps-2}
  \sum_{n\in\Ns} \left(1-\frac{\varepsilon}{\mu_n}\right)_+^2
  = - \int_{\varepsilon}^{\infty} \left(1-\frac{\varepsilon}{\dumvar}\right)^2 \diff \ecv(\dumvar),
  \quad \text{for all } \varepsilon>0.
\end{equation}
Performing integration by parts in the right-hand side of \eqref{eq:formula-risk-infimum-eps-2}, we obtain
\begin{equation}\label{eq:formula-risk-infimum-eps-3}
  - \int_{\varepsilon}^{\infty} \left(1-\frac{\varepsilon}{\dumvar}\right)^2 \diff \ecv(\dumvar)
  = \int_{\varepsilon}^{\infty} \frac{2\varepsilon}{\dumvar^2} \left(1-\frac{\varepsilon}{\dumvar}\right) \ecv(\dumvar)\diff \dumvar
  = 2\varepsilon \, \left(\intk{2}(\varepsilon)-\intk{3}(\varepsilon) \, \varepsilon \right),
\end{equation}
for all $\varepsilon>0$.
Therefore, by defining an auxiliary function $\Phi \colon \Rp \to \Rp$ according to
\begin{equation}\label{eq:definition-Phi-RL}
  \Phi (\varepsilon)
  \coloneqq  2\sigma^2 \varepsilon \,  \left(\intk{2}(\varepsilon)-\intk{3}(\varepsilon) \, \varepsilon \right)+ \varepsilon^2,
  \quad \text{for all } \varepsilon >0,
\end{equation}
we can summarize the findings of \eqref{eq:formula-risk-infimum-eps}, \eqref{eq:formula-risk-infimum-eps-2}, and \eqref{eq:formula-risk-infimum-eps-3} with the formula
\begin{equation}\label{eq:reformulation-linear-risk-integral}
   R_{\sigma}^L \left(\ellip\right) = \newinf_{\varepsilon>0} \Phi(\varepsilon).
\end{equation}
One easily verifies, from its definition, 
that $\Phi$ is differentiable. 
In order to compute its derivative, we first observe that the relations
\begin{equation}\label{eq:derivative-integrals}
  \intk{2}'(\varepsilon) = -\frac{\ecv(\varepsilon)}{\varepsilon^2}
  \quadtext{and}
  \intk{3}'(\varepsilon) = -\frac{\ecv(\varepsilon)}{\varepsilon^3},
  \quad \text{for all }\varepsilon >0,
\end{equation}
directly follow from the definition of the type-$2$ and type-$3$ integrals $\intk{2}$ and $\intk{3}$ in \eqref{eq:main-definition-integrals-order-k}.
In particular, we derive 
\begin{equation*}
  \left(\intk{2}(\varepsilon)-\intk{3}(\varepsilon) \, \varepsilon \right)'
  = -\frac{\ecv(\varepsilon)}{\varepsilon^2} + \frac{\ecv(\varepsilon)}{\varepsilon^3}\varepsilon - \intk{3}(\varepsilon)
  = - \intk{3}(\varepsilon),
    \quad \text{for all } \varepsilon >0,
\end{equation*}
from which the derivative of $\Phi$ is deduced as follows:
\begin{align}
  \Phi'(\varepsilon)
  &=2\sigma^2 \left(\intk{2}(\varepsilon)-\intk{3}(\varepsilon) \, \varepsilon \right)
  - 2\sigma^2 \intk{3}(\varepsilon) \, \varepsilon + 2\varepsilon \label{eq:derivative-of-g-1} \\
  &=2\sigma^2 \left(\intk{2}(\varepsilon)-2\intk{3}(\varepsilon) \, \varepsilon \right)
  + 2\varepsilon,\label{eq:derivative-of-g-2}
\end{align}
for all $\varepsilon >0$.
The critical points of $\Phi$, that is, 
the values of $\varepsilon$ for which $\Phi'(\varepsilon)=0$, 
are precisely the solutions of  \eqref{eq:definition-critical-radius}.
We now prove that \eqref{eq:definition-critical-radius} has a solution and that this solution is unique.
To this end, we introduce a second auxiliary function  $\Psi \colon \Rp \to \R$ according to  
\begin{equation}\label{eq:definition-Psi-fct}
  \Psi(\varepsilon) 
  \coloneqq  2 \intk{3}(\varepsilon) - \frac{\intk{2}(\varepsilon)}{\varepsilon},
  \quad \text{for all } \varepsilon >0,
\end{equation}
and study its variations.
The function $\Psi$ is  also differentiable;  
its derivative can be explicitly derived---with the help of \eqref{eq:derivative-integrals}---according to
\begin{equation*}
  \Psi'(\varepsilon)
  =\left(2 \intk{3}(\varepsilon) - \frac{\intk{2}(\varepsilon)}{\varepsilon}\right)'
  = \frac{\intk{2}(\varepsilon)}{\varepsilon^2} - \frac{\ecv(\varepsilon)}{\varepsilon^3},
    \quad \text{for all } \varepsilon >0.
\end{equation*}
Recalling the definition of $\intk{2}$ in \eqref{eq:main-definition-integrals-order-k} and using that $\ecv$ is non-increasing and not identically zero yields
\begin{equation}\label{eq:inequality-second-derivative}
    \intk{2}(\varepsilon) - \frac{\ecv(\varepsilon)}{\varepsilon}
    = \int_{\varepsilon}^{\infty}\frac{\ecv(\dumvar) - \ecv(\varepsilon)}{\dumvar^2} \diff \dumvar
    < 0
    \quadtext{and, a fortiori,}
    \Psi '(\varepsilon)<0,
\end{equation}
for all $\varepsilon \in (0, \mu_*)$.
Furthermore, choosing an arbitrary $\varepsilon_0\in (0, \mu_*)$, we get 
\begin{align*}
  \varepsilon^2\, \Psi'(\varepsilon) = 
  \intk{2}(\varepsilon) - \frac{\ecv(\varepsilon)}{\varepsilon}
  &= \int_{\varepsilon}^{\varepsilon_0}\frac{\ecv(\dumvar)}{\dumvar^2} \diff \dumvar + \intk{2}(\varepsilon_0) - \frac{\ecv(\varepsilon)}{\varepsilon}\\
  &\leq \ecv (\varepsilon)\left(\frac{1}{\varepsilon} - \frac{1}{\varepsilon_0} \right) + \intk{2}(\varepsilon_0) - \frac{\ecv(\varepsilon)}{\varepsilon}
  \leq \intk{2}(\varepsilon_0) - \frac{\ecv(\varepsilon_0)}{\varepsilon_0},
\end{align*}
for all $\varepsilon\in (0, \varepsilon_0)$.
Therefore,
setting $c_0 \coloneqq \ecv(\varepsilon_0) / \varepsilon_0 - \intk{2}(\varepsilon_0)$, which is positive by \eqref{eq:inequality-second-derivative}, 
we get the bound $\Psi'(\varepsilon) \leq -c_0/\varepsilon^2$, for all $\varepsilon\in (0, \varepsilon_0)$.
Integrating this bound between $\varepsilon$ and $\varepsilon_0$, we get 
\begin{equation}\label{eq:LB-on-Psi}
  \Psi (\varepsilon) 
  \geq \Psi (\varepsilon_0) + c_0 \left(\varepsilon^{-1}-\varepsilon^{-1}_0\right),
  \quad \text{for all } \varepsilon >0.
\end{equation}
The right-hand side of \eqref{eq:LB-on-Psi} tends to $\infty$ when $\varepsilon\to 0$, and thus so does $\Psi(\varepsilon)$.
We have proven that the function $\Psi$ is continuous strictly decreasing on $(0, \mu_*)$, with respective limits $\infty$ and $0$ as $\varepsilon \to 0$ and $\varepsilon \to \mu_*$.
The intermediate value theorem \cite[Theorem~4.23]{rudin1953principles} allows us to conclude about the existence and uniqueness of the solution of $\Psi(\varepsilon) = 1/\sigma^2$.
Let us call $\varepsilon_{\sigma}$ this unique solution.
Additionally, from \eqref{eq:derivative-of-g-1}--\eqref{eq:derivative-of-g-2} and \eqref{eq:definition-Psi-fct},
we get $\Phi'(\varepsilon) = 2\varepsilon(1-\sigma^2\Psi(\varepsilon))$, for all $\varepsilon>0$,
so that we can also conclude that $\varepsilon_{\sigma}$ is the unique critical point (and, in fact, the unique minimizer as $\Psi$  is decreasing) of $\Phi$.
We finish the proof of Theorem~\ref{thm:linear-risk-formula} by using successively \eqref{eq:reformulation-linear-risk-integral} and $\Phi'(\varepsilon_{\sigma})=0$ to prove \eqref{eq:linear-risk-closed-form} as follows
\begin{align}
   R_{\sigma}^L \left(\ellip\right)
   = \newinf_{\varepsilon>0} \Phi(\varepsilon)
   = \Phi(\varepsilon_{\sigma})
   &= 2\sigma^2 \varepsilon_{\sigma} \left(\intk{2}(\varepsilon_{\sigma})-\intk{3}(\varepsilon_{\sigma}) \, \varepsilon_{\sigma} \right)+ \varepsilon_{\sigma}^2 \label{eq:last-step-proof-RL-1} \\
   &= \sigma^2 \varepsilon_{\sigma} \intk{2}(\varepsilon_{\sigma}) + \sigma^2 \varepsilon_{\sigma}^2 \left( \frac{\intk{2}(\varepsilon_{\sigma})}{\varepsilon_{\sigma}} - 2 \intk{3}(\varepsilon_{\sigma}) \right) + \varepsilon_{\sigma}^2 \nonumber \\
   &= \sigma^2 \varepsilon_{\sigma} \intk{2}(\varepsilon_{\sigma}) + \frac{\varepsilon_{\sigma}}{2}\Phi'(\varepsilon_{\sigma}) 
   = \sigma^2 \varepsilon_{\sigma} \intk{2}(\varepsilon_{\sigma}).\label{eq:last-step-proof-RL-2}
\end{align}

\subsection{Proof of Theorem~\ref{thm:risk-linear-2-nonlinear}}\label{sec:proof-theorem-risk-linear-2-nonlinear}

Let us first call $r_\sigma$ the right-hand side of  \eqref{eq:non-asymp-ratio-lin-non-lin-risks} with a leading constant of $2\sqrt{2} $ instead of $4\sqrt{2}$, that is, we set
\begin{equation}\label{eq:nvdksjnkenljkqbrkjetre}
    r_\sigma \coloneqq 
     \frac{2 \sqrt{2} \,  \sigma}{\varepsilon_\sigma}\sqrt{W \left(\left(\frac{\left(1+\sqrt{3}\right) \mu_*^2 \, \varepsilon_\sigma}{ \sqrt{2} \, \sigma R^L_{\sigma}\left(\ellip\right)}\right)^2\right)},
    \quad \text{for all } \sigma>0.
\end{equation}
We will adapt the classical proof of Pinsker's asymptotic minimaxity theorem (see, e.g., \cite[Section 5.3]{johnstone2019estimation} and \cite[Chapter~3]{tsybakovIntroductionNonparametricEstimation2009}) by keeping track of the second-order terms.
Since the infimum defining the non-linear minimax risk in \eqref{eq:def-non-linear-risk} is taken over a larger set than the one in \eqref{eq:def-linear-risk},
we must have $ R_{\sigma} \left(\ellip\right)\leq R_{\sigma}^L\left(\ellip\right)$, for all $\sigma>0$.
We are thus left to show
\begin{equation}\label{eq:target-rsigma-bound}
    R_{\sigma} \left(\ellip\right)
    \geq  R_{\sigma}^L\left(\ellip\right)\left(1-2 r_\sigma\right),
    \quad \text{for all }\sigma >0.
\end{equation}
The bound holds trivially if $r_\sigma \geq 1/2$.
Therefore, in what follow, we fix $\sigma>0$ and may assume, without loss of generality, that $r_\sigma \in (0, 1/2)$.

To establish \eqref{eq:target-rsigma-bound}, we follow a variational approach and apply the Bayes-minimax theorem (see, e.g., \cite[Theorem~4.12]{johnstone2019estimation} and the discussion thereafter) to write the minimax risk according to
\begin{equation*}
      R_{\sigma} \left(\ellip\right)
    = \sup_{\text{supp}(\nu) \, \subseteq \, \ellip} B_\sigma(\nu),
\end{equation*}
where the supremum is taken over all probability measures $\nu$ supported in $\ellip$ and $B_\sigma(\nu)$ denotes the Bayes risk for $\nu$, that is,
\begin{equation}\label{eq:def-bayes-risk}
  B_\sigma(\nu)
  \coloneqq \newinf_{\hat x_\sigma} B_\sigma\left(\hat x_\sigma, \nu\right)
  \quadtext{where}
  B_\sigma\left(\hat x_\sigma, \nu\right) \coloneqq \mathbb{E}_{x\sim \nu}\left[\textnormal{MSE}_{\sigma}  \left(x, \hat x_\sigma; \ellip\right) \right].
\end{equation}
Therefore, it is sufficient to show that 
there exists a probability measure $\nu_\sigma$ supported on $\ellip$ satisfying
\begin{equation}\label{eq:nzkajrebnkjbzhbqhqwwdsza}
    B_\sigma (\nu_\sigma) 
    \geq   R_{\sigma}^L\left(\ellip\right)\left(1-2 r_\sigma\right).
\end{equation}

In view of constructing such a probability measure $\nu_\sigma$,
let us first introduce a Gaussian prior $\nu_\sigma^G$, supported on $\ell^2(\mathbb{N}^*)$, with diagonal covariance operator. 
Specifically, we set 
\begin{equation}\label{eq:dvnjknknfjke}
    \nu_\sigma^G \coloneqq  \bigotimes_{n\in\Ns} \nu_{\sigma,n}^G,
\end{equation}
where $\otimes$ denotes the tensor product of measures and $\nu_{\sigma,n}^G$ is a Gaussian $\mathcal{N}\left(0, (1-r_\sigma) \,  \tau_{\sigma,n}^2 \right)$, with 
\begin{equation}\label{eq:def-tau-variances}
  \tau_{\sigma,n} \coloneqq 
  \sigma \sqrt{\left(\frac{\mu_n}{\varepsilon_\sigma}-1\right)_+},
  \quad \text{for all } n\in \Ns.
\end{equation}
The assumption $\mu_n\to 0$ (as $n\to\infty$) guarantees that there are only finitely many non-zero $\tau_{\sigma,n}$.
Our candidate for the target probability measure $\nu_\sigma$ is obtained by conditioning $\nu_\sigma^G$ to the ellipsoid $\ellip$, i.e., 
\begin{equation}\label{eq:znrkaegqhkjbazhrjsbghjvzeszergerge}
    \nu_\sigma (A) = \frac{\nu_\sigma^G \left(A\cap \ellip \right)}{\nu_\sigma^G \left(\ellip \right)}, 
    \quad \text{for all measurable sets } A.
\end{equation}
By definition, $\nu_\sigma$ is supported on $\ellip$. 
We next proceed to show that (\ref{eq:nzkajrebnkjbzhbqhqwwdsza}) is satisfied with the choice \eqref{eq:znrkaegqhkjbazhrjsbghjvzeszergerge}.

To this end, we begin with a few useful observations.
First, it is a standard result that the Bayes risk for the Gaussian probability measure $\nu_\sigma^G$ is given by the sum of univariate estimates (see, e.g.,  the discussion in \cite[Chapter~2.3]{johnstone2019estimation} or \cite[Lemma~3.4]{tsybakovIntroductionNonparametricEstimation2009}, \cite[Lemma~5.1.13 and Example~5.1.14]{Lehmann1998}) according to
\begin{equation}\label{eq:klsdnsssdffnkqma22}
  B_\sigma\left(\nu_\sigma^G \right)
  = \sum_{n\in\Ns} \frac{\sigma^2 \, (1-r_\sigma) \, \tau_{\sigma, n}^2 }{\sigma^2+(1-r_\sigma)\, \tau_{\sigma, n}^2}.
\end{equation}
Next,
using a Stieltjes integral representation and performing integration by parts (similarly to \eqref{eq:formula-risk-infimum-eps-2}--\eqref{eq:formula-risk-infimum-eps-3}),
we get the relation
\begin{align}
  \sum_{n \colon \mu_n>\, 0} \frac{\sigma^2}{\varepsilon_\sigma\mu_n}\left(1- \frac{\varepsilon_\sigma}{\mu_n}\right)_+
  &= - \int_{\varepsilon_\sigma}^{\infty} \frac{\sigma^2}{\varepsilon_\sigma \dumvar}\left(1- \frac{\varepsilon_\sigma}{\dumvar}\right) \diff \ecv(\dumvar)\label{eq:klsdnsssdffnkqma0}\\
  &= \sigma^2 \int_{\varepsilon_\sigma}^{\infty} \ecv(\dumvar) \left(\frac{2}{\dumvar^3}- \frac{1}{\varepsilon_\sigma \dumvar^2}\right) \diff \dumvar \nonumber \\
  &= \sigma^2\left(2 \intk{3}(\varepsilon_\sigma) - \frac{\intk{2}(\varepsilon_\sigma)}{\varepsilon_\sigma}\right)
  = 1,\label{eq:klsdnsssdffnkqma33}
\end{align}
where the last step follows from the identity \eqref{eq:definition-critical-radius}.
Combining \eqref{eq:def-tau-variances} with \eqref{eq:klsdnsssdffnkqma0}--\eqref{eq:klsdnsssdffnkqma33} yields 
\begin{equation}\label{eq:klsdnsssdffnkqma}
  \sum_{n \colon \mu_n>\, 0} \frac{\tau_{\sigma,n}^2}{\mu_n^2}
  = \sum_{n \colon \mu_n>\, 0} \frac{\sigma^2}{\varepsilon_\sigma\mu_n}\left(1- \frac{\varepsilon_\sigma}{\mu_n}\right)_+
  =1.
\end{equation}
Likewise, we have
\begin{equation}\label{eq:jkkjhbfhdshhhfdhsjjsksskk}
   \sum_{n \colon \mu_n>\, 0} \left(1- \frac{\varepsilon_\sigma}{\mu_n}\right)_+
  = - \int_{\varepsilon_\sigma}^{\infty} \left(1- \frac{\varepsilon_\sigma}{\dumvar}\right) \diff \ecv(\dumvar)
  = \varepsilon_\sigma \int_{\varepsilon_\sigma}^{\infty} \frac{\ecv(\dumvar)}{\dumvar^2} \diff \dumvar
  = \varepsilon_\sigma \intk{2}(\varepsilon_\sigma).
\end{equation}
Multiplying both sides of \eqref{eq:jkkjhbfhdshhhfdhsjjsksskk} by $\sigma^2$ and using the identity \eqref{eq:linear-risk-closed-form} yields 
\begin{equation}\label{eq:klsdnsssdffnkqma11}
  R_{\sigma}^L\left(\ellip\right) 
  = \sigma^2 \sum_{n \colon \mu_n>\, 0} \left(1- \frac{\varepsilon_\sigma}{\mu_n}\right)_+
  = \sum_{n\in\Ns} \frac{\sigma^2 \, \tau_{\sigma, n}^2 }{\sigma^2+\tau_{\sigma,n}^2},
\end{equation}
where the last step directly follows by substituting the definition \eqref{eq:def-tau-variances} of $\tau_{\sigma, n}$.
From \eqref{eq:klsdnsssdffnkqma22} and \eqref{eq:klsdnsssdffnkqma11}, the Bayes risk for the probability measure $\nu_\sigma^G$ defined in (\ref{eq:dvnjknknfjke}) provides an upper bound on the linear minimax risk according to 
\begin{align}
  B_\sigma\left(\nu_\sigma^G \right)
  &= \sum_{n\in\Ns} \frac{\sigma^2 \, (1-r_\sigma) \, \tau_{\sigma, n}^2 }{\sigma^2+(1-r_\sigma)\, \tau_{\sigma, n}^2} \label{eq:nvkjnrjekajqkbvjanjkjjj1}\\
  &\geq (1-r_\sigma)\,\sum_{n\in\Ns} \frac{\sigma^2 \, \tau_{\sigma, n}^2 }{\sigma^2+\tau_{\sigma,n}^2} 
  =  (1-r_\sigma)\,R_{\sigma}^L\left(\ellip\right).\label{eq:nvkjnrjekajqkbvjanjkjjj2}
\end{align}
Moreover, from the definition \eqref{eq:def-bayes-risk} of the Bayes risk, we directly get that the inequality 
\begin{equation}\label{eq:dvjerkgntrjyehnkjyhnezzzzzz}
  B_\sigma \left(\nu_\sigma^G \right)
  \leq B_\sigma\left(\hat x_{\sigma}, \nu_\sigma^G \right), 
  \quadtext{holds for all} \hat x_{\sigma}.
\end{equation}
From the properties of the conditional expectation (see, e.g., \cite[Chapter~11.2.1]{legallprobaint2022}), 
we can upper-bound the right-hand side of \eqref{eq:dvjerkgntrjyehnkjyhnezzzzzz} according to
\begin{align}
    B_\sigma\left(\hat x_{\sigma}, \nu_\sigma^G \right)
    &= \mathbb{E}_{x \sim \nu_\sigma^G} \left[ \mathbb{E}_{{y\sim x}}\left[\left\|\hat x_{\sigma} (y)-x \right\|^2_2\right] \right] \label{eq:jezhaorheqjsfejkadlzeirkuhgdebz1}\\
    &=  \mathbb{E}_{x \sim \nu_\sigma^G} \left[ \mathbb{E}_{{y\sim x}} \left[\left\| \hat x_{\sigma}(y) - x \right\|^2_2 \right] \mid x \in \ellip \right] \, \nu_\sigma^G \left(\ellip\right)\nonumber\\
    &\quad + \mathbb{E}_{x \sim \nu_\sigma^G} \left[ \mathbb{E}_{{y\sim x}} \left[\left\| \hat x_{\sigma}(y) - x \right\|^2_2 \right] \mid x \in  \ellip^c \right] \, \nu_\sigma^G \left(\ellip^c\right) \nonumber\\
    &\leq \mathbb{E}_{x \sim \nu_\sigma} \left[ \mathbb{E}_{{y\sim x}} \left[\left\| \hat x_{\sigma}(y) - x \right\|^2_2 \right] \right] \nonumber\\
    &\quad + \mathbb{E}_{x \sim \nu_\sigma^G} \left[ \mathbb{E}_{{y\sim x}} \left[\left\| \hat x_{\sigma}(y) - x \right\|^2_2 \right] \mid x \in  \ellip^c \right]\, \nu_\sigma^G \left(\ellip^c\right),\label{eq:jezhaorheqjsfejkadlzeirkuhgdebz2}
\end{align}
for all estimator $\hat x_{\sigma}$, and where we used, in the last step, that $\nu_\sigma$ has been defined in (\ref{eq:znrkaegqhkjbazhrjsbghjvzeszergerge}) to be $\nu_\sigma^G$ conditioned to $\ellip$.
We now choose $\hat x_{\sigma}$ to be the Bayes estimator for $\nu_\sigma$, that is, the estimator
\begin{equation}\label{eq:vnqkjnzerkjqbkjmdlpsp}
  \hat x_{\sigma}^B
  \colon y \longmapsto \mathbb{E}_{x \sim \nu_\sigma} \left[x  \mid  y \right]. 
\end{equation} 
The estimator $\hat x_{\sigma}^B$ attains the infimum in \eqref{eq:def-bayes-risk}; this is a consequence of the properties of the conditional expectation---see, e.g., \cite[Chapter~11.2.3]{legallprobaint2022}.
In particular, the first term in the upper bound (\ref{eq:jezhaorheqjsfejkadlzeirkuhgdebz1})--(\ref{eq:jezhaorheqjsfejkadlzeirkuhgdebz2}) simply becomes the Bayes risk for $\nu_\sigma$, i.e.,
\begin{equation}\label{eq:nsfdkjnrbekjzngqjkndsfs}
    B_\sigma(\nu_\sigma)
    = B_\sigma\left(\hat x^B_\sigma, \nu_\sigma\right)
    =\mathbb{E}_{x \sim \nu_\sigma} \left[ \mathbb{E}_{{y\sim x}} \left[\left\| \hat x^B_\sigma(y) - x \right\|^2_2 \right] \right].
\end{equation}
Assembling (\ref{eq:nvkjnrjekajqkbvjanjkjjj1})--(\ref{eq:nvkjnrjekajqkbvjanjkjjj2}), (\ref{eq:dvjerkgntrjyehnkjyhnezzzzzz}), (\ref{eq:jezhaorheqjsfejkadlzeirkuhgdebz1})--(\ref{eq:jezhaorheqjsfejkadlzeirkuhgdebz2}),
and (\ref{eq:nsfdkjnrbekjzngqjkndsfs}), we have proven so far that 
\begin{equation}\label{eq:kenzfbjhbvhjaebqhjvfeqfff}
    (1-r_\sigma) R_{\sigma}^L\left(\ellip\right)
    \leq  B_\sigma(\nu_\sigma)  +  \mathbb{E}_{x \sim \nu_\sigma^G} \left[ \mathbb{E}_{{y\sim x}} \left[\left\| \hat x_{\sigma}^B(y) - x \right\|^2_2 \right] \mid x \in  \ellip^c \right] \nu_\sigma^G \left(\ellip^c\right).
\end{equation}
In order to obtain the desired bound (\ref{eq:nzkajrebnkjbzhbqhqwwdsza}), we are left to prove that the second term in the right-hand side of (\ref{eq:kenzfbjhbvhjaebqhjvfeqfff}) is upper-bounded by $r_\sigma R_{\sigma}^L(\ellip)$, 
which we proceed to do next.

From the standard inequality $(a+b)^2\leq 2(a^2+b^2)$ on the real line, it is sufficient to bound individually both
\begin{equation}\label{eq:nfqekvnklnzkjekklllqqqq}
      \mathbb{E}_{x \sim \nu_\sigma^G} \left[ \mathbb{E}_{{y\sim x}} \left[\left\| \hat x_{\sigma}^B(y) \right\|^2_2\right]\mid x \in  \ellip^c \right] \nu_\sigma^G \left(\ellip^c\right) 
      \quad \text{and} \quad
      \mathbb{E}_{x \sim \nu_\sigma^G} \left[\left\| x \right\|^2_2 \mid x \in  \ellip^c \right] \nu_\sigma^G \left(\ellip^c\right).
\end{equation}
Recalling the definition (\ref{eq:vnqkjnzerkjqbkjmdlpsp}) of $\hat x_{\sigma}^B$ and applying Jensen's inequality, we obtain 
\begin{equation}\label{eq:jkvbavbebrjhbhjerbvzknjh11}
    \left\| \hat x_{\sigma}^B(y) \right\|^2_2
    \leq \mathbb{E}_{x \sim \nu_\sigma} \left[\left\|x \right\|^2_2 \mid  y \right], 
      \quad \text{for all }  y \in \R^{\Ns}. 
\end{equation}
Furthermore, by definition of the ellipsoid $\ellip$, we have 
\begin{equation}\label{eq:jkvbavbebrjhbhjerbvzknjh12}
    \left\|x \right\|^2_2
    = \sum_{n \colon \mu_n>0} \mu_n^2 \frac{x_n^2}{\mu_n^2}
    \leq \mu_*^2 \sum_{n \colon \mu_n>0}  \frac{x_n^2}{\mu_n^2}
    \leq \mu_*^2,
    \quad \text{for all }x \in \ellip.
\end{equation}
Combining (\ref{eq:jkvbavbebrjhbhjerbvzknjh11}) with (\ref{eq:jkvbavbebrjhbhjerbvzknjh12}), we get 
\begin{equation}\label{eq:ejrkntkzlnskjnbzgrbtryt}
    \mathbb{E}_{x \sim \nu_\sigma^G} \left[ \mathbb{E}_{{y\sim x}} \left[\left\| \hat x_{\sigma}^B(y) \right\|^2_2\right] \mid x \in  \ellip^c \right]\nu_\sigma^G \left(\ellip^c\right) 
    \leq  \mu_*^2\, \nu_\sigma^G \left(\ellip^c\right).
\end{equation}
This provides a bound for the first term in (\ref{eq:nfqekvnklnzkjekklllqqqq}).
In order to bound the second term in (\ref{eq:nfqekvnklnzkjekklllqqqq}), we first recall that 
$\nu_{\sigma,n}^G $ is, by definition, a centered Gaussian probability measure with variance $(1-r_\sigma) \, \tau_{\sigma,n}^2$, so that its fourth moment satisfies
\begin{equation*}
    \mathbb{E}_{x_n \sim \nu_{\sigma,n}^G } \left[x_n^4\right] 
    = 3 \, (1-r_\sigma)^2 \, \tau_{\sigma,n}^4
    \leq 3 \, \tau_{\sigma,n}^4,
    \quad \text{for all }  n\in\Ns.
\end{equation*}
In particular, upon application of Cauchy-Schwarz's inequality, 
we get the upper bound 
\begin{align}
    \mathbb{E}_{x \sim \nu_\sigma^G} \left[\left\| x \right\|^2_2 \mid x \in \ellip^c \right]
    &= \sum_{n\in\Ns} \mathbb{E}_{x \sim \nu_\sigma^G} \left[x_n^2 \mid x \in \ellip^c \right]\label{eq:hjkojezkjklrea1} \\
    &\leq \sum_{n\in\Ns} \sqrt{\mathbb{E}_{x \sim \nu_\sigma^G} \left[x_n^4  \right]\, \mathbb{E}_{x \sim \nu_\sigma^G} \left[ \nu_\sigma^G \left(\ellip^c\right)^{-2}\mathbbm{1}_{\{x\in\ellip^c\}}\right]}\nonumber\\
    &\leq \sqrt{\frac{3}{\nu_\sigma^G \left(\ellip^c\right)}} \, \sum_{n\in\Ns} \tau_{\sigma,n}^2.\label{eq:hjkojezkjklrea2} 
\end{align}
We 
can further upper-bound \eqref{eq:hjkojezkjklrea1}--\eqref{eq:hjkojezkjklrea2} according to 
\begin{equation}\label{eq:hjkojezkjklrea3} 
    \sqrt{\frac{3}{\nu_\sigma^G \left(\ellip^c\right)}} \, \sum_{n\in\Ns} \tau_{\sigma,n}^2
    \leq \mu_*^2\sqrt{\frac{3}{\nu_\sigma^G \left(\ellip^c\right)}} \, \sum_{n \colon \mu_n>0} \frac{\tau_{\sigma, n}^2}{\mu_n^2} 
    = \mu_*^2\sqrt{\frac{3}{\nu_\sigma^G \left(\ellip^c\right)}},
\end{equation}
where we also $\tau_{\sigma,n}=0$ whenever $\mu_n=0$ in the first step, and \eqref{eq:klsdnsssdffnkqma} in the second step.
Putting (\ref{eq:kenzfbjhbvhjaebqhjvfeqfff}), (\ref{eq:ejrkntkzlnskjnbzgrbtryt}), (\ref{eq:hjkojezkjklrea1})--(\ref{eq:hjkojezkjklrea2}), and (\ref{eq:hjkojezkjklrea3}) together, we get 
\begin{equation}\label{eq:zenkjrjkesnwknaezmlkrfkesdgnj}
(1-r_\sigma) \,  R_{\sigma}^L\left(\ellip\right) \leq B_\sigma(\nu_\sigma) + 2\left(1+\sqrt{3}\right)\mu_*^2  \sqrt{\nu_\sigma^G \left(\ellip^c\right)}.
\end{equation}

The probability $\nu_\sigma^G (\ellip^c)$ can be bounded via concentration inequalities.
Specifically, we first use \eqref{eq:klsdnsssdffnkqma} to obtain
\begin{equation}\label{eq:njksvnkjfjnjnnnncjhslqlmmm}
    \sum_{n \colon \mu_n>0} \frac{\mathbb{E}_{x_n \sim \nu_{\sigma,n}^G} \left[x_n^2 \right]}{\mu_n^2}
    = \sum_{n \colon \mu_n>0} \frac{(1-r_\sigma) \,  \tau_{\sigma, n}^2}{\mu_n^2}
    = 1-r_\sigma,
\end{equation}
which then allows us to rewrite the complement of $\ellip$ according to
\begin{equation*}
    \ellip^c
    = \left\{ x \in \ell^2(\Ns) \mid \sum_{n \colon \mu_n>0} \frac{x_n^2- \mathbb{E}_{x_n \sim \nu_{\sigma, n}^G} \left[x_n^2 \right]}{\mu_n^2} > r_\sigma \right\}.
\end{equation*}
By construction of the probability measure $\nu_{\sigma, n}^G$, we can set $x_n^2 = \gamma_n \, \mu_n^2 \,  Z_n^2$, for all $n\in\Ns$,
where $\{Z_n\}_{n \in \mathbb{N}^*}$ are independent standard Gaussian random variables and 
\begin{equation}\label{eq:nksdjjjjbshvsggggggsdfvvdsdv}
    \gamma=\{\gamma_n\}_{n \in \mathbb{N}^*}
    \quad \text{has components} \quad
    \gamma_n \coloneqq 
    \begin{dcases}
      \frac{(1-r_\sigma)\, \tau_{\sigma, n}^2}{\mu_n^2}, \quad &\text{if } \mu_n>0,\\ 
      0, \quad &\text{otherwise.}
    \end{dcases}
\end{equation}
Note that, by the definition (\ref{eq:def-tau-variances}) of $\tau_{\sigma, n}$ and the assumption $\mu_n\to 0$, as $n\to\infty$, we are guaranteed that $\gamma$ has only finitely many non-zero components.
Consequently, the random variable $ \sum_{n\in\Ns} \gamma_n Z_n^2$, follows a weighted $\chi^2$ distribution.
Therefore, by the concentration inequality for weighted $\chi^2$ variables given (see, e.g., \cite[Chapter~2.8~(iv)]{johnstone2019estimation} and, more generally, the concentration theory for Lipschitz functions of Gaussian random variables \cite[Theorem~2.26]{wainwrightHighDimensionalStatistics2019}), we get 
\begin{equation*}
    \nu_\sigma^G \left(\ellip^c\right)
    = \nu_\sigma^G \left(\sum_{n\in\Ns} \gamma_n \left(Z_n^2-1\right)>r_\sigma\right)
    \leq \exp \left(-\frac{r_\sigma^2}{32 \,  \|\gamma\|_1 \, \|\gamma\|_\infty}\right).
\end{equation*}
This upper bound on $\nu_\sigma^G (\ellip^c)$ can be simplified by employing \eqref{eq:njksvnkjfjnjnnnncjhslqlmmm} with \eqref{eq:nksdjjjjbshvsggggggsdfvvdsdv} to bound $\|\gamma\|_1$, and \eqref{eq:def-tau-variances} with \eqref{eq:nksdjjjjbshvsggggggsdfvvdsdv} to bound $\|\gamma\|_\infty$, according to
\begin{equation*}
    \|\gamma\|_1 
    = \sum_{n \colon \mu_n>0} \frac{(1-r_\sigma)\, \tau_{\sigma,n }^2}{\mu_n^2}
    \leq 1
    \quad \text{and} \quad
    \|\gamma\|_\infty
    \leq \max_{n \colon \mu_n>0} \left\{ \frac{\sigma^2\varepsilon_\sigma }{\mu_n \varepsilon_\sigma^2}\left(1-\frac{\varepsilon_\sigma}{\mu_n}\right)_+ \right\}
    \leq \frac{\sigma^2}{4\, \varepsilon_\sigma^2},
\end{equation*}
where we have used, in the last step, that $a(1-a)\leq 1/4$ holds for all $a\in \mathbb{R}$.
We thus obtain
\begin{equation}\label{eq:snvkqjaknqkjbezjbdcqkll}
    \nu_\sigma^G \left(\ellip^c\right)
    \leq \exp\left(-\frac{r_\sigma^2\, \varepsilon_\sigma^2}{8\, \sigma^{2}}   \right).
\end{equation}
Therefore, combining (\ref{eq:zenkjrjkesnwknaezmlkrfkesdgnj}) and (\ref{eq:snvkqjaknqkjbezjbdcqkll}), we have
\begin{equation}\label{eq:nfkjanjknezkjsnkjgssww}
(1-r_\sigma) \,  R_{\sigma}^L\left(\ellip\right) \leq B_\sigma(\nu_\sigma) + 2\left(1+\sqrt{3}\right)\mu_*^2  \exp\left(-\frac{r_\sigma^2\, \varepsilon_\sigma^2}{16\, \sigma^{2}}   \right).
\end{equation}
We can rewrite the definition \eqref{eq:nvdksjnkenljkqbrkjetre} of $r_\sigma$ according to
\begin{equation*}
  \frac{r_\sigma^2\, \varepsilon_\sigma^2}{8\, \sigma^{2}} 
  = W \left(\left(\frac{\left(1+\sqrt{3}\right) \mu_*^2 \, \varepsilon_\sigma}{ \sqrt{2} \, \sigma R^L_{\sigma}\left(\ellip\right)}\right)^2\right),
\end{equation*}
from which the defining property $x=W(x) \exp(W(x))$, for all $x \geq 0$, of the Lambert W function (see, e.g., \cite{corlessLambertFunction1996}) gives
\begin{equation}\label{eq:Lambert-def-appli}
  \frac{r_\sigma^2\, \varepsilon_\sigma^2}{8\, \sigma^{2}} 
  \exp\left(\frac{r_\sigma^2\, \varepsilon_\sigma^2}{8\, \sigma^{2}}   \right)
  = \left(\frac{\left(1+\sqrt{3}\right) \mu_*^2 \, \varepsilon_\sigma}{ \sqrt{2} \, \sigma R^L_{\sigma}\left(\ellip\right)}\right)^2.
\end{equation}
Rearranging terms in \eqref{eq:Lambert-def-appli} yields
\begin{equation}\label{eq:Lambert-deduction}
  2\left(1+\sqrt{3}\right)\mu_*^2 
  \exp\left(-\frac{r_\sigma^2\, \varepsilon_\sigma^2}{16\, \sigma^{2}}   \right)
  = r_\sigma  R^L_{\sigma}\left(\ellip\right).
\end{equation}
Plugging \eqref{eq:Lambert-deduction} in \eqref{eq:nfkjanjknezkjsnkjgssww}, we obtain
\begin{equation*}
(1-r_\sigma) \,  R_{\sigma}^L\left(\ellip\right) \leq B_\sigma(\nu_\sigma) + r_\sigma  R^L_{\sigma}\left(\ellip\right),
\end{equation*}
from which the desired bound \eqref{eq:nzkajrebnkjbzhbqhqwwdsza}
follows, thereby finishing the proof of Theorem~\ref{thm:risk-linear-2-nonlinear}.

\subsection{Proof of Theorem~\ref{thm:bias-variance-decomp}}\label{sec:proof-lemma-RnL-vs-ME}

By Pinsker's asymptotic minimaxity  theorem, it is sufficient to prove the result for the linear minimax risk.
Before we begin, let us take $f$ as in (RC) and introduce its type-$1$, type-$2$, and type-$3$ integrals defined according to
\begin{equation}
  \tilde I_1(\varepsilon)
  \coloneqq \int_{\varepsilon}^{\infty} \frac{f(\dumvar)}{\dumvar} \diff \dumvar,
  \quad \tilde I_2(\varepsilon)
  \coloneqq \int_{\varepsilon}^{\infty} \frac{f(\dumvar)}{\dumvar^2} \diff \dumvar,
  \quadtext{and}
  \tilde I_3(\varepsilon)
  \coloneqq \int_{\varepsilon}^{\infty} \frac{f(\dumvar)}{\dumvar^3} \diff \dumvar,
\end{equation}
for all $\varepsilon > 0$.
Since the function $f$ satisfies $f(\varepsilon)\sim\ecv(\varepsilon)$ as $\varepsilon\to 0$,
we have, by Lemma~\ref{lem:abelian-integration} in Appendix~\ref{sec:asymp-equiv-integration},
\begin{equation}\label{eq:auxiliary-integral-binft}
  \tilde I_1(\varepsilon) \sim \intk{1}(\varepsilon),
  \quad  \tilde I_2(\varepsilon) \sim \intk{2}(\varepsilon),
  \quadtext{and}  \tilde I_3(\varepsilon) \sim \intk{3}(\varepsilon),
  \quad \text{as } \varepsilon \to 0.
\end{equation}

We first consider the case $b=0$.
We derive from the asymptotic equivalence \eqref{eq:auxiliary-integral-binft} for the type-$2$ integral $\intk{2}$ that
\begin{equation}\label{eq:IA-asympt-neg-binft-bis0}
  \lim_{\varepsilon\to 0} \frac{\intk{2}(\varepsilon)\, \varepsilon}{\ecv(\varepsilon)}
  = \lim_{\varepsilon\to 0}  \frac{\tilde I_2(\varepsilon)\, \varepsilon}{f(\varepsilon)}.
\end{equation}
Next, upon application of l'H\^{o}pital's rule to the functions $\tilde I_2$ and $\varepsilon\mapsto f(\varepsilon) \,  \varepsilon^{-1}$---which is justified by the fact that $f(\varepsilon) \,  \varepsilon^{-1} \to \infty$ as $\varepsilon\to 0$,
see \cite[Theorem~5.13]{rudin1953principles}---we get
\begin{equation}\label{eq:IA-asympt-neg-binft-bis1}
  \lim_{\varepsilon\to 0}  \frac{\tilde I_2(\varepsilon)\, \varepsilon}{f(\varepsilon)} 
  = \lim_{\varepsilon\to 0} \frac{\tilde I_2'(\varepsilon)}{f'(\varepsilon)/\varepsilon-f(\varepsilon)/\varepsilon^2}
  = \lim_{\varepsilon\to 0} \frac{f(\varepsilon)}{f(\varepsilon)-f'(\varepsilon)\, \varepsilon}.
\end{equation}
We further observe that (see \eqref{eq:chain-rule-elasticity}--\eqref{eq:application-lhopital-1} for an analogous derivation)
\begin{equation*}
  \rho(\log(\varepsilon^{-1}))
  = - \frac{f'(\varepsilon)\, \varepsilon}{f(\varepsilon)}
  \to b= 0,
  \quad \text{as } \varepsilon \to 0,
\end{equation*}
so that we have proven in \eqref{eq:IA-asympt-neg-binft-bis0} and \eqref{eq:IA-asympt-neg-binft-bis1} that 
\begin{equation}\label{eq:IA-asympt-neg-binft-bis}
  \lim_{\varepsilon\to 0} \frac{\intk{2}(\varepsilon)\, \varepsilon}{\ecv(\varepsilon)}
  = 1.
\end{equation}
A similar line of arguments yields 
\begin{equation}\label{eq:IA-asympt-neg-binft-ter}
  \lim_{\varepsilon\to 0} \frac{\intk{3}(\varepsilon)\, \varepsilon^2}{\ecv(\varepsilon)}
  = \lim_{\varepsilon\to 0} \frac{\tilde I_3'(\varepsilon)}{f'(\varepsilon)/\varepsilon^2-2f(\varepsilon)/\varepsilon^3}
  = \lim_{\varepsilon\to 0} \frac{1}{2+\rho(\log(\varepsilon^{-1}))}
  = \frac{1}{2}.
\end{equation}
Therefore, combining \eqref{eq:IA-asympt-neg-binft-bis} and \eqref{eq:IA-asympt-neg-binft-ter}, we get 
\begin{equation}\label{eq:case-b0-asymp-I2-I3}
  \intk{2}(\varepsilon)-\intk{3}(\varepsilon) \, \varepsilon
  \sim \frac{\ecv(\varepsilon)}{2\varepsilon},
  \quad\text{as } \varepsilon \to 0.
\end{equation}
Plugging \eqref{eq:case-b0-asymp-I2-I3} in \eqref{eq:linear-risk-integral} and applying Lemma~\ref{lem:abelian-infimum} (see Appendix~\ref{sec:asymp-equiv-integration})  yields 
\begin{equation*}
  R^L_{\sigma} \left(\ellip\right)
    \sim 
  \newinf_{\varepsilon > 0} \left\{  \sigma^2 \ecv(\varepsilon) + \varepsilon^2 \right\}
  \sim   \sigma^2 \ecv(\varepsilon_\sigma) + \varepsilon^2_\sigma,
  \quad\text{as } \sigma \to 0,
\end{equation*}
which is the desired result in the case $b=0$.

From now on, we may assume that $b\neq 0$. 
In this case, we make the following useful observation for later:
\begin{equation}\label{eq:IA-asympt-neg-binft}
  \lim_{\varepsilon\to 0} \frac{\intk{1}(\varepsilon)}{\ecv(\varepsilon)}
  = \lim_{\varepsilon\to 0}  \frac{\tilde I_1(\varepsilon)}{f(\varepsilon)}
  = \lim_{\varepsilon\to 0} \frac{\tilde I_1'(\varepsilon)}{f'(\varepsilon)}
  = \lim_{\varepsilon\to 0}  - \frac{f(\varepsilon)}{\varepsilon f'(\varepsilon)} 
  = \lim_{\varepsilon\to 0}  \frac{1}{\rho(\log(\varepsilon^{-1}))}
  = \frac{1}{b},
\end{equation}
where the first equality follows from \eqref{eq:auxiliary-integral-binft}, the second one is an application of l'H\^{o}pital's rule (note that $f(\varepsilon)\to\infty$, as $\varepsilon\to 0$, since $b\neq 0$), 
and the last step uses $\lim_{t\to\infty}\rho(t)=b$. 
In view of relating the right-hand side of \eqref{eq:linear-risk-integral} to the metric entropy,
we apply twice  Lemma~\ref{lem:integral-k-to-k+1}---first with $\typevar_1=2$ and $\typevar_2=3$, then with $\typevar_1=2$ and $\typevar_2=1$---to obtain
\begin{equation}\label{eq:integral-of-integral-of-integral}
  \intk{2}(\varepsilon)-\intk{3}(\varepsilon) \, \varepsilon
  = \varepsilon \int_{\varepsilon}^{\infty} \frac{\intk{2}(\dumvar)}{\dumvar^2}\diff \dumvar
  = \varepsilon \int_{\varepsilon}^{\infty}\frac{\intk{1}(\dumvar)}{\dumvar^3} - \left(\int_{\dumvar}^{\infty} \frac{\intk{1}(\dumvartwo)}{\dumvar^2 \,  \dumvartwo^2}\diff \dumvartwo \right)\diff \dumvar, 
  \quad\text{for all }\varepsilon>0.
\end{equation}
We will split the analysis into two cases depending on whether $b$ is finite.

If $b\in(0, \infty)$, since (RC) is satisfied by assumption, 
a direct application of Theorem~\ref{thm:metric-entropy-integral-form} guarantees that the metric entropy of $\ellip$ satisfies \eqref{eq:asymp-equivalence-ME-I1-first-order}.
By Lemma~\ref{lem:abelian-infimum} (see Appendix~\ref{sec:asymp-equiv-integration}), it is thus sufficient to show that 
\begin{equation}\label{eq:linear-risk-and-I1}
    R_{\sigma}^L \left(\ellip\right)
    \sim  \newinf_{\varepsilon>0} \left\{  \frac{2 \, b\, \sigma^2}{(b+1)(b+2)} \intk{1}(\varepsilon) + \varepsilon^2\right\}
    \sim  \frac{2 \, b\, \sigma^2}{(b+1)(b+2)} \intk{1}(\varepsilon_\sigma) + \varepsilon^2_\sigma,
    \quad \text{as } \sigma\to 0.
\end{equation}
Invoking Karamata's theorem (see \cite[Theorem~1.5.11]{binghamRegularVariation1987} and Lemma~\ref{lem:karamata-around-zero} in Appendix~\ref{sec:reg-var-zero}),
\eqref{eq:IA-asympt-neg-binft} guarantees that $\ecv$ and $\intk{1}$ are regularly varying at zero with index $b$.
A second application of Karamata's theorem---this time to $\intk{1}$---thus gives
\begin{equation}\label{eq:first-application-Karamata}
  \int_{\dumvar}^{\infty} \frac{\intk{1}(\dumvartwo)}{\dumvartwo^2}\diff \dumvartwo
  \sim 
  \frac{\intk{1}(\dumvar)}{(b+1)\, \dumvar}, 
  \quad\text{as } \dumvar\to 0.
\end{equation}
In particular, plugging the asymptotic scaling \eqref{eq:first-application-Karamata} in \eqref{eq:integral-of-integral-of-integral} yields
\begin{equation}\label{eq:result-first-Karamata}
  \intk{2}(\varepsilon)-\intk{3}(\varepsilon) \, \varepsilon
  \sim \varepsilon \int_{\varepsilon}^{\infty}\frac{b\, \intk{1}(\dumvar)}{(b+1)\, \dumvar^3} \diff \dumvar, 
  \quad\text{as } \varepsilon\to 0.
\end{equation}
Applying Karamata's theorem one more time, 
we further get
\begin{equation}\label{eq:result-first-Karamata-2}
\int_{\varepsilon}^{\infty}\frac{\intk{1}(\dumvar)}{\dumvar^3} \diff \dumvar
  \sim  \frac{\intk{1}(\varepsilon)\, \varepsilon^{-2}}{b+2}, 
  \quad\text{as } \varepsilon\to 0.
\end{equation}
Finally, combining \eqref{eq:result-first-Karamata} with \eqref{eq:result-first-Karamata-2}
gives
\begin{equation*}
  \intk{2}(\varepsilon)-\intk{3}(\varepsilon) \, \varepsilon
  \sim \frac{b\,\intk{1}(\varepsilon)\, \varepsilon^{-1}}{(b+1)(b+2)}, 
  \quad\text{as } \varepsilon\to 0,
\end{equation*}
which, when plugged in \eqref{eq:linear-risk-integral}, 
yields \eqref{eq:linear-risk-and-I1}.

Finally, we turn to the case $b=\infty$.
Since (RC) is satisfied by assumption, 
a direct application of Theorem~\ref{thm:metric-entropy-integral-form} guarantees that the metric entropy of $\ellip$ satisfies \eqref{eq:asymp-equivalence-ME-I1-first-order}.
By Lemmas~\ref{lem:abelian-integration} and \ref{lem:abelian-infimum} (see Appendix~\ref{sec:asymp-equiv-integration}),
it is thus sufficient to prove
\begin{equation}\label{eq:linear-risk-and-I1-infty}
    R_{\sigma}^L \left(\ellip\right)
    \sim  \newinf_{\varepsilon>0} \left\{ 2 \,  \sigma^2  \int_{\varepsilon}^{\infty}\frac{\intk{1}(\dumvar)}{\dumvar}\diff \dumvar + \varepsilon^2\right\}
    \sim 2 \,  \sigma^2  \int_{\varepsilon_\sigma}^{\infty}\frac{\intk{1}(\dumvar)}{\dumvar}\diff \dumvar + \varepsilon_\sigma^2,
    \quad \text{as } \sigma\to 0.
\end{equation}
The limit \eqref{eq:IA-asympt-neg-binft} with $b=\infty$ implies that $\intk{1}(\varepsilon)$ is asymptotically negligible compared to $\ecv(\varepsilon)$, as $\varepsilon\to 0$.
Consequently, l'H\^{o}pital's rule gives
\begin{equation}\label{eq:IA-asympt-neg-binft-2}
  \lim_{\dumvar\to 0} \frac{1}{\intk{1}(\dumvar)/\dumvar} \int_{\dumvar}^{\infty} \frac{\intk{1}(\dumvartwo)}{ \dumvartwo^2}\diff \dumvartwo
  = \lim_{\dumvar\to 0} \frac{\intk{1}(\dumvar)/\dumvar^2}{\ecv(\dumvar)/\dumvar^2+ \intk{1}(\dumvar)/\dumvar^2}
  = 0.
\end{equation}
We deduce from \eqref{eq:IA-asympt-neg-binft-2} that 
\begin{equation*}
  \int_{\dumvar}^{\infty} \frac{\intk{1}(\dumvartwo)}{\dumvar^2 \, \dumvartwo^2}\diff \dumvartwo
  \quadtext{is asymptotically negligible compared to}
  \frac{\intk{1}(\dumvar)}{\dumvar^3},
  \quad \text{as } \dumvar \to 0,
\end{equation*}
which, together with \eqref{eq:integral-of-integral-of-integral} and Lemma~\ref{lem:abelian-integration}, yields
\begin{equation}\label{eq:asymp-I2-I3-binft}
  \intk{2}(\varepsilon)-\intk{3}(\varepsilon) \, \varepsilon
  \sim \varepsilon \int_{\varepsilon}^{\infty}\frac{\intk{1}(\dumvar)}{\dumvar^3}\diff \dumvar, 
  \quad\text{as }\varepsilon\to 0.
\end{equation}
Next, introduce the auxiliary functions 
\begin{equation*}
  \psi_1(\varepsilon) \coloneqq \int_{\varepsilon}^{\infty}\frac{\intk{1}(\dumvar)}{\dumvar}\diff \dumvar,
  \quadtext{and} 
  \psi_3(\varepsilon)\coloneqq \int_{\varepsilon}^{\infty}\frac{\intk{1}(\dumvar)}{\dumvar^3}\diff \dumvar,
  \quad \text{for all } \varepsilon >0,
\end{equation*}
whose asymptotics can be related via l'H\^{o}pital's rule according to
\begin{equation}\label{eq:almost-there-binft}
  \lim_{\varepsilon\to 0}
  \frac{\psi_1(\varepsilon)}{\varepsilon^{2} \, \psi_3(\varepsilon)}
  = \lim_{\varepsilon\to 0}
  \frac{\psi_1'(\varepsilon)/\varepsilon^2 - 2\psi_1(\varepsilon)/\varepsilon^3}{\psi_3'(\varepsilon)}
  = 1+2\lim_{\varepsilon\to 0} \frac{\psi_1(\varepsilon)}{\intk{1}(\varepsilon)}.
\end{equation}
Applying once more l'H\^{o}pital's rule to the right-hand side of \eqref{eq:almost-there-binft},
we get 
\begin{equation}\label{eq:almost-there-binft2}
  \lim_{\varepsilon\to 0} \frac{\psi_1(\varepsilon)}{\intk{1}(\varepsilon)}
  = \lim_{\varepsilon\to 0} \frac{\intk{1}(\varepsilon)}{\ecv(\varepsilon)} = 0,
\end{equation}
where the last step uses the previous observation that $\intk{1}(\varepsilon)$ is asymptotically negligible when compared to $\ecv(\varepsilon)$, as $\varepsilon\to 0$.
Putting \eqref{eq:almost-there-binft} and \eqref{eq:almost-there-binft2} together yields 
\begin{equation*}
  \int_{\varepsilon}^{\infty}\frac{\intk{1}(\dumvar)}{\dumvar^3}\diff \dumvar
  =\psi_3(\varepsilon)
  \sim \varepsilon^{-2} \,  \psi_1(\varepsilon) 
  = \varepsilon^{-2} \,  \int_{\varepsilon}^{\infty}\frac{\intk{1}(\dumvar)}{\dumvar}\diff \dumvar,
  \quad\text{as } \varepsilon\to 0,
\end{equation*}
which allows us to rewrite \eqref{eq:asymp-I2-I3-binft} according to
\begin{equation}\label{eq:asymp-I2-I3-binft-updated}
  \intk{2}(\varepsilon)-\intk{3}(\varepsilon) \, \varepsilon
  \sim \varepsilon^{-1} \int_{\varepsilon}^{\infty}\frac{\intk{1}(\dumvar)}{\dumvar}\diff \dumvar, 
  \quad\text{as }\varepsilon\to 0.
\end{equation}
Plugging \eqref{eq:asymp-I2-I3-binft-updated} in \eqref{eq:linear-risk-integral}
establishes \eqref{eq:linear-risk-and-I1-infty}, thereby concluding the proof of Theorem~\ref{thm:bias-variance-decomp}.

\subsection{Proof of Theorem~\ref{thm:Metric-Entropy-Sobolev}}\label{sec:proof-theorem-Metric-Entropy-Sobolev}

Recall from the discussion following \eqref{eq:operator-T} that the semi-axes $\{\mu_n\}_{n\in\Ns}$ of the Sobolev ellipsoid are the eigenvalues of the operator $T$.
Furthermore, we see from the expression \eqref{eq:operator-T} that they are given by $\mu_n = \varphi(\lambda_n)$, for all $n\in \Ns$, where $\{\lambda_n\}_{n\in\Ns}$ are the (positive) eigenvalues of the Dirichlet Laplacian $-\Delta$,
and  $\varphi$ is the function defined as 
\begin{equation}\label{eq:varphi-and-its-inverse}
  \varphi \colon s\in \Rp \longmapsto \frac{1}{\sqrt{1+s^k}} \in (0,1),
  \quad \text{with inverse given by } \varphi^{-1}(u) = \left(u^{-2}-1\right)^{\frac{1}{k}},
\end{equation}
for all $u\in (0,1)$.
The eigenvalue-counting function $M_\lambda$ of the Laplacian and the semi-axis-counting function $\ecv$ of the Sobolev ellipsoid are therefore related according to
\begin{equation}\label{eq:relation-ecv-Laplace-ecv-T}
  M_\lambda (s) 
  \coloneqq \left| \left\{n \mid \lambda_n \leq s \right\}\right|
  = \left| \left\{n \mid \mu_n \geq \varphi(s) \right\}\right|
  = \ecv \left(\varphi(s)\right),
  \quad \text{for all } s>0.
\end{equation}
In particular, we deduce from the one-term version of the Weyl law \eqref{eq:Weyl-law-one-term} that 
\begin{equation}\label{eq:relation-ecv-Laplace-ecv-T-asymp}
  \ecv \left(u\right)
  = M_\lambda \left(\varphi^{-1}(u)\right) 
  \sim d\, \chi_d(\Omega) \, u^{-\frac{d}{k}},
  \quad \text{as } u \to 0,
\end{equation}
where we used \eqref{eq:varphi-and-its-inverse} to get $\varphi^{-1}(u) \sim u^{-2/k}$, as $u\to 0$. 
Moreover, the eigenvalue-counting function $M_\lambda$ can be related to the Riesz means appearing in \eqref{eq:asymptotics-Laplace-story} through the following Stieltjes integral
\begin{equation}\label{eq:Stieltjes-integral-Laplace}
  \sum_{n\in\Ns} \left(1-h^2\lambda_n\right)_+
  = \int_{0}^{h^{-2}}\left(1-h^2 s\right) \diff M_\lambda (s),
  \quad \text{for all } h >0.
\end{equation}
We can now successively integrate the left-hand side of \eqref{eq:Stieltjes-integral-Laplace} by parts and apply the relation \eqref{eq:relation-ecv-Laplace-ecv-T} to obtain
\begin{equation}\label{eq:sum-to-integral-Laplace}
  \sum_{n\in\Ns} \left(1-h^2\lambda_n\right)_+
  = h^2\int_{0}^{h^{-2}} M_\lambda (s) \diff s 
  = h^2\int_{0}^{h^{-2}} \ecv \left(\varphi(s)\right) \diff s,
  \quad \text{for all } h >0. 
\end{equation}
Upon changing variables in the right-hand side of \eqref{eq:sum-to-integral-Laplace}, by setting $\dumvar = \varphi(s)$, we get
\begin{equation}\label{eq:sum-to-integral-Laplace-2}
  \sum_{n\in\Ns} \left(1-h^2\lambda_n\right)_+
  = \frac{2h^2}{k} \int_{\varphi(h^{-2})}^{\infty} \frac{\ecv(\dumvar)}{\dumvar^3\left(\dumvar^{-2}-1\right)^{\frac{k-1}{k}}} \diff \dumvar,
  \quad \text{for all } h >0. 
\end{equation}
The relation \eqref{eq:sum-to-integral-Laplace-2} can be reformulated by letting $\varepsilon = \varphi(h^{-2})$, so that we arrive at the formula
\begin{equation}\label{eq:sum-to-integral-Laplace-3}
  \frac{k \varphi^{-1}(\varepsilon)}{2}
  \sum_{n\in\Ns} \left(1-\frac{\lambda_n}{\varphi^{-1}(\varepsilon)}\right)_+
  = \int_{\varepsilon}^{\infty} \frac{\ecv(\dumvar)}{\dumvar^3\left(\dumvar^{-2}-1\right)^{\frac{k-1}{k}}} \diff \dumvar,
  \quad \text{for all } \varepsilon >0. 
\end{equation}

We will treat both sides of \eqref{eq:sum-to-integral-Laplace-3} separately. 
Toward simplifying the right-hand side,
we first perform a Taylor expansion to get
\begin{equation*}
   \frac{1}{\dumvar^3\left(\dumvar^{-2}-1\right)^{\frac{k-1}{k}}}
   = \frac{1}{\dumvar^{1+\frac{2}{k}}\left(1-\dumvar^{2}\right)^{\frac{k-1}{k}}}
   = \frac{1}{\dumvar^{1+\frac{2}{k}}}\left(1+ O_{\dumvar\to 0}\left(\dumvar^{2}\right)\right)
   = \frac{1}{\dumvar^{1+\frac{2}{k}}}\left(1+ O_{\dumvar\to 0}\left(\dumvar^{\frac{2}{k}}\right)\right),
\end{equation*}
which, using that asymptotics are preserved under integration (see Lemma~\ref{lem:abelian-integration} in Appendix~\ref{sec:asymp-equiv-integration}), provides the asymptotic relation
\begin{equation}\label{eq:simplify-integral-Taylor}
  \int_{\varepsilon}^{\infty} \frac{\ecv(\dumvar)}{\dumvar^3\left(\dumvar^{-2}-1\right)^{\frac{k-1}{k}}} \diff \dumvar
  = \intk{1+\frac{2}{k}}(\varepsilon) + O_{\varepsilon\to 0}\left( \intk{1}(\varepsilon)\right).
\end{equation}
Upon application of Lemma~\ref{lem:integral-k-to-k+1}, 
with the choice $\typevar_1 = 1$ and $\typevar_2 = 1+{2}/{k}$, 
we can relate the integrals $\intk{1}$ and $\intk{1+\frac{2}{k}}$ according to
\begin{equation}\label{eq:relate-integrals-2-k-plus-2-k-minus-00}
  \intk{1}(\varepsilon)
  = \intk{1+\frac{2}{k}}(\varepsilon) \, \varepsilon^{\frac{2}{k}} + \frac{2}{k} \int_{\varepsilon}^{\infty} \intk{1+\frac{2}{k}}(\dumvar) \,  \dumvar^{\frac{2}{k}-1} \diff \dumvar,
  \quad \text{for all } \varepsilon >0.
\end{equation}
We know from \eqref{eq:relation-ecv-Laplace-ecv-T-asymp} that
$\ecv$ is regularly varying at zero with index $d/k$.
Therefore, a direct application of Karamata's theorem to $\ecv$ (see Lemma~\ref{lem:karamata-around-zero} in Appendix~\ref{sec:reg-var-zero}), with the choice $\typevar = 1+2/k$ and $b=d/k$, implies that $\intk{1+\frac{2}{k}}$ is regularly varying at zero with index $(d+2)/k$.
Applying once more Karamata's theorem, this time to $\intk{1+\frac{2}{k}}$ with $\typevar=1-2/k$ and $b=(d+2)/k$, yields
\begin{equation}\label{eq:relate-integrals-2-k-plus-2-k-minus-01}
  \int_{\varepsilon}^{\infty} \intk{1+\frac{2}{k}}(\dumvar) \,  \dumvar^{\frac{2}{k}-1} \diff \dumvar
  \sim \frac{k}{d}\intk{1+\frac{2}{k}}(\varepsilon) \, \varepsilon^{\frac{2}{k}},
  \quad \text{as } \varepsilon  \to 0.
\end{equation}
In particular, 
combining \eqref{eq:simplify-integral-Taylor}, \eqref{eq:relate-integrals-2-k-plus-2-k-minus-00}, and \eqref{eq:relate-integrals-2-k-plus-2-k-minus-01},
we have proven that 
\begin{equation}\label{eq:integral-many-terms-to-simple-integral}
  \int_{\varepsilon}^{\infty} \frac{\ecv(\dumvar)}{\dumvar^3\left(\dumvar^{-2}-1\right)^{\frac{k-1}{k}}} \diff \dumvar
  = \intk{1+\frac{2}{k}}(\varepsilon)\left(1 + O_{\varepsilon\to 0}\left( \varepsilon^{\frac{2}{k}}\right)\right).
\end{equation}

Now turning to the left-hand side of \eqref{eq:sum-to-integral-Laplace-3}, 
we deduce, from the expression of the inverse of $\varphi$ provided in \eqref{eq:varphi-and-its-inverse}, that 
\begin{equation}\label{eq:relation-h-and-vareps}
  \varphi^{-1}(\varepsilon)
  = \varepsilon^{-\frac{2}{k}}\left(1+O_{\varepsilon\to 0}\left(\varepsilon^2\right)\right).
\end{equation}
Furthermore, we can use the formula \eqref{eq:asymptotics-Laplace-story} to get 
\begin{equation}\label{eq:asymp-eps-sum-Laplace}
  \sum_{n\in\Ns} \left(1-\frac{\lambda_n}{\varphi^{-1}(\varepsilon)}\right)_+
  = C_1 \varphi^{-1}(\varepsilon)^{\frac{d}{2}} - C_2 \varphi^{-1}(\varepsilon)^{\frac{d-1}{2}}\left(1+o_{\varepsilon\to 0}\left(1\right)\right),
\end{equation}
where we introduced the constants 
\begin{equation}\label{eq:LT-to-rescaled-Hausdorff}
  C_1 \coloneqq \frac{2d}{d+2} \chi_d(\Omega)
  \quadtext{and}
  C_2 \coloneqq \frac{d-1}{2(d+1)} \chi_{d-1}(\partial\Omega).
\end{equation}
When combined with \eqref{eq:relation-h-and-vareps}, \eqref{eq:asymp-eps-sum-Laplace}
gives
\begin{align}
  \sum_{n\in\Ns} \left(1-\frac{\lambda_n}{\varphi^{-1}(\varepsilon)}\right)_+
  &= C_1 \varepsilon^{-\frac{d}{k}}\left(1+O_{\varepsilon\to 0}\left(\varepsilon^2\right)\right) - C_2 \varepsilon^{-\frac{d-1}{k}} + o_{\varepsilon\to 0}\left(\varepsilon^{-\frac{d-1}{k}}\right)\label{eq:asymp-eps-sum-Laplace-bis-AA} \\
  &= C_1 \varepsilon^{-\frac{d}{k}} - C_2 \varepsilon^{-\frac{d-1}{k}} + o_{\varepsilon\to 0}\left(\varepsilon^{-\frac{d-1}{k}}\right) + O_{\varepsilon\to 0}\left(\varepsilon^{-\frac{d-2k}{k}}\right) \nonumber \\
  &= C_1 \varepsilon^{-\frac{d}{k}} - C_2 \varepsilon^{-\frac{d-1}{k}} + o_{\varepsilon\to 0}\left(\varepsilon^{-\frac{d-1}{k}}\right),\label{eq:asymp-eps-sum-Laplace-bis-AB}
\end{align}
where we used $k\geq 1$ in the last step.
Putting  \eqref{eq:sum-to-integral-Laplace-3},  \eqref{eq:integral-many-terms-to-simple-integral}, \eqref{eq:relation-h-and-vareps}, and \eqref{eq:asymp-eps-sum-Laplace-bis-AA}--\eqref{eq:asymp-eps-sum-Laplace-bis-AB} together yields the relation
\begin{equation*}
  \frac{k\varepsilon^{-\frac{2}{k}}}{2}\left(1+O_{\varepsilon\to 0}\left(\varepsilon^2\right)\right)
  \left( C_1 \varepsilon^{-\frac{d}{k}} - C_2 \varepsilon^{-\frac{d-1}{k}} + o_{\varepsilon\to 0}\left(\varepsilon^{-\frac{d-1}{k}}\right)\right)
  = \intk{1+\frac{2}{k}}(\varepsilon)\left(1 + O_{\varepsilon\to 0}\left( \varepsilon^{\frac{2}{k}}\right)\right).
\end{equation*}
Arguments analogous to the ones used in \eqref{eq:asymp-eps-sum-Laplace-bis-AA}--\eqref{eq:asymp-eps-sum-Laplace-bis-AB} guarantee that the $O_{\varepsilon\to 0}(\cdot)$-terms can be absorbed into the $o_{\varepsilon\to 0}(\cdot)$-term, which allows for the simplification
\begin{equation}\label{eq:asymptotics-integral-12k}
  \intk{1+\frac{2}{k}}(\varepsilon)
  = \frac{k\varepsilon^{-\frac{2}{k}}}{2} 
  \left(C_1\varepsilon^{-\frac{d}{k}}
  - C_2\varepsilon^{-\frac{d-1}{k}}
  + o_{\varepsilon\to 0}\left(\varepsilon^{-\frac{d-1}{k}} \right)
  \right).
\end{equation}

We next relate the type-$1$ integral $\intk{1}$ to the type-$(1+2/k)$ integral $\intk{1+\frac{2}{k}}$ via \eqref{eq:relate-integrals-2-k-plus-2-k-minus-00}.
From the asymptotics in \eqref{eq:asymptotics-integral-12k}, 
we deduce
\begin{equation}\label{eq:integral-expansion-asymp-1}
  \intk{1+\frac{2}{k}}(\varepsilon)\,  \varepsilon^{\frac{2}{k}}
   = \frac{C_1k}{2} \varepsilon^{-\frac{d}{k}}
   - \frac{C_2k}{2}\varepsilon^{-\frac{d-1}{k}}
  + o_{\varepsilon\to 0}\left(\varepsilon^{-\frac{d-1}{k}}\right),
\end{equation}
as well as 
\begin{align}
  \frac{2}{k} \int_{\varepsilon}^{\infty} \intk{1+\frac{2}{k}}(\dumvar)\, \dumvar^{\frac{2}{k} - 1} \diff \dumvar 
  &= C_1 \int_{\varepsilon}^{\infty}\dumvar^{-\frac{d}{k}-1}\diff \dumvar  - C_2  \int_{\varepsilon}^{\infty}\dumvar^{-\frac{d-1}{k}-1}\left(1 + o_{\dumvar\to 0}(1)\right) \diff \dumvar \label{eq:integral-expansion-asymp-2}  \\
  &= \frac{C_1k}{d} \varepsilon^{-\frac{d}{k}}
  - \frac{C_2k}{d-1}\varepsilon^{-\frac{d-1}{k}}
  + o_{\varepsilon\to 0}\left(\varepsilon^{-\frac{d-1}{k}}\right). \label{eq:integral-expansion-asymp-3}
\end{align}
We are left to insert the asymptotics \eqref{eq:integral-expansion-asymp-1} and \eqref{eq:integral-expansion-asymp-2}--\eqref{eq:integral-expansion-asymp-3}
in \eqref{eq:relate-integrals-2-k-plus-2-k-minus-00} to obtain the asymptotic behavior of $\intk{1}$ according to
\begin{align*}
  \intk{1}(\varepsilon)
  &= C_1\left(\frac{k}{2}+ \frac{k}{d}\right) \, \varepsilon^{-\frac{d}{k}}
  - C_2\left(\frac{k}{2}+ \frac{k}{d-1}\right) \, \varepsilon^{-\frac{d-1}{k}}
  + o_{\varepsilon\to 0}\left(\varepsilon^{-\frac{d-1}{k}}\right) \\
  &= \frac{C_1k(d+2)}{2d}\varepsilon^{-\frac{d}{k}}
  - \frac{C_2k(d+1)}{2(d-1)}\varepsilon^{-\frac{d-1}{k}}
  + o_{\varepsilon\to 0}\left(\varepsilon^{-\frac{d-1}{k}}\right).
\end{align*}
Using \eqref{eq:LT-to-rescaled-Hausdorff}, we get
\begin{equation}\label{eq:asymptotics-I1-0}
  \intk{1}(\varepsilon)
  = k\, \chi_d(\Omega) \, {\varepsilon}^{-\frac{d}{k}}  
  - \frac{k\, \chi_{d-1}(\partial \Omega)}{4}\, 
  \varepsilon^{-\frac{d-1}{k}} 
  + o_{\varepsilon \to 0}\left(\varepsilon^{-\frac{d-1}{k}}\right).
\end{equation}

The last step consists in relating the asymptotic behavior of $\intk{1}$ obtained in \eqref{eq:asymptotics-I1-0} to that of the metric entropy.
We intend to apply Theorem~\ref{thm:metric-entropy-integral-form}
and therefore need to show that (i) the error term in \eqref{eq:main-result-ME} is negligible compared to $\varepsilon^{-\frac{d-1}{k}}$, and (ii)  (RC) is verified.
We deduce from \eqref{eq:relation-ecv-Laplace-ecv-T-asymp} the asymptotic relation
\begin{equation*}
\sqrt{\ecv (\varepsilon)\ln\left(\ecv (\varepsilon)\right)\ln\left(\varepsilon^{-1}\right)}
  = o_{\varepsilon\to 0}\left(\varepsilon^{-\frac{d}{2k}-\eta}\right),
  \quad \text{for all } \eta>0.
\end{equation*}
Choosing $\eta = (d-2)/(2k)$---which is possible thanks to the assumption $d\geq 3$---we thus  get 
\begin{equation*}
  \sqrt{\ecv (\varepsilon)\ln\left(\ecv (\varepsilon)\right)\ln\left(\varepsilon^{-1}\right)}
  = o_{\varepsilon\to 0}\left(\varepsilon^{-\frac{d-1}{k}}\right),
\end{equation*}
which answers item (i).
Next, again using \eqref{eq:relation-ecv-Laplace-ecv-T-asymp}, 
we see that $\rcb$ is satisfied with $b=d/k$ by choosing $f\colon x \mapsto d\, \chi_d(\Omega) \, x^{-\frac{d}{k}}$, thereby answering item (ii).
We can therefore apply Theorem~\ref{thm:metric-entropy-integral-form} in \eqref{eq:asymptotics-I1-0}, yielding
\begin{equation*}
    H \left(\varepsilon; \mathcal{E}^{\text{Sob}}_{d, k}\right) 
    = \intk{1}(\varepsilon) + o_{\varepsilon\to 0}\left(\varepsilon^{-\frac{d-1}{k}}\right) 
  =k\, \chi_d(\Omega) \, {\varepsilon}^{-\frac{d}{k}}  
  - \frac{k\, \chi_{d-1}(\partial \Omega)}{4}\, 
  \varepsilon^{-\frac{d-1}{k}} 
  + o_{\varepsilon \to 0}\left(\varepsilon^{-\frac{d-1}{k}}\right),
\end{equation*}
which is the desired result

\subsection{Proof of Theorem~\ref{thm:pinsker-sobolev-order-two}}\label{sec:proof-pinsker-sob-order-two}

We first prove \eqref{eq:Pinsker-order-1}.
We get from \eqref{eq:relation-ecv-Laplace-ecv-T-asymp} that the semi-axis-counting function $\ecv$ of the Sobolev ellipsoid satisfies $\rcb$,  with $b=d/k$, so that we can apply Theorem~\ref{thm:bias-variance-decomp} to get
\begin{equation}\label{eq:asymp-risk-pinsker-order-1}
  R_{\sigma} \left(\mathcal{E}^{\text{Sob}}_{d, k}\right)
  \sim \newinf_{\varepsilon>0} \left\{ \frac{2 \, k \, d \, \sigma^2}{(d+k)(d+2k)} H \left(\varepsilon; \mathcal{E}^{\text{Sob}}_{d, k}\right)  + \varepsilon^2\right\}
  \sim \newinf_{\varepsilon>0} \Phi_\sigma(\varepsilon),
  \quad\text{as } \sigma\to 0, 
\end{equation}
where $\Phi_\sigma$ can be taken, using \eqref{eq:metric-entropy-sob-first-order}, as
\begin{equation}\label{eq:expression-phi-pinsker-sob-order-1}
  \Phi_\sigma (\varepsilon) \coloneqq \frac{2 \, k^2 \,  d  \, \chi_d(\Omega)\, \sigma^2}{(d+k)(d+2k)}   \, \varepsilon^{-\frac{d}{k}}+\varepsilon^2, \quad \text{for all } \varepsilon>0 \text{ and } \sigma>0.
\end{equation}
In order to find the minimum of $\Phi_\sigma$, we compute its derivative according to
\begin{equation}\label{eq:pinsker-sob-order-1-derivative}
    \Phi_\sigma'(\varepsilon) 
    = - \frac{2 \, k \,  d^2  \, \chi_d(\Omega)\, \sigma^2}{(d+k)(d+2k)}   \, \varepsilon^{-\frac{d}{k}-1}+2 \varepsilon, \quad \text{for all } \varepsilon>0\text{ and } \sigma>0.
\end{equation}
Setting the derivative in \eqref{eq:pinsker-sob-order-1-derivative} to zero, 
we obtain that $\Phi_\sigma$ is minimized at
\begin{equation}\label{eq:pinsker-order-1-sob-argmin}
    \varepsilon_\sigma^* \coloneqq \left( \frac{ k \,  d^2  \, \chi_d(\Omega)\, \sigma^2}{(d+k)(d+2k)}   \right)^{\frac{k}{d+2k}},
    \quad \text{for all } \sigma>0.
\end{equation}
Consequently, the minimum of $\Phi_\sigma$ can be computed by plugging \eqref{eq:pinsker-order-1-sob-argmin} in \eqref{eq:expression-phi-pinsker-sob-order-1} according to 
\begin{align}
  \newinf_{\varepsilon>0} \Phi(\varepsilon)
  &= \Phi(\varepsilon_\sigma^*) \label{eq:minimum-sobolev-pinsker-order1-1}\\
  &=  \frac{2 \, k^2 \,  d  \, \chi_d(\Omega)\, \sigma^2}{(d+k)(d+2k)}   \, \left( \frac{ k \,  d^2  \, \chi_d(\Omega)\, \sigma^2}{(d+k)(d+2k)}   \right)^{- \frac{d}{d+2k}}+\left( \frac{ k \,  d^2  \, \chi_d(\Omega)\, \sigma^2}{(d+k)(d+2k)}   \right)^{\frac{2k}{d+2k}} \nonumber\\
  &= \frac{d+2k}{d} \left(\frac{k \,  d^2 \,  \chi_d(\Omega)  \, \sigma^2}{(d+k)(d+2k)}\right)^{\frac{2k}{d+2k}},
  \quad \text{for all } \sigma>0.
  \label{eq:minimum-sobolev-pinsker-order1-2}
\end{align}
Combining \eqref{eq:asymp-risk-pinsker-order-1} and \eqref{eq:minimum-sobolev-pinsker-order1-1}--\eqref{eq:minimum-sobolev-pinsker-order1-2} yields the desired result.

We now proceed to establish \eqref{eq:Pinsker-order-2}.
We begin by applying Lemma~\ref{lem:integral-k-to-k+1},
first with the choice $\typevar_1=2$ and $\typevar_2=1+{2}/{k}$ to get 
\begin{equation}\label{eq:integral-2-vs-integral-k2}
  \intk{2}(\varepsilon)
  = \intk{1+\frac{2}{k}}(\varepsilon) \, \varepsilon^{\frac{2}{k}-1}
   + \left(\frac{2}{k}-1\right) \int_{\varepsilon}^{\infty} \intk{1+\frac{2}{k}}(\dumvar) \, \dumvar^{\frac{2}{k} - 2} \diff \dumvar,
   \quad \text{for all } \varepsilon >0,
\end{equation}
and then with the choice $\typevar_1=3$ and $\typevar_2=1+{2}/{k}$ to get
\begin{equation}\label{eq:integral-3-vs-integral-k2}
  \intk{3}(\varepsilon)
  = \intk{1+\frac{2}{k}}(\varepsilon) \, \varepsilon^{\frac{2}{k}-2}
   + \left(\frac{2}{k}-2\right) \int_{\varepsilon}^{\infty} \intk{1+\frac{2}{k}}(\dumvar)\, \dumvar^{\frac{2}{k} - 3} \diff \dumvar,
   \quad \text{for all } \varepsilon >0.
\end{equation}
We next plug the asymptotics \eqref{eq:asymptotics-integral-12k} in  \eqref{eq:integral-2-vs-integral-k2} to derive the asymptotic behavior of $\intk{2}$.
Specifically,
we obtain 
\begin{align}
  \intk{2}(\varepsilon)
  &= \frac{k\varepsilon^{-1}}{2} 
  \left( C_1\varepsilon^{-\frac{d}{k}}
  - C_2\varepsilon^{-\frac{d-1}{k}}
  + o_{\varepsilon\to 0}\left(\varepsilon^{-\frac{d-1}{k}} \right)
  \right) \label{eq:sobolev-derivation-I2-1}\\
  &\quad \quad + \left(\frac{2}{k}-1\right) \frac{k}{2} \int_{\varepsilon}^{\infty} C_1\dumvar^{-\frac{d}{k}-2}
  - C_2\dumvar^{-\frac{d-1}{k}-2}\left(1+
  + o_{\dumvar\to 0}\left(1 \right)\right)
   \diff \dumvar \nonumber\\
   &= \frac{k\varepsilon^{-1}}{2} 
  \left(C_1\varepsilon^{-\frac{d}{k}}
  - C_2\varepsilon^{-\frac{d-1}{k}}
  + o_{\varepsilon\to 0}\left(\varepsilon^{-\frac{d-1}{k}} \right)
  \right) \nonumber\\
  &\quad \quad + \left(\frac{2}{k}-1\right) \frac{k}{2} \left(C_1\frac{\varepsilon^{-\frac{d}{k}-1}}{1+\frac{d}{k}}
  - C_2\frac{\varepsilon^{-\frac{d-1}{k}-1}}{1+\frac{d-1}{k}}
  + o_{\varepsilon\to 0}\left(\varepsilon^{-\frac{d-1}{k}-1} \right)\right) \nonumber\\
  &= \frac{C_1k}{2} \left(1+ \left(\frac{2}{k}-1\right) \frac{k}{d+k} \right) \varepsilon^{-\frac{d}{k}-1} \nonumber \\
  &\quad \quad -\frac{C_2k}{2} \left(1+ \left(\frac{2}{k}-1\right) \frac{k}{d+k-1} \right) {\varepsilon^{-\frac{d-1}{k}-1}} + o_{\varepsilon\to 0}\left(\varepsilon^{-\frac{d-1}{k}-1} \right),\label{eq:sobolev-derivation-I2-2}
\end{align}
where we recall that the constants $C_1$ and $C_2$ have been introduced in \eqref{eq:LT-to-rescaled-Hausdorff}.
After simplification, \eqref{eq:sobolev-derivation-I2-1}--\eqref{eq:sobolev-derivation-I2-2} becomes
\begin{equation}\label{eq:integral-2-asymptotics-Sobolev}
  \intk{2}(\varepsilon)
  = \frac{C_1 k (d+2 )}{2(d+k)}\varepsilon^{-\frac{d}{k}-1}
  -\frac{C_2 k(d+1)}{2(d+k-1)} \varepsilon^{-\frac{d-1}{k}-1} + o_{\varepsilon\to 0}\left(\varepsilon^{-\frac{d-1}{k}-1} \right).
\end{equation}
A similar approach, relying on \eqref{eq:integral-3-vs-integral-k2} instead of \eqref{eq:integral-2-vs-integral-k2}, yields 
\begin{equation}\label{eq:integral-3-asymptotics-Sobolev}
  \intk{3}(\varepsilon)
  = \frac{C_1 k (d+2 )}{2(d+2k)} \varepsilon^{-\frac{d}{k}-2}
  -\frac{C_2 k(d+1)}{2(d+2k-1)} {\varepsilon^{-\frac{d-1}{k}-2}} + o_{\varepsilon\to 0}\left(\varepsilon^{-\frac{d-1}{k}-2} \right).
\end{equation}
Putting the asymptotic relations obtained in \eqref{eq:integral-2-asymptotics-Sobolev} and \eqref{eq:integral-3-asymptotics-Sobolev} together implies
\begin{align}
  2\intk{3}(\varepsilon) - \frac{\intk{2}(\varepsilon)}{\varepsilon}
  &= \frac{C_1 k (d+2 )}{2} \left(\frac{2}{d+2k}-\frac{1}{d+k}\right) \, \varepsilon^{-\frac{d}{k}-2}\label{eq:difference-integrals-Sobolev-1}\\
  & \quad \quad -\frac{C_2 k(d+1)}{2} \left(\frac{2}{d+2k-1}-\frac{1}{d+k-1}\right) \, {\varepsilon^{-\frac{d-1}{k}-2}}  + o_{\varepsilon\to 0}\left(\varepsilon^{-\frac{d-1}{k}-2} \right) \nonumber\\
  & = \frac{C_1k d (d+2 )}{2(d+k)(d+2k)}  \varepsilon^{-\frac{d}{k}-2}\nonumber\\
  & \quad \quad -\frac{C_2 k (d-1) (d+1 )}{2(d+k-1)(d+2k-1)} \varepsilon^{-\frac{d-1}{k}-2}  + o_{\varepsilon\to 0}\left(\varepsilon^{-\frac{d-1}{k}-2} \right).\label{eq:difference-integrals-Sobolev-2}
\end{align}
Next, we introduce the exponents $\beta_1 \coloneqq 2+d/k$ and $\beta_2\coloneqq 2+(d-1)/k$ 
as well as the constants
\begin{equation}\label{eq:define-kappa}
  \kappa 
  \coloneq \frac{C_1k d (d+2 )}{2(d+k)(d+2k)}
  \quadtext{and}
  \kappa' 
  \coloneqq \frac{C_2 k (d-1) (d+1 )}{2(d+k-1)(d+2k-1)},
\end{equation}
so that we can rewrite \eqref{eq:difference-integrals-Sobolev-1}--\eqref{eq:difference-integrals-Sobolev-2} according to
\begin{equation}\label{eq:asymp-int-before-crit-radius}
  2\intk{3}(\varepsilon) - \frac{\intk{2}(\varepsilon)}{\varepsilon}
  = \kappa \varepsilon^{-\beta_1}
  - \kappa' \varepsilon^{-\beta_2}
  + o_{\varepsilon\to 0}\left(\varepsilon^{-\beta_2} \right).
\end{equation}
Since, for all $\sigma>0$, the critical radius $\varepsilon_{\sigma}$ is the solution of \eqref{eq:definition-critical-radius}, 
we can deduce an asymptotic characterization of $\varepsilon_{\sigma}$ from \eqref{eq:asymp-int-before-crit-radius} according to
\begin{equation}\label{eq:def-critical-radius-Sobolev}
  \sigma^{-2} 
  = 2\intk{3}(\varepsilon_\sigma) - \frac{\intk{2}(\varepsilon_\sigma)}{\varepsilon_\sigma}
  = \kappa \varepsilon_{\sigma}^{-\beta_1}
  - \kappa' \varepsilon_{\sigma}^{-\beta_2}
  + o_{\sigma\to 0}\left(\varepsilon_{\sigma}^{-\beta_2} \right).
\end{equation}
Inverting  \eqref{eq:def-critical-radius-Sobolev} (using, for instance, \cite[Lemma~9]{ALLARD2025101762}), we obtain the asymptotic behavior of $\varepsilon_\sigma$ according to
\begin{align}
  \varepsilon_{\sigma}
  &= \left(\kappa\sigma^2\right)^{\frac{1}{\beta_1}}
  - \frac{\kappa'}{\kappa \beta_1} \left(\kappa\sigma^2\right)^{\frac{1+\beta_1-\beta_2}{\beta_1}}
  + o_{\sigma\to 0} \left(\sigma^{\frac{2+2\beta_1 -2\beta_2}{\beta_1}} \right) \label{eq:epsn-to-sigman}\\
  &= \left(\kappa\sigma^2\right)^{\frac{k}{d+2k}}
  - \frac{\kappa'k}{\kappa (d+2k)} \left(\kappa\sigma^2\right)^{\frac{k+1}{d+2k}}
  + o_{\sigma\to 0} \left(\sigma^{\frac{2(k+1)}{d+2k}} \right).\label{eq:epsn-to-sigman-2}
\end{align}
Now, recall from Theorem~\ref{thm:linear-risk-formula} that the linear risk $R_{\sigma}^L$ is determined by the critical radius $\varepsilon_\sigma$ via the formula
\begin{align}
  R_{\sigma}^L \left(\mathcal{E}^{\text{Sob}}_{d, k}\right)
     &= \sigma^2 \varepsilon_{\sigma} \intk{2}(\varepsilon_{\sigma}) \label{eq:linear-risk-Sobolev-epsn-1} \\
     &= \sigma^2 \left(\frac{C_1 k (d+2 )}{2(d+k)} \varepsilon_{\sigma}^{-\frac{d}{k}}
  -\frac{C_2k(d+1)}{2(d+k-1)}  \varepsilon_{\sigma}^{-\frac{d-1}{k}} + o_{\sigma\to 0}\left(\varepsilon^{-\frac{d-1}{k}} \right)\right),\label{eq:linear-risk-Sobolev-epsn-2}
\end{align}
where the second step relies on \eqref{eq:integral-2-asymptotics-Sobolev}.
Using \eqref{eq:epsn-to-sigman}--\eqref{eq:epsn-to-sigman-2}, 
we can expand the occurrences of $\varepsilon_{\sigma}$ in the right-hand side of \eqref{eq:linear-risk-Sobolev-epsn-1}--\eqref{eq:linear-risk-Sobolev-epsn-2} according to
\begin{align}
   \varepsilon_{\sigma}^{-\frac{d}{k}}
   &= \left(\kappa\sigma^2\right)^{-\frac{d}{d+2k}}
  \left(1- \frac{\kappa'k}{\kappa (d+2k)} \left(\kappa\sigma^2\right)^{\frac{1}{d+2k}}
  + o_{\sigma\to 0} \left(\sigma^{\frac{2}{d+2k}} \right)\right)^{-\frac{d}{k}}\label{eq:power-epsn-1}\\
  &= \left(\kappa\sigma^2\right)^{-\frac{d}{d+2k}}
  + \frac{\kappa'd}{\kappa (d+2k)} \left(\kappa\sigma^2\right)^{-\frac{d-1}{d+2k}}
  + o_{\sigma\to 0} \left(\sigma^{-\frac{2(d-1)}{d+2k}} \right),\label{eq:power-epsn-2}
\end{align}
and 
\begin{equation}\label{eq:power-epsn-3}
   \varepsilon_{\sigma}^{-\frac{d-1}{k}}
   = \left(\kappa\sigma^2\right)^{-\frac{d-1}{d+2k}} + o_{\sigma\to 0} \left(\sigma^{-\frac{2(d-1)}{d+2k}} \right).
\end{equation}
Injecting \eqref{eq:power-epsn-1}--\eqref{eq:power-epsn-2} and \eqref{eq:power-epsn-3}
in \eqref{eq:linear-risk-Sobolev-epsn-1}--\eqref{eq:linear-risk-Sobolev-epsn-2} yields
\begin{align}
  R_{\sigma}^L \left(\mathcal{E}^{\text{Sob}}_{d, k}\right)
     &= \sigma^2 \Biggl(
      \frac{C_1 k (d+2 )}{2(d+k)}\left(\kappa\sigma^2\right)^{-\frac{d}{d+2k}} \label{eq:final-step-risk-sob-1} \\
      &\quad  +\left( \frac{C_1 \kappa'\, k \, d (d+2 )}{2\kappa (d+k) (d+2k)} -\frac{C_2k(d+1)}{2(d+k-1)} \right) \left(\kappa\sigma^2\right)^{-\frac{d-1}{d+2k}} + o_{\sigma\to 0} \left(\sigma^{-\frac{2(d-1)}{d+2k}} \right)
     \Biggr)\nonumber\\
      &=\frac{C_1 k (d+2 )}{2\kappa (d+k)}\left(\kappa\sigma^2\right)^{\frac{2k}{d+2k}} \label{eq:final-step-risk-sob-2}\\
      &\quad  +\left( \frac{C_1 \kappa'\, k \,  d (d+2 )}{2\kappa^2 (d+k) (d+2k)} -\frac{C_2k(d+1)}{2\kappa(d+k-1)} \right) \left(\kappa\sigma^2\right)^{\frac{2k+1}{d+2k}} + o_{\sigma\to 0} \left(\sigma^{\frac{4k+2}{d+2k}} \right).\nonumber
\end{align}
From the definitions of  $\kappa$ and $\kappa'$ in \eqref{eq:define-kappa}, 
we verify that
\begin{equation}\label{eq:final-step-risk-sob-3}
  \frac{C_1 k (d+2 )}{2\kappa (d+k)} = \frac{d+2k}{d},
\end{equation}
and 
\begin{equation}\label{eq:final-step-risk-sob-4}
  \frac{C_1 \kappa'\, k \, d (d+2 )}{2\kappa^2 (d+k) (d+2k)} -\frac{C_2k(d+1)}{2\kappa(d+k-1)}
  = \frac{\kappa'}{\kappa}\left(1-\frac{d+2k-1}{d-1}\right)= - \frac{2\kappa'k}{\kappa(d-1)}.
\end{equation}
The ratio $\kappa'/\kappa$ can be further developed, using  \eqref{eq:LT-to-rescaled-Hausdorff} and  \eqref{eq:define-kappa}, according to
\begin{equation}\label{eq:final-step-risk-sob-5}
  \frac{\kappa'}{\kappa}
  = \frac{C_2(d-1)(d+1)(d+k)(d+2k)}{C_1 d (d+2) (d+k-1)(d+2k-1)}
  = \frac{(d-1)^2(d+k)(d+2k)\chi_{d-1}(\partial\Omega)}{4 d^2 (d+k-1)(d+2k-1)\chi_d(\Omega)}.
\end{equation}
Putting \eqref{eq:final-step-risk-sob-1}--\eqref{eq:final-step-risk-sob-2}, \eqref{eq:final-step-risk-sob-3}, \eqref{eq:final-step-risk-sob-4}, and \eqref{eq:final-step-risk-sob-5} together, we get 
\begin{equation}\label{eq:linear-risk-sob-final}
    R_{\sigma}^L \left(\mathcal{E}^{\text{Sob}}_{d, k}\right)
    = K_1 \left(\kappa\sigma^2\right)^{\frac{2k}{d+2k}}
    + K_2 \left(\kappa\sigma^2\right)^{\frac{2k+1}{d+2k}} + o_{\sigma\to 0} \left(\sigma^{\frac{4k+2}{d+2k}} \right),
\end{equation}
where 
\begin{equation*}
  \kappa 
  = \frac{C_1k \,  d (d+2 )}{2(d+k)(d+2k)}
  = \frac{k \, d^2 \, \chi_d(\Omega)}{(d+k)(d+2k)}
\end{equation*}
and
\begin{equation*}
  K_1
  = \frac{C_1 k (d+2 )}{2\kappa (d+k)} = \frac{d+2k}{d}
  \quadtext{and}
  K_2 
  = - \frac{k(d-1)(d+k)(d+2k)\chi_{d-1}(\partial\Omega)}{2 d^2 (d+k-1)(d+2k-1)\chi_d(\Omega)}.
\end{equation*}

The asymptotic characterization \eqref{eq:linear-risk-sob-final} holds for the linear minimax risk.
In order to deduce a corresponding characterization for the non-linear minimax risk using \eqref{eq:asymp-linear-to-nonlinear-risk}, 
it is sufficient to show that 
\begin{equation}\label{eq:to-prove-end-pinsker-sobb}
  \frac{\sqrt{\ln(v_\sigma)}}{v_\sigma}
  = o_{\sigma\to 0} \left(\sigma^{\frac{4k+2}{d+2k}} \right),
  \quadtext{where} v_\sigma \coloneqq \frac{\varepsilon_\sigma}{\sigma R^L_{\sigma}\left(\ellip\right)}.
\end{equation}
To this end, we use \eqref{eq:epsn-to-sigman}--\eqref{eq:epsn-to-sigman-2} and \eqref{eq:linear-risk-sob-final} to obtain the respective asymptotics
\begin{equation*}
  \varepsilon_\sigma 
  \stackrel{K}{\sim} \sigma^{\frac{2k}{d+2k}}
  \quadtext{and} 
  R^L_{\sigma}\left(\ellip\right) 
  \stackrel{K}{\sim} \sigma^{\frac{4k}{d+2k}},
  \quad\text{as } \sigma\to 0.
\end{equation*}
They imply
\begin{equation}\label{eq:asymp-scale-vsigmm}
  v_\sigma 
  = \frac{\varepsilon_\sigma}{\sigma R^L_{\sigma}\left(\ellip\right)}
  \stackrel{K}{\sim} \sigma^{-\frac{d+4k}{d+2k}},
  \quad\text{as } \sigma\to 0.
\end{equation}
Together, the assumption $d\geq 3$ and the asymptotic scaling \eqref{eq:asymp-scale-vsigmm} yield
\begin{equation*}
    \frac{\sqrt{\ln(v_\sigma)}}{v_\sigma}
  = o_{\sigma\to 0} \left(\sigma^{\frac{4k+2}{d+2k}} \right),
\end{equation*}
which establishes the desired result \eqref{eq:to-prove-end-pinsker-sobb}, thereby finishing the proof of \eqref{eq:Pinsker-order-2}.

\section{Auxiliary Results}

\subsection{Regular Variation at Zero and Karamata's Theorem}\label{sec:reg-var-zero}

In many places in the paper we use tools from the theory of regular variation 
(see, e.g., Theorem~\ref{thm:bias-variance-decomp} and the end of Appendix~\ref{sec:proof-metric-entropy-integral-form}).
The classical theory is formulated for functions whose argument tends to infinity \cite{binghamRegularVariation1987,korevaar2004tauberian}, 
whereas we are mostly interested in variables tending to zero.
Fortunately, statements `at infinity' translate easily to statements `at zero'
via the change of variables $s=1/\varepsilon$; 
the present appendix illustrates this claim on Karamata's theorem.

Specifically, given $b\in\R$, we say that a function $f\colon \Rp\to\R_+$ is regularly varying at zero with index $b$ if 
\begin{equation*}
  \lim_{\varepsilon\to 0}\frac{f(\lambda\varepsilon)}{f(\varepsilon)} = \lambda^{-b},
  \quad \text{for all } \lambda>0.
\end{equation*}
Equivalently, the map $s \mapsto f(s^{-1})$ is regularly varying (at infinity) with index $b$ in the sense of \cite[Definition~1.4.2]{binghamRegularVariation1987}.
Additionally, $f$ is slowly varying at zero if it is regularly varying at zero with index $b=0$.
We can then reformulate Karamata's theorem at zero as follows (see \cite[Theorems~1.5.11 and 1.6.1]{binghamRegularVariation1987} for the statement at infinity).

\begin{lemma}\label{lem:karamata-around-zero}
  Let $b\in\R$, let $\typevar > 1-b$, and let $f\colon \Rp \to \R_+$ be  bounded on $(\varepsilon, \infty)$, for all $\varepsilon>0$.
  Then, $f$ is regularly varying at zero with index $b$ if and only if
  \begin{equation}\label{eq:karamata-around-zero}
    \int_{\varepsilon}^{\infty} \frac{f(\dumvar)}{u^{\typevar}} \diff \dumvar
    \sim \frac{\varepsilon^{1-\typevar}}{b+\typevar-1}f(\varepsilon),
    \quad\text{as } \varepsilon \to 0.
  \end{equation}
\end{lemma}

\begin{proof}
  Let us first perform the change of variables $\dumvartwo=\dumvar^{-1}$ in the left-hand side of \eqref{eq:karamata-around-zero} to get 
  \begin{equation}\label{eq:proof-karamata-around-zero-1}
    \int_{\varepsilon}^{\infty} \frac{f(\dumvar)}{\dumvar^{\typevar}} \diff \dumvar
    = \int_{0}^{s(\varepsilon)} f\left(\dumvartwo^{-1}\right)\,  \dumvartwo^{\typevar-2} \diff \dumvartwo,
    \quad \text{for all } \varepsilon >0,
  \end{equation}
  where $s(\varepsilon)\coloneqq \varepsilon^{-1}$ tends to $\infty$ as $\varepsilon\to 0$.
  The function $f$ being regularly varying at zero with index $b$ is equivalent to the map $s \mapsto f(s^{-1})$ being regularly varying at infinity, also with index $b$.
  Additionally, one gets, from the  corresponding assumption on $f$, that $s \mapsto f(s^{-1})$ is locally bounded---therefore also locally integrable---on $[0, \infty)$, so that the conditions of the classical version of Karamata's theorem \cite[Theorems~1.5.11~(i) and 1.6.1~(i)]{binghamRegularVariation1987} are satisfied.
  Upon its application, one gets that $s \mapsto f(s^{-1})$ is regularly varying at infinity with index $b$ if and only if
  \begin{equation}\label{eq:proof-karamata-around-zero-2}
    \int_{0}^{s(\varepsilon)} f\left(\dumvartwo^{-1}\right)\,  \dumvartwo^{\typevar-2} \diff \dumvartwo
    \sim \frac{s(\varepsilon)^{\typevar-1}}{b+\typevar-1}f\left(s(\varepsilon)^{-1}\right),
    \quad\text{as } s(\varepsilon) \to \infty.
  \end{equation}
  Replacing $s(\varepsilon)$ by $\varepsilon^{-1}$ in \eqref{eq:proof-karamata-around-zero-2} and combining with \eqref{eq:proof-karamata-around-zero-1} yields the desired result.
\end{proof}

\noindent
Regularly varying functions appear throughout the paper because item (i) in (RC) is equivalent to $\ecv$ being regularly varying at zero.
Indeed, if $\ecv$ is regularly varying at zero with index $b\in [0,\infty)$,
the smooth variation theorem applied to the regularly varying function $\ecv$ (see \cite[Definition~1.8 and Theorem~1.8.2]{binghamRegularVariation1987}) guarantees the existence of $f \in C^1(\Rp, \Rp)$ satisfying $ \ecv (x) \sim f(x)$ as $x\to 0$ and whose elasticity $\rho$ has the limit $\rho(t)\to b$ as $t\to\infty$. 
This is precisely $\rcb$.
Conversely, arguments as in \eqref{eq:chain-rule-elasticity}--\eqref{eq:application-lhopital-1} yield
\begin{equation*}
  \lim_{\varepsilon\to 0} \frac{f(\varepsilon)}{\int_{\varepsilon}^{\infty} f(\dumvar)\, \dumvar^{-1}\diff\dumvar} 
  = \lim_{\varepsilon\to 0} \rho\left(\ln\left(\varepsilon^{-1}\right)\right)
  \eqqcolon b,
\end{equation*}
from which we conclude---using Lemma~\ref{lem:karamata-around-zero} with $\typevar=1$---that $\ecv$ is indeed regularly varying at zero with index $b$.

\subsection{Asymptotic Equivalence and Integration}\label{sec:asymp-equiv-integration}

Throughout the paper, we have repeatedly integrated asymptotic equivalences.
This is justified by the following lemma.

\begin{lemma}\label{lem:abelian-integration}
  Let $f_1, f_2 \colon \Rp \to \R_+$ be integrable on $(\varepsilon, \infty)$,
  for all $\varepsilon>0$. 
  If $f_1(\varepsilon) \sim f_2(\varepsilon)$ and $\int_{\varepsilon}^{\infty}f_1(\dumvar)\diff\dumvar\to\infty$, as $\varepsilon\to 0$,
  then 
  \begin{equation*}
    \int_{\varepsilon}^{\infty} f_1(\dumvar) \diff\dumvar 
    \sim \int_{\varepsilon}^{\infty} f_2(\dumvar) \diff\dumvar, 
    \quad\text{as } \varepsilon\to 0. 
  \end{equation*}
\end{lemma}

\begin{proof}
  Let us fix $\delta>0$. 
  The asymptotic equivalence $f_1(\varepsilon) \sim f_2(\varepsilon)$, as $\varepsilon\to 0$, guarantees the existence of $\varepsilon_\delta>0$ such that 
  \begin{equation}\label{eq:asymp-to-integral-proof-1}
    (1-\delta) f_1(\varepsilon)
    \leq f_2(\varepsilon) 
    \leq (1+\delta) f_1(\varepsilon),
    \quad\text{for all } \varepsilon \in (0,\varepsilon_\delta).
  \end{equation}
  Integrating \eqref{eq:asymp-to-integral-proof-1} yields
  \begin{equation}\label{eq:asymp-to-integral-proof-2}
    (1-\delta) \int_{\varepsilon}^{\varepsilon_\delta}f_1(\dumvar)\diff\dumvar
    \leq \int_{\varepsilon}^{\varepsilon_\delta}f_2(\dumvar)\diff\dumvar
    \leq (1+\delta) \int_{\varepsilon}^{\varepsilon_\delta}f_1(\dumvar)\diff\dumvar,
    \quad\text{for all } \varepsilon \in (0,\varepsilon_\delta).
  \end{equation}
  Moreover, under the assumptions 
  \begin{equation*}
      \int_{\varepsilon}^{\infty}f_1(\dumvar)\diff\dumvar\to\infty,  
      \quad\text{as } \varepsilon\to 0, 
      \quadtext{and} \int_{\varepsilon_\delta}^{\infty}f_1(\dumvar)\diff\dumvar, 
      \int_{\varepsilon_\delta}^{\infty}f_2(\dumvar)\diff\dumvar <\infty,
  \end{equation*}
  we get 
  \begin{equation}\label{eq:asymp-to-integral-proof-1-25}
    \liminf_{\varepsilon\to 0} \cfrac{\int_{\varepsilon}^{\infty}f_2(\dumvar)\diff\dumvar}{\int_{\varepsilon}^{\infty}f_1(\dumvar)\diff\dumvar}
    =\liminf_{\varepsilon\to 0} \cfrac{\int_{\varepsilon}^{\varepsilon_\delta}f_2(\dumvar)\diff\dumvar+\int_{\varepsilon_\delta}^{\infty}f_2(\dumvar)\diff\dumvar}{\int_{\varepsilon}^{\varepsilon_\delta}f_1(\dumvar)\diff\dumvar+\int_{\varepsilon_\delta}^{\infty}f_1(\dumvar)\diff\dumvar}
    = \liminf_{\varepsilon\to 0} \cfrac{\int_{\varepsilon}^{\varepsilon_\delta}f_2(\dumvar)\diff\dumvar}{\int_{\varepsilon}^{\varepsilon_\delta}f_1(\dumvar)\diff\dumvar}.
  \end{equation}
  With a similar line of arguments, we also come to the conclusion that 
  \begin{equation}\label{eq:asymp-to-integral-proof-1-75}
    \limsup_{\varepsilon\to 0} \cfrac{\int_{\varepsilon}^{\infty}f_2(\dumvar)\diff\dumvar}{\int_{\varepsilon}^{\infty}f_1(\dumvar)\diff\dumvar}
    = \limsup_{\varepsilon\to 0} \cfrac{\int_{\varepsilon}^{\varepsilon_\delta}f_2(\dumvar)\diff\dumvar}{\int_{\varepsilon}^{\varepsilon_\delta}f_1(\dumvar)\diff\dumvar}.
  \end{equation}
  Combining \eqref{eq:asymp-to-integral-proof-1-25} and \eqref{eq:asymp-to-integral-proof-1-75} with \eqref{eq:asymp-to-integral-proof-2}, we have proven that
    \begin{equation}\label{eq:asymp-to-integral-proof-3}
    1-\delta
    \leq \liminf_{\varepsilon\to 0} \cfrac{\int_{\varepsilon}^{\infty}f_2(\dumvar)\diff\dumvar}{\int_{\varepsilon}^{\infty}f_1(\dumvar)\diff\dumvar}
    \leq \limsup_{\varepsilon\to 0} \cfrac{\int_{\varepsilon}^{\infty}f_2(\dumvar)\diff\dumvar}{\int_{\varepsilon}^{\infty}f_1(\dumvar)\diff\dumvar}
    \leq 1+\delta.
  \end{equation}
  The parameter $\delta$ can be taken arbitrarily small in \eqref{eq:asymp-to-integral-proof-3}, implying
  \begin{equation*}
    \lim_{\varepsilon\to 0} \cfrac{\int_{\varepsilon}^{\infty}f_2(\dumvar)\diff\dumvar}{\int_{\varepsilon}^{\infty}f_1(\dumvar)\diff\dumvar}
    =1,
  \end{equation*}
  and thus establishing the desired result.
\end{proof}

\noindent
Lemma~\ref{lem:abelian-integration} is particularly useful for us because it transfers asymptotics of $\ecv$ to asymptotics of $\intk{\typevar}$, for every $\typevar\geq 1$.
Specifically, suppose the semi-axis-counting function satisfies $\ecv(\varepsilon) \sim \psi(\varepsilon)$ as $\varepsilon\to 0$ for a given function $\psi$. 
Applying Lemma~\ref{lem:abelian-integration} with $f_1 \colon \dumvar \mapsto \ecv(\dumvar) / \dumvar^\typevar$ and $f_2 \colon \dumvar \mapsto \psi(\dumvar) / \dumvar^\typevar $ yields $\intk{\typevar}(\varepsilon) \sim \int_{\varepsilon}^{\infty}\psi(\dumvar) / \dumvar^\typevar \diff\dumvar$.
For example, if $\ecv(\varepsilon) \sim c \,  \varepsilon^{-\alpha}$, for $c, \alpha>0$, then
\begin{equation}\label{eq:ecv-implies-Itau-asymp}
  \intk{\typevar}(\varepsilon) 
  = \int_{\varepsilon}^{\infty} \frac{\ecv(\dumvar)}{\dumvar^\typevar} \diff \dumvar
  \sim \int_{\varepsilon}^{\infty} \frac{c \,  \dumvar^{-\alpha}}{\dumvar^\typevar} \diff \dumvar
  = \frac{c \, \varepsilon^{-\alpha-\typevar +1}}{\alpha+\typevar - 1},
  \quad \text{as } \varepsilon\to 0.
\end{equation}
Note that the infinite-limit condition of Lemma~\ref{lem:abelian-integration},
namely $\int_{\varepsilon}^{\infty}\ecv(\dumvar) / \dumvar^\typevar \diff\dumvar \to \infty$ as $\varepsilon\to 0$,
always comes for free.
Moreover, we emphasize that only the local properties around zero matter.
For instance, when deriving the metric entropy of $d$-dimensional ellipsoids, 
i.e., when $\ecv(\varepsilon) \sim d$ and $\typevar=1$, 
one wishes to integrate the map $u\mapsto d/u$, which is not integrable at infinity. 
However, it may be replaced by the truncated function $u\mapsto d \mathbbm{1}_{(0,1)}(u)/u$, which is integrable at infinity and has the same asymptotics at zero.
Applying Lemma~\ref{lem:abelian-integration} then yields the standard result \eqref{eq:finite-dim-ME-formula} according to 
\begin{equation*}
  \intk{1}(\varepsilon) 
  = \int_{\varepsilon}^{\infty} \frac{\ecv(\dumvar)}{\dumvar} \diff \dumvar
  \sim \int_{\varepsilon}^{1} \frac{d}{\dumvar} \diff \dumvar
  = d\ln\left(\varepsilon^{-1}\right),
  \quad \text{as } \varepsilon\to 0.
\end{equation*}

We next ask whether one can deduce the asymptotics of $\intk{\typevar_1}$ from those of $\intk{\typevar_2}$ when $\typevar_1\neq \typevar_2$.
From Lemma~\ref{lem:integral-k-to-k+1}, this is possible provided the leading terms on the right-hand side of \eqref{eq:lemma-integral-k-to-k+1} do not cancel out, 
i.e., unless we have 
\begin{equation}\label{eq:canceling-terms-integrals}
   \int_{\varepsilon}^{\infty} \intk{\typevar_2}(\dumvar)\, \dumvar^{\typevar_2-\typevar_1-1} \diff \dumvar
   \sim \frac{\varepsilon^{\typevar_2-\typevar_1}}{\typevar_1-\typevar_2} \intk{\typevar_2}(\varepsilon) ,
  \quad \text{as } \varepsilon\to 0.
\end{equation}
Such a cancelation can only happen when both sides of \eqref{eq:canceling-terms-integrals} are non-negative, 
which forces $\typevar_1>\typevar_2$.
Moreover, applying Lemma~\ref{lem:karamata-around-zero} with $b=0$ and $\typevar= 1 + \typevar_1 - \typevar_2>1$ shows that the scenario \eqref{eq:canceling-terms-integrals} can occur only if $\intk{\typevar_2}$ is slowly varying at zero.
In this case, l'H\^{o}pital's rule (see \cite[Theorem~5.13]{rudin1953principles}) applied to the functions $\varepsilon \mapsto \intk{\typevar_2}(\lambda \varepsilon)$ and $\varepsilon \mapsto \intk{\typevar_2}( \varepsilon)$, for $\lambda>0$, yields the limits
\begin{equation}\label{eq:lhopital-Itau2-to-ecv}
  1
  = \lim_{\varepsilon\to 0} \frac{\intk{\typevar_2}(\lambda \varepsilon)}{\intk{\typevar_2}(\varepsilon)}
  = \lim_{\varepsilon\to 0} \frac{-\lambda \ecv(\lambda \varepsilon)/(\lambda \varepsilon)^{\typevar_2}}{\ecv(\varepsilon)/\varepsilon^{\typevar_2}}
  = \lambda^{1-\typevar_2} \lim_{\varepsilon\to 0} \frac{\ecv(\lambda \varepsilon)}{\ecv(\varepsilon)}.
\end{equation}
In particular, \eqref{eq:lhopital-Itau2-to-ecv} implies that 
if $\intk{\typevar_2}$ is slowly varying at zero then $\ecv$ is regularly varying at zero with index $1-\typevar_2$.
If we had $\tau_2 > 1$, this index would be negative so that $\ecv(\varepsilon)\to 0$ as $\varepsilon\to 0$---a contradiction.
Hence, we must have $\tau_2 = 1$;
then, by \eqref{eq:lhopital-Itau2-to-ecv}, $\ecv$ is slowly varying.
In summary, the asymptotics of $\intk{\typevar_2}$ determine those of $\intk{\typevar_1}$ whenever either $\tau_2 > 1$ or $\ecv$ is not slowly varying.
Returning to example \eqref{eq:ecv-implies-Itau-asymp}, for all $c, \alpha>0$ and $\typevar_1\neq \typevar_2$, $\ecv$ is not slowly varying, and we obtain
\begin{equation*}
  \intk{\typevar_1}(\varepsilon) 
  \sim  \frac{c \, \varepsilon^{-\alpha-\typevar_1 +1}}{\alpha+\typevar_1 - 1},
  \quadtext{if and only if}
  \intk{\typevar_2}(\varepsilon) 
  \sim  \frac{c \, \varepsilon^{-\alpha-\typevar_2 +1}}{\alpha+\typevar_2 - 1},
  \quad \text{as } \varepsilon\to 0.
\end{equation*}

In general, the converse to Lemma~\ref{lem:abelian-integration} fails: 
we cannot recover the asymptotics of a function given that of its integral.
In particular, the asymptotic behavior of $\intk{\typevar}$ need not determine that of $\ecv$.
However, the converse does hold when the integral is regularly varying at zero and the integrand is eventually monotone.
This is the content of the next lemma.

\begin{lemma}\label{lem:tauberian-integration}
  Let $b>0$, let $f\colon \Rp\to\R$  be integrable and eventually monotone near zero, and let  $g\colon \Rp\to\R^*$ be regularly varying at zero with index $b$.
  If 
  \begin{equation}\label{eq:tauberian-karamata-integral-0}
    \int_{\varepsilon}^{\infty} f(\dumvar) \diff\dumvar 
    \sim g(\varepsilon) , 
    \quadtext{then} f(\varepsilon) \sim b\, \varepsilon^{-1} \, g(\varepsilon)
    \quad\text{as } \varepsilon \to 0. 
  \end{equation}
\end{lemma}

\begin{proof}
  We adapt the proof of the classical monotone density theorem to regular variations at zero; see \cite[Theorem~1.7.2]{binghamRegularVariation1987}.
  Let us first introduce the integral
  \begin{equation*}
    J(\varepsilon) 
    \coloneqq \int_{\varepsilon}^{\infty} f(\dumvar) \diff\dumvar,
    \quad \text{for all } \varepsilon>0.
  \end{equation*}
  For $c_2>c_1>0$, one can express the difference  between two evaluations of these integrals according to
  \begin{equation*}
    J(c_1\, \varepsilon) - J(c_2\, \varepsilon)
    = \int_{c_1\, \varepsilon}^{c_2\, \varepsilon} f(\dumvar) \diff\dumvar,
    \quad \text{for all } \varepsilon>0.
  \end{equation*}
  This difference can be upper- and lower-bounded for $\varepsilon$ small enough using the assumption that $f$ is eventually monotone near zero.
  (We assume that $f$ is eventually non-increasing, the non-decreasing case follows analogously.)
  Specifically, there exists $\varepsilon_*>0$  such that 
  \begin{equation}\label{eq:tauberian-karamata-integral-1}
    (c_2-c_1) \,\varepsilon \,f(c_2\, \varepsilon)
    \leq J(c_1\,\varepsilon) - J(c_2\,\varepsilon)
    \leq (c_2-c_1)\, \varepsilon\, f(c_1\varepsilon),
    \quad \text{for all } \varepsilon\in(0,\varepsilon_*).
  \end{equation}
  Moreover, we have  
  \begin{equation}\label{eq:tauberian-karamata-integral-2}
    \lim_{\varepsilon\to 0}\frac{J(c_1\,\varepsilon)}{g(\varepsilon)}
    =\lim_{\varepsilon\to 0}\frac{J(c_1\,\varepsilon)}{g(c_1\,\varepsilon)} \cdot \frac{g(c_1\,\varepsilon)}{g(\varepsilon)}
    = c_1^{-b},
  \end{equation}
  where we used the assumption $J(\varepsilon) \sim g(\varepsilon)$ from \eqref{eq:tauberian-karamata-integral-0} and the fact that $g$ is regularly varying at zero with index $b$.
  Similarly, we have 
  \begin{equation}\label{eq:tauberian-karamata-integral-3}
    \lim_{\varepsilon\to 0}\frac{J(c_2 \, \varepsilon)}{g(\varepsilon)}
    = c_2^{-b}.
  \end{equation}
  Combining  \eqref{eq:tauberian-karamata-integral-1} with the limits  \eqref{eq:tauberian-karamata-integral-2} and \eqref{eq:tauberian-karamata-integral-3} yields 
  \begin{equation}\label{eq:tauberian-karamata-integral-4}
    \limsup_{\varepsilon \to 0} \frac{\varepsilon f(c_2\, \varepsilon)}{g(\varepsilon)} 
    \leq \frac{ c_1^{-b}- c_2^{-b}}{c_2-c_1}
    \leq \liminf_{\varepsilon \to 0} \frac{\varepsilon f(c_1\, \varepsilon)}{g(\varepsilon)}.
  \end{equation}
  Choosing $c_2 = 1$ and taking the limit $c_1\to 1$, the lower bound in \eqref{eq:tauberian-karamata-integral-4} yields 
  \begin{equation}\label{eq:tauberian-karamata-integral-5}
    \limsup_{\varepsilon \to 0} \frac{\varepsilon f(\varepsilon)}{g(\varepsilon)} 
    \leq - \lim_{c_1\to 1}  \frac{ c_1^{-b}- 1}{c_1-1}
    = b.
  \end{equation}
  Likewise, choosing $c_1 = 1$ and taking the limit $c_2\to 1$, the upper bound in \eqref{eq:tauberian-karamata-integral-4} yields
  \begin{equation}\label{eq:tauberian-karamata-integral-6}
    \liminf_{\varepsilon \to 0} \frac{\varepsilon f(\varepsilon)}{g(\varepsilon)} \geq b.
  \end{equation}
  Combining \eqref{eq:tauberian-karamata-integral-5} and \eqref{eq:tauberian-karamata-integral-6} gives 
  \begin{equation*}
    \lim_{\varepsilon \to 0} \frac{\varepsilon f(\varepsilon)}{g(\varepsilon)} = b,
  \end{equation*}
  which is the desired result.
\end{proof}

\noindent
The hypothesis in Lemma~\ref{lem:tauberian-integration} that $f$ is eventually monotone near zero is not restrictive in our setting since we are mostly interested in the case $f=\ecv$.
The important assumption is regular variation of the integral.
It is for instance satisfied in the example \eqref{eq:ecv-implies-Itau-asymp}, so that
\begin{equation*}
    \intk{\typevar}(\varepsilon) 
  \sim  \frac{c \, \varepsilon^{-\alpha-\typevar +1}}{\alpha+\typevar - 1},
  \quad \text{as } \varepsilon\to 0,
  \quadtext{implies} \ecv(\varepsilon) \sim c \,  \varepsilon^{-\alpha},
\end{equation*}
for $c, \alpha>0$ and $\typevar\geq 1$.
In principle, the result can be extended to higher orders.
For instance, suppose 
\begin{equation*}
    \intk{\typevar}(\varepsilon) 
  = \frac{c_1 \, \varepsilon^{-\alpha_1-\typevar +1}}{\alpha_1+\typevar - 1}
  + \frac{c_2 \, \varepsilon^{-\alpha_2-\typevar +1}}{\alpha_2+\typevar - 1}
  + o_{\varepsilon\to 0}\left(\varepsilon^{-\alpha_2-\typevar +1}\right),
\end{equation*}
for $c_1, c_2>0$, $\alpha_1>\alpha_2>0$, and $\typevar\geq 1$.
Equivalently,
\begin{equation*}
    \int_{\varepsilon}^{\infty} \frac{\ecv(\dumvar)-c_1 \, \dumvar^{-\alpha_1}}{\dumvar^\typevar} \diff\dumvar
  =  \frac{c_2 \, \varepsilon^{-\alpha_2-\typevar +1}}{\alpha_2+\typevar - 1} + o_{\varepsilon\to 0}\left(\varepsilon^{-\alpha_2-\typevar +1}\right).
\end{equation*}
Therefore, assuming the second-order term of $\ecv$ is eventually monotone near zero, Lemma~\ref{lem:tauberian-integration} yields
\begin{equation*}
  \ecv(\varepsilon) 
  = c_1 \, \varepsilon^{-\alpha_1}
  + c_2 \, \varepsilon^{-\alpha_2}
  + o_{\varepsilon\to 0}\left(\varepsilon^{-\alpha_2}\right).
\end{equation*}
In practice, however, verifying the eventual monotonicity of the second-order term is rarely feasible.

In view of Theorem~\ref{thm:bias-variance-decomp}, 
we also need to understand how asymptotics behave under an infimum.
The next lemma shows that one can take the infimum without losing asymptotic equivalence.

\begin{lemma}\label{lem:abelian-infimum}
  Let $f_1, f_2 \colon \Rp \to \R_+$. 
  If $f_1(\varepsilon) \sim f_2(\varepsilon)$, as $\varepsilon\to 0$,
  then 
  \begin{equation*}
    \newinf_{\varepsilon > 0} \left\{  \sigma^2 f_1(\varepsilon) + \varepsilon^2 \right\}
    \sim \newinf_{\varepsilon > 0} \left\{  \sigma^2 f_2(\varepsilon) + \varepsilon^2 \right\}, 
    \quad \text{as } \sigma\to 0.
  \end{equation*}
\end{lemma}

\begin{proof}
  Let us fix $\delta>0$. 
  The asymptotic equivalence $f_1(\varepsilon) \sim f_2(\varepsilon)$, as $\varepsilon\to 0$, guarantees the existence of $\varepsilon_\delta>0$ such that 
  \begin{equation*}
    (1-\delta) f_1(\varepsilon)
    \leq f_2(\varepsilon) 
    \leq (1+\delta) f_1(\varepsilon),
    \quad\text{for all } \varepsilon \in (0,\varepsilon_\delta).
  \end{equation*}
  Consequently, we get 
  \begin{equation}\label{eq:proof-lemma-inf-7}
    (1-\delta)  \left(\sigma^2 f_1(\varepsilon) + \varepsilon^2\right)
    \leq  \sigma^2 f_2(\varepsilon) + \varepsilon^2 
    \leq (1+\delta)  \left(\sigma^2 f_1(\varepsilon) + \varepsilon^2\right),
  \end{equation}
  for all $ \varepsilon \in (0,\varepsilon_\delta)$  and  $\sigma>0$.
  We next observe that the map $\sigma\mapsto \inf_{\varepsilon > 0} \{  \sigma^2 f_2(\varepsilon) + \varepsilon^2 \}$ is non-decreasing.
  (Indeed, let us choose $\sigma_+> \sigma_->0$ arbitrarily, and let $\{\varepsilon_{n, +}\}_{n\in\Ns}$ be such that $  \sigma_+^2 f_2(\varepsilon_{n, +}) + \varepsilon_{n, +}^2 \to \inf_{\varepsilon > 0} \{  \sigma_+^2 f_2(\varepsilon) + \varepsilon^2 \}$, as $n\to\infty$.
  Then, 
  \begin{align*}
    \newinf_{\varepsilon > 0} \left\{  \sigma_+^2 f_2(\varepsilon) + \varepsilon^2 \right\} 
    &= \lim_{n\to \infty }  \sigma_+^2 f_2(\varepsilon_{n, +}) + \varepsilon_{n, +}^2 \\
    &\geq \lim_{n\to \infty }  \sigma_-^2 f_2(\varepsilon_{n, +}) + \varepsilon_{n, +}^2 
    \geq \newinf_{\varepsilon > 0} \left\{  \sigma_-^2 f_2(\varepsilon) + \varepsilon^2 \right\}, 
  \end{align*}
  from which the desired non-decreasing property follows.)
  Consequently, we can find $\sigma_1>0$ small enough to have 
  \begin{equation}\label{eq:proof-lemma-inf-8}
    \newinf_{\varepsilon > 0} \left\{  \sigma^2 f_2(\varepsilon) + \varepsilon^2 \right\} 
    \leq \frac{\varepsilon_\delta^2}{2},
    \quad \text{for all } \sigma\in(0,\sigma_1);
  \end{equation}
  take, for instance, $\sigma_1 = \varepsilon_\delta/(2\sqrt{f_2(\varepsilon_\delta/2)})$.
  This implies the existence of a sequence $\{\varepsilon_n\}_{n\in\Ns}$ attaining the infimum in the left-hand side of \eqref{eq:proof-lemma-inf-8} such that---using the assumption that $f_2$ is non-negative---we have $\varepsilon_n < \varepsilon_\delta$, for all $n\in\Ns$.
  Combining this observation with \eqref{eq:proof-lemma-inf-7} yields 
  \begin{align}
    (1-\delta)  \newinf_{\varepsilon > 0} \left\{  \sigma^2 f_1(\varepsilon) + \varepsilon^2 \right\} 
    &\leq  \lim_{n\to\infty} (1-\delta)  \left(\sigma^2 f_1(\varepsilon_n) + \varepsilon_n^2\right) \label{eq:proof-lemma-inf-1} \\
    &\leq \lim_{n\to\infty}  \sigma^2 f_2(\varepsilon_n) + \varepsilon_n^2
    = \newinf_{\varepsilon > 0} \left\{ \sigma^2 f_2(\varepsilon) + \varepsilon^2\right\},\label{eq:proof-lemma-inf-1-bis}
  \end{align}
  for all $\sigma\in(0,\sigma_1)$.
  Likewise, there exists $\sigma_2>0$ such that 
  \begin{equation}\label{eq:proof-lemma-inf-2}
    \newinf_{\varepsilon > 0} \left\{  \sigma^2 f_2(\varepsilon) + \varepsilon^2 \right\}
    \leq (1+\delta) \newinf_{\varepsilon > 0} \left\{  \sigma^2 f_1(\varepsilon) + \varepsilon^2 \right\},
    \quad\text{for all }\sigma\in(0,\sigma_2).
  \end{equation}
  Putting \eqref{eq:proof-lemma-inf-1}--\eqref{eq:proof-lemma-inf-1-bis} and \eqref{eq:proof-lemma-inf-2} together, we get 
  \begin{equation}\label{eq:proof-lemma-inf-9}
    (1-\delta)  \newinf_{\varepsilon > 0} \left\{  \sigma^2 f_1(\varepsilon) + \varepsilon^2 \right\} 
    \leq \newinf_{\varepsilon > 0} \left\{  \sigma^2 f_2(\varepsilon) + \varepsilon^2 \right\}  
    \leq (1+\delta) \newinf_{\varepsilon > 0} \left\{  \sigma^2 f_1(\varepsilon) + \varepsilon^2 \right\},
  \end{equation}
  for all $\sigma \in (0, \sigma_*)$, with $\sigma_*\coloneqq \min\{\sigma_1, \sigma_2\}$.
  The parameter $\delta $ can be taken arbitrarily small in \eqref{eq:proof-lemma-inf-9}, 
  thus establishing the desired result.
\end{proof}

\noindent 
The converse of Lemma~\ref{lem:abelian-infimum} holds provided the infimum is regularly varying at zero.
We further assume that the functions $f_1$ and $f_2$ are continuous, which entails no loss of generality in our setting because, by (RC), we can replace $\ecv$ by a continuous asymptotic equivalent.

\begin{lemma}\label{lem:tauberian-infimum}
  Let $b\in(-2,0)$, 
  let $f_1,f_2\colon \Rp\to\R_+$ be continuous non-increasing,
  and let $g\colon \Rp\to\Rp$ be regularly varying at zero with index $b$.
  If
  \begin{equation}\label{eq:proof-tauberian-inf-0}
    g(\sigma)
    \sim \newinf_{\varepsilon > 0} \left\{  \sigma^2 f_1(\varepsilon) + \varepsilon^2 \right\}
    \sim \newinf_{\varepsilon > 0} \left\{  \sigma^2 f_2(\varepsilon) + \varepsilon^2 \right\}, 
    \quad \text{as } \sigma\to 0,
  \end{equation}
  then $f_1(\varepsilon) \sim f_2(\varepsilon)$, as $\varepsilon\to 0$, and $f_1, f_2$ are regularly varying at zero with index $-(2b+4)/b$.
\end{lemma}

\begin{proof}
  Let us introduce, for $i\in\{1,2\}$, the function $g_i$ as the infimum appearing in \eqref{eq:proof-tauberian-inf-0}, that is, 
  \begin{equation}\label{eq:proof-tauberian-inf-1}
    g_i(s) \coloneqq \newinf_{\varepsilon > 0} \left\{  s f_i(\varepsilon) + \varepsilon^2 \right\},
    \quad\text{for all } s>0.
  \end{equation}
  By assumption, we have the asymptotic equivalence
  \begin{equation}\label{eq:proof-tauberian-inf-2}
      g_1(s) \sim g_2(s) \sim g(\sqrt{s}),
    \quad\text{as } s\to 0,
  \end{equation}
  implying that both $g_1$ and $g_2$ are regularly varying at zero with index $b/2$, and, in particular, that $g_1(s) \to 0$ and $g_2(s) \to 0$, as $s\to 0$.
  We next verify that the function $g_i$, for $i\in\{1,2\}$, is concave. 
  This is a consequence of the observation that, for all $s_1,s_2>0$ and $\lambda\in(0,1)$, we have 
  \begin{align*}
    g_i\left(\lambda\,s_1 + (1-\lambda)\, s_2\right)
    &= \newinf_{\varepsilon > 0} \left\{  \left(\lambda\,s_1 + (1-\lambda) \,s_2\right) f_i(\varepsilon) + \varepsilon^2 \right\} \\ 
    &= \newinf_{\varepsilon > 0} \left\{ \lambda \left(s_1 f_i(\varepsilon) + \varepsilon^2\right) + (1-\lambda) \left(s_2 f_i(\varepsilon) + \varepsilon^2\right) \right\}\\
    &\geq \lambda \newinf_{\varepsilon > 0} \left\{  s_1 f_i(\varepsilon) + \varepsilon^2 \right\} + (1-\lambda) \newinf_{\varepsilon > 0} \left\{  s_2 f_i(\varepsilon) + \varepsilon^2 \right\} \\
    &= \lambda \, g_i(s_1)+ (1-\lambda) \, g_i(s_2).
  \end{align*}
  The concavity of $g_i$---or, equivalently, the convexity of $-g_i$---implies the existence of its right derivative $g'_i$ \cite[Theorem~23.1]{rock70}, which, via  \cite[Corollary~24.2.1]{rock70}, satisfies
  \begin{equation*}
    g_i(s) = \int_{0}^{s} g'_i(\dumvar) \diff \dumvar,
    \quad\text{for all } s>0,
  \end{equation*}
  where we also used that $g_i(s)\to 0$, as $s \to 0$.
  Furthermore, the concavity of $g_i$ guarantees that $g'_i$ is non-increasing \cite[Theorem~24.1]{rock70}.
  Together with the regularly-varying nature of $g_i$, this allows us to apply the  monotone density theorem \cite[Theorem~1.7.2b]{binghamRegularVariation1987} to obtain
  \begin{equation}\label{eq:proof-tauberian-inf-3}
    g_i (s) 
    \sim - \frac{2s \, g'_i(s)}{b},
    \quad\text{as } s\to 0.
  \end{equation}

  Now, let us introduce the set of minimizers in the definition \eqref{eq:proof-tauberian-inf-1} of $g_i$ according to 
  \begin{equation*}
    E_i(s)
    \coloneqq \left\{ \varepsilon \mid  g_i(s) = s f_i(\varepsilon) + \varepsilon^2 \right\},
    \quad\text{for all } s>0.
  \end{equation*}
  This set is closed and non-empty by continuity of $f_i$.
  We additionally verify that $E_i(s)$ is bounded for all $s>0$---all members of this set must be less that $\sqrt{g_i(s)}$---and define $\varepsilon_i(s)$ as its maximal element.
  Then, by Danskin's theorem (see, e.g., \cite[Theorem~1]{BERNHARD19951163} and the references therein), the right derivative of $g_i$ satisfies
  \begin{equation}\label{eq:proof-tauberian-inf-4}
    g'_i(s) 
    = \min_{\varepsilon \in E_i(s)}\left\{ f_i(\varepsilon) \right\} 
    = f_i(\varepsilon_i(s)), 
    \quad\text{for all } s>0,
  \end{equation} 
 where the second step follows from $f_i$ being non-increasing.
  Therefore, combining \eqref{eq:proof-tauberian-inf-4} with  \eqref{eq:proof-tauberian-inf-2} and \eqref{eq:proof-tauberian-inf-3}, we get 
  \begin{equation}\label{eq:proof-tauberian-inf-45}
    f_i(\varepsilon_i(s))
    \sim \frac{-b}{2s} g_i(s)
    \sim  \frac{-b}{2s} g(\sqrt{s}),
    \quad\text{as } s\to 0.
  \end{equation}
  Furthermore, using that $\varepsilon_i(s) \in E_i(s)$ together with \eqref{eq:proof-tauberian-inf-2} and \eqref{eq:proof-tauberian-inf-45}, we obtain
  \begin{equation}\label{eq:proof-tauberian-inf-5}
    \varepsilon_i^2 (s)
    = g_i(s) -s f_i(\varepsilon_i(s))
    \sim \left(1+\frac{b}{2}\right)g(\sqrt{s}) ,
    \quad\text{as } s\to 0.
  \end{equation}
  In particular, the right-most term of \eqref{eq:proof-tauberian-inf-5} being independent of the index $i$, we must have $\varepsilon_1(s) \sim \varepsilon_2(s)$, as $s\to 0$.
  Moreover, $g$ being regularly varying at zero with index $b$ implies that $\varepsilon_i$ is as well regularly varying at zero, but with index $b/4$.

Let us now prove that the map $s\mapsto \varepsilon_i(s)$ is non-decreasing, for all $i\in\{1,2\}$.
To this end, let us introduce $s_2>s_1>0$, and observe that, by definition of $\varepsilon_i$, we must have 
\begin{equation}\label{eq:esp-monotone-1}
  s_1 f_i(\varepsilon_i(s_1)) + \varepsilon_i^2(s_1) 
  \leq s_1 f_i(\varepsilon_i(s_2)) + \varepsilon_i^2(s_2),
\end{equation}
and 
\begin{equation}\label{eq:esp-monotone-2}
  s_2 f_i(\varepsilon_i(s_2)) + \varepsilon_i^2(s_2) 
  \leq s_2 f_i(\varepsilon_i(s_1)) + \varepsilon_i^2(s_1).
\end{equation}
Combining \eqref{eq:esp-monotone-1}
 and \eqref{eq:esp-monotone-2} yields 
 \begin{equation*}
   s_1 f_i(\varepsilon_i(s_1)) + \varepsilon_i^2(s_1)
   + s_2 f_i(\varepsilon_i(s_2)) + \varepsilon_i^2(s_2) 
  \leq s_1 f_i(\varepsilon_i(s_2)) + \varepsilon_i^2(s_2) 
  + s_2 f_i(\varepsilon_i(s_1)) + \varepsilon_i^2(s_1) ,
 \end{equation*}
or, equivalently, 
\begin{equation}
  (s_2-s_1) \left(f_i(\varepsilon_i(s_1)) - f_i(\varepsilon_i(s_2))\right) \geq 0.
\end{equation}
Since we have taken $s_2>s_1$, we either have $f_i(\varepsilon_i(s_1)) = f_i(\varepsilon_i(s_2))$, in which case \eqref{eq:esp-monotone-1} and \eqref{eq:esp-monotone-2} imply $\varepsilon_i(s_1) = \varepsilon_i(s_2)$, or we have $f_i(\varepsilon_i(s_1)) > f_i(\varepsilon_i(s_2))$, in which case $f_i$ being non-increasing implies $\varepsilon_i(s_1) <  \varepsilon_i(s_2)$.
Both cases give $\varepsilon_i(s_1) \leq \varepsilon_i(s_2)$, which allows to conclude that $s\mapsto \varepsilon_i(s)$ is indeed non-decreasing.

  Following \cite[Chapter~1.5.7]{binghamRegularVariation1987},
  let us now introduce the generalized inverse $s_i$ of $\varepsilon_i$ according to
  \begin{equation}\label{eq:definition-si}
    s_i(\varepsilon) 
    \coloneqq \newinf_{\varepsilon>0} \left\{s>0 \mid \varepsilon_i(s) > \varepsilon \right\}, 
    \quad\text{for all } \varepsilon>0.
  \end{equation}
  Upon application of \cite[Theorem~1.5.12]{binghamRegularVariation1987},
  we are guaranteed that $s_i$ is regularly varying at zero with index $4/b$.
  Let $\delta\in(0,1)$.
  Now, using that the map $s\mapsto \varepsilon_i(s)$ is non-decreasing, we get from \eqref{eq:definition-si}
  \begin{equation}\label{eq:proof-tauberian-inf-6}
    \varepsilon_i\big(s_i(\varepsilon) (1-\delta)\big)
    \leq \varepsilon
    < \varepsilon_i\big(s_i(\varepsilon) (1+\delta)\big), 
    \quad\text{for all }\varepsilon>0.
  \end{equation}
  Applying $f_i$ to \eqref{eq:proof-tauberian-inf-6} and using that it is non-increasing, we get 
  \begin{equation}\label{eq:proof-tauberian-inf-7}
    \cfrac{f_i\Big(\varepsilon_i\big(s_i(\varepsilon) (1+\delta)\big)\Big)}{ \frac{-b}{2s_i(\varepsilon)} g\left(\sqrt{s_i(\varepsilon)}\right)}
    \leq \cfrac{f_i(\varepsilon)}{ \frac{-b}{2s_i(\varepsilon)} g\left(\sqrt{s_i(\varepsilon)}\right)}
    \leq \cfrac{f_i\Big(\varepsilon_i\big(s_i(\varepsilon) (1-\delta)\big)\Big)}{ \frac{-b}{2s_i(\varepsilon)} g\left(\sqrt{s_i(\varepsilon)}\right)}, 
    \quad\text{for all }\varepsilon>0.
  \end{equation}
  Then, using \eqref{eq:proof-tauberian-inf-45} with the observation that the map  $s\mapsto-bg(\sqrt{s})/(2s)$ is regularly varying at zero with index $1+b/2$, we get  
  \begin{align*}
    \lim_{\varepsilon \to 0} \cfrac{f_i\Big(\varepsilon_i\big(s_i(\varepsilon) (1+\delta)\big)\Big)}{ \frac{-b}{2s_i(\varepsilon)} g\left(\sqrt{s_i(\varepsilon)}\right)}
    &= \lim_{\varepsilon \to 0}  \cfrac{f_i\Big(\varepsilon_i\big(s_i(\varepsilon) (1+\delta)\big)\Big)}{\frac{-b}{2s_i(\varepsilon) (1+\delta)} g\left(\sqrt{(1+\delta) s_i(\varepsilon)}\right)} 
     \cfrac{ \frac{-b}{2s_i(\varepsilon) (1+\delta)} g\left(\sqrt{(1+\delta)s_i(\varepsilon)}\right)}{ \frac{-b}{2s_i(\varepsilon)} g\left(\sqrt{s_i(\varepsilon)}\right)} \\
    &= (1+\delta)^{-(1+b/2)}.
  \end{align*}
  Likewise, we have 
   \begin{equation*}
    \lim_{\varepsilon \to 0} \cfrac{f_i\Big(\varepsilon_i\big(s_i(\varepsilon) (1-\delta)\big)\Big)}{\frac{-b}{2s_i(\varepsilon)} g\left(\sqrt{s_i(\varepsilon)}\right)}
    = (1-\delta)^{-(1+b/2)},
  \end{equation*}
  so that taking the limit $\varepsilon\to 0$ in \eqref{eq:proof-tauberian-inf-7} yields
  \begin{align}
    (1+\delta)^{-(1+b/2)}
    &\leq \liminf_{\varepsilon\to 0} \cfrac{f_i(\varepsilon)}{\frac{-b}{2s_i(\varepsilon)} g\left(\sqrt{s_i(\varepsilon)}\right)}\label{eq:proof-tauberian-inf-8-1}\\
    &\leq \limsup_{\varepsilon\to 0} \cfrac{f_i(\varepsilon)}{\frac{-b}{2s_i(\varepsilon)} g\left(\sqrt{s_i(\varepsilon)}\right)}
    \leq (1-\delta)^{-(1+b/2)}.\label{eq:proof-tauberian-inf-8-2}
  \end{align}
  Choosing $\delta$ arbitrary small in \eqref{eq:proof-tauberian-inf-8-1}--\eqref{eq:proof-tauberian-inf-8-2} gives 
  \begin{equation}\label{eq:proof-tauberian-inf-9}
    \lim_{\varepsilon\to 0} \cfrac{f_i(\varepsilon)}{\frac{-b}{2s_i(\varepsilon)} g\left(\sqrt{s_i(\varepsilon)}\right)} = 1.
  \end{equation}
  Now, by uniqueness of the asymptotic inverse within asymptotic equivalence (see \cite[Theorem~1.5.12]{mityagin1961}), the previously established relation $\varepsilon_1(s) \sim \varepsilon_2(s)$, as $s\to 0$, implies  $s_1(\varepsilon)\sim s_2(\varepsilon)$, as $\varepsilon\to 0$.
  Additionally, asymptotic equivalence being stable under composition of regularly varying functions \cite[Theorem~1.8.7]{binghamRegularVariation1987}, we get $g(\sqrt{s_1(\varepsilon)})\sim g(\sqrt{s_2(\varepsilon)})$, as $\varepsilon\to 0$.
  Consequently, \eqref{eq:proof-tauberian-inf-9} yields 
  \begin{equation}\label{eq:proof-tauberian-inflast}
    f_1(\varepsilon)
    \sim  \frac{-b}{2s_1(\varepsilon)} g\left(\sqrt{s_1(\varepsilon)}\right) 
    \sim -\frac{-b}{2s_2(\varepsilon)} g\left(\sqrt{s_2(\varepsilon)}\right)
    \sim f_2(\varepsilon), 
    \quad\text{as }\varepsilon \to 0.
  \end{equation}
  We finish the proof by observing that $s_i$ and $g$ being regularly varying, with respective indices $4/b$ and $b$, implies, via \cite[Theorem~1.8.7]{binghamRegularVariation1987} and \eqref{eq:proof-tauberian-inflast}, that $f_1$ an $f_2$ are regularly varying at zero with index $-(2b+4)/b$.
\end{proof}

\subsection{Proof of Lemma~\ref{lem:integral-k-to-k+1}}\label{sec:proof-lemma-integral-k-to-k+1}

Fix $\typevar_1, \typevar_2 \geq 1$ and $\varepsilon >0$.
Note that the assumption $\mu_n\to 0$ as $n\to\infty$ guarantees $\ecv(\dumvar) < \infty$ for all $\dumvar>0$.
We start from the observation that
\begin{align}
  I_{\typevar_1}(\varepsilon)
  = \int_{\varepsilon}^{\infty} \frac{\ecv(\dumvar)}{\dumvar^{\typevar_1}} \diff\dumvar
  &= \int_{\varepsilon}^{\infty} \dumvar^{\typevar_2-\typevar_1} \left(-\int_{\dumvar}^\infty \frac{\ecv(\dumvartwo)}{\dumvartwo^{\typevar_2}} \diff \dumvartwo \right)' \diff\dumvar \label{eq:proof-int-by-part-1}\\
  &= - \int_{\varepsilon}^{\infty} \intk{\typevar_2}'\left(\dumvar\right)\, \dumvar^{\typevar_2-\typevar_1} \diff\dumvar, \label{eq:proof-int-by-part-2}
\end{align}
and then integrate by parts the right-hand side of \eqref{eq:proof-int-by-part-1}--\eqref{eq:proof-int-by-part-2} to get the desired result according to 
\begin{align*}
    \intk{\typevar_1}(\varepsilon) 
     & = \intk{\typevar_2}(\varepsilon)\, \varepsilon^{\typevar_2-\typevar_1} 
  + \int_{\varepsilon}^{\infty} \intk{\typevar_2}(\dumvar)\, (\dumvar^{\typevar_2-\typevar_1})' \diff \dumvar\\
   & = \intk{\typevar_2}(\varepsilon)\, \varepsilon^{\typevar_2-\typevar_1} 
  + (\typevar_2-\typevar_1) \int_{\varepsilon}^{\infty} \intk{\typevar_2}(\dumvar)\, \dumvar^{\typevar_2-\typevar_1-1} \diff \dumvar.
\end{align*}

\subsection{Statement and Proof of Lemma~\ref{lem:regularity-evf-o-term}}\label{sec:proof-lem-regularity-evf-o-term}

\begin{lemma}\label{lem:regularity-evf-o-term}
    Let $\mu=\semax$ be a sequence such that $\mu_n\to 0$ as $n\to\infty$,
  with corresponding semi-axis-counting function $\ecv$ satisfying (RC).
  Let $x^*>0$ be such that $\ecv(x^*)\geq 1$, let $f$ be as in (RC), and
  let 
  \begin{equation*}
    g(x) \coloneqq x\left(1+\frac{1}{\ecv(x)}\right),
    \quad\text{for all } x\in (0, x^*).
  \end{equation*}
  Then, we have $f(g (x)) \sim {f(x)}$, as $x\to 0$.
\end{lemma}

\begin{proof}
  Note that we can assume---without loss of generality---that $f(x) \to \infty$ as $x\to 0$, since, otherwise, the result directly follows from $\lim_{x\to 0} f(x) = \lim_{x\to 0} f(g(x))$.

  Let us introduce the auxiliary function $s$, defined according to $s(x)\coloneqq 1/f(x)$, for all $x>0$.
  Then, we apply the mean value theorem \cite[Theorem~5.10]{rudin1953principles} to $s$ to obtain 
  \begin{equation}\label{eq:application-mvt}
    \left\lvert s( g (x) )- s(x) \right\rvert
    = \left\lvert g(x)- x\right\rvert \cdot \left\lvert s'(y_x)\right\rvert
    = \frac{x}{\ecv(x)} \left\lvert s'(y_x)\right\rvert,
    \quad\text{for all } x\in (0, x^*), 
  \end{equation}
  and for some $y_x \in (x, g(x))$.
  Dividing both sides of \eqref{eq:application-mvt} by $s(x)$ and using that $y_x > x $, we obtain 
  \begin{equation*}
    \left\lvert \frac{f(x)}{f(g (x))}-1 \right\rvert
    =\left\lvert \frac{s(g (x))}{s(x)}-1 \right\rvert
    < \frac{\left\lvert y_x \, s'(y_x) \right\rvert}{s(x)\ecv(x)},
    \quad\text{for all } x\in (0, x^*),
  \end{equation*}
  with $y_x \in (x, g(x))$.
  Since $f$ satisfies (RC), we get $s(x)\ecv(x) \to 1$, as $x\to 0$.
  It is therefore sufficient to show that $ys'(y)\to 0$, as $y\to 0$, which we proceed to do next.

  To this end, let us first introduce $t \coloneqq \ln(1/y)$, for all $y>0$.
  Then, applying consecutively twice the chain rule gives
  \begin{equation}\label{eq:ratio-rho-f-proof-lemma}
    ys'(y) 
    = - \frac{y f'(y)}{f^2(y)}
    = - \frac{e^{-t} f'(e^{-t})}{f^2(e^{-t})}
    = \frac{\rho(t)}{f(e^{-t})},
    \quad \text{for all } y>0,
  \end{equation}
  where $\rho$ is the elasticity of $f$ as in (RC).
  If $\rho (t)$ has a finite limit as $t\to\infty$, then  the right-most term in \eqref{eq:ratio-rho-f-proof-lemma} goes to zero, yielding the desired result---recall that we have assumed that $f(x)\to \infty$, as $x\to 0$.
  We thus assume in what follows that $\rho (t)\to\infty$, as $t\to\infty$.
  Setting $z(t)\coloneqq s(e^{-t})$, for all $t>0$, we get from \eqref{eq:ratio-rho-f-proof-lemma} 
  \begin{equation*}
    \rho(t) 
    = - \frac{e^{-t} f'(e^{-t})}{f(e^{-t})}
    = \frac{e^{-t} s'(e^{-t})}{s(e^{-t})} 
    = -\frac{z'(t)}{z(t)},
    \quad\text{for all } t>0,
  \end{equation*}
  which, upon integration and fixing $t_0>0$, implies 
  \begin{equation*}
    z(t) = z(t_0) \, \exp\left(-\int_{t_0}^{t} \rho(\dumvar) \diff\dumvar\right), 
    \quad\text{for all } t\geq t_0.
  \end{equation*}
  Consequently, we obtain, using \eqref{eq:ratio-rho-f-proof-lemma}, the relation 
  \begin{equation}\label{eq:to-prove-last-lemma}
    ys'(y) 
    = \rho(t) z(t)  
    = z(t_0) \,\rho(t) \, \exp\left(-\int_{t_0}^{t} \rho(\dumvar) \diff\dumvar\right), 
    \quad\text{for all } t\geq t_0.
  \end{equation}
  Together \eqref{eq:definition-elasticity} and $f$ being non-decreasing ensure that $\rho$ is non-negative.
  Moreover, by (RC), $\rho$ is non-decreasing on $(t^*,\infty)$, for some $t^*>0$. 
  We can therefore lower-bound the integral of $\rho$ according to
  \begin{equation}\label{eq:LB-integral-rho-proof-lemma}
    \int_{t_0}^{t}\rho(u)\diff u 
    \geq \int_{t/2}^{t}\rho(u)\diff u 
    \geq \frac{t\rho(t/2)}{2},
    \quad \text{for all } t\geq 2\max\{t_0, t^*\}.
  \end{equation}
  Again using (RC), we have $t\rho(t/2)/2 - \ln\left(\rho(t)\right) \to \infty$, as $t \to \infty$, which, when combined with  \eqref{eq:LB-integral-rho-proof-lemma}, yields 
  \begin{equation}\label{eq:last-equation-paper}
    \rho(t) \, \exp\left(-\int_{t_0}^{t} \rho(\dumvar) \diff\dumvar\right) \to 0,
    \quad\text{as } t\to\infty.
  \end{equation}
  Injecting the limit \eqref{eq:last-equation-paper} in \eqref{eq:to-prove-last-lemma} gives $ys'(y)\to0$, as $y\to 0$, as desired, thereby finishing the proof of Lemma~\ref{lem:regularity-evf-o-term}.
\end{proof}

\end{document}